\documentclass[11pt]{amsart}
\usepackage{amssymb,amsmath,amsthm,amsrefs}
\usepackage{appendix}
\usepackage[mathscr]{eucal}
\usepackage{color}
\usepackage{enumerate}
\usepackage[utf8]{inputenc}
\usepackage[leftcaption]{sidecap}
\usepackage{tikz, pgfplots}
 \usepackage{tikz-cd}
 \usepackage{cancel}
 \usepackage{stmaryrd}
 \usepackage{multicol}
\usetikzlibrary{matrix,arrows,decorations.pathmorphing}
\usetikzlibrary{positioning}
\usetikzlibrary{matrix,arrows,decorations.pathmorphing, shapes.geometric}
\tikzset{>=latex}

\usepackage[T1, T2A]{fontenc}

\usepackage{psfrag}
\usepackage{hyperref}

\usepackage[safe]{tipa}
\DeclareFontFamily{U}{mathx}{\hyphenchar\font45}
\DeclareFontShape{U}{mathx}{m}{n}{
      <5> <6> <7> <8> <9> <10>
      <10.95> <12> <14.4> <17.28> <20.74> <24.88>
      mathx10
      }{}
\DeclareSymbolFont{mathx}{U}{mathx}{m}{n}
\DeclareFontSubstitution{U}{mathx}{m}{n}
\DeclareMathAccent{\widecheck}{0}{mathx}{"71}

\makeatletter
\newcounter{savesection}
\newcounter{apdxsection}
\renewcommand\appendix{\par
  \setcounter{savesection}{\value{section}}%
  \setcounter{section}{\value{apdxsection}}%
  \setcounter{subsection}{0}%
  \gdef\thesection{\@Alph\c@section}}
\newcommand\unappendix{\par
  \setcounter{apdxsection}{\value{section}}%
  \setcounter{section}{\value{savesection}}%
  \setcounter{subsection}{0}%
  \gdef\thesection{\@arabic\c@section}}
\makeatother

\usepackage{epstopdf}
\usepackage{subcaption}
\usepackage{graphicx}
\usepackage{verbatim}
\usepackage{amsmath,amsthm}
\usepackage{graphics}
\usepackage{color}
\usepackage{epsfig}
\usepackage{fullpage}
\usepackage{amssymb,amsmath}
\usepackage[mathscr]{eucal}
\usepackage{ upgreek }
\usepackage{tensor}

\newtheorem{theorem}[equation]{Theorem}
\newtheorem*{theorem*}{Theorem}
\newtheorem{lemma}[equation]{Lemma}

\newtheorem{proposition}[equation]{Proposition}
\newtheorem{cor}[equation]{Corollary}
\newtheorem{corollary}[equation]{Corollary}

\newtheorem{definition}[equation]{Definition}

\newtheorem{remark}[equation]{Remark}
\newtheorem{notation}[equation]{Notation}
\newtheorem{convention}[equation]{Convention}
\newtheorem{assumption}[equation]{Assumption}

\numberwithin{equation}{section}

\newcommand{\R}{\mathbb{R}}
\newcommand{\N}{\mathbb{N}}

\newcommand{\K}{\mathbb{K}}

\newcommand{\Z}{\mathbb{Z}}

\newcommand{\Sph}{\mathbb{S}}
\newcommand{\Spheqduo}{\mathbb{S}_{eq}^2}

\newcommand{\sgn}{\operatorname{sgn}}

\newcommand{\Grp}{\mathscr{G}}

\newcommand{\Hgrp}{\mathscr{H}}

\newcommand{\Fcal}{\mathcal{F}}
\newcommand{\Ccal}{\mathcal{C}}
\newcommand{\Acal}{\mathcal{A}}
\newcommand{\Hcal}{\mathcal{H}}
\newcommand{\Lcal}{\mathcal{L}}

\newcommand{\Rcal}{\mathcal{R}}

\newcommand{\uC}{{\mathscr{C}}}

\DeclareFontFamily{U}{mathx}{}
\DeclareFontShape{U}{mathx}{m}{n}{<-> mathx10}{}
\DeclareSymbolFont{mathx}{U}{mathx}{m}{n}
\DeclareMathAccent{\widehat}{0}{mathx}{"70}
\DeclareMathAccent{\widecheck}{0}{mathx}{"71}

\newcommand{\zz}{\ensuremath{\mathrm{z}}}
\newcommand{\rr}{\ensuremath{\mathrm{r}}}

\newcommand{\cu}{\underline{c}} 
\newcommand{\hu}{\underline{h}}

\newcommand{\abs}[1]{\left\lvert#1\right\rvert}
\newcommand{\norm}[1]{\left\|#1\right\|}
\newcommand{\cutoff}[3]{\mathbf{\Psi}\left[ #1,#2;#3 \right]}

\newcommand{\osc}{{{osc}}}

\newcommand{\CCC}{\mathsf{C}}
\newcommand{\UUU}{\mathsf{U}}
\newcommand{\RRR}{\mathsf{R}}

\newcommand{\skernel}{\mathscr{K}}
\newcommand{\skernelv}{\widehat{\mathscr{K}}}

\newcommand{\psicut}{{\psi_{cut}}}
\newcommand{\Psibold}{{\boldsymbol{\Psi}}}

\newcommand{\arcosh}{\mathrm{arcosh}}

\newcommand{\avg}{\mathrm{avg}}
\newcommand{\supp}{\mathrm{supp}}

\newcommand{\dist}{\mathbf{d}}

\newcommand{\Id}{\mathrm{Id}}

\newcommand{\Graph}{\mathrm{Graph}}
\newcommand{\inj}{\mathrm{inj}}
\newcommand{\dd}{\mathrm{d}}
\newcommand{\ds}{\mathrm{d}s}
\newcommand{\dt}{\mathrm{d}t}

\newcommand{\dr}{\mathrm{d}r}
\newcommand{\sym}{\mathrm{sym}}

\newcommand{\Cylinder}{\mathrm{Cyl}}

\newcommand{\Ecal}{\mathcal{E}}
\newcommand{\Ecalu}{\underline{\mathcal{E}}}
\newcommand{\Vcal}{\mathcal{V}}

\newcommand{\Mcal}{\mathcal{M}}

\newcommand{\Pcal}{\mathcal{P}}
\newcommand{\Jcal}{\mathcal{J}}
\newcommand{\BPcal}{B_{\Pcal}}
\newcommand{\BPcals}{B_{\Pcal_s}}

\newcommand{\deltamin}{\delta_{\mathrm{min}}}
\newcommand{\deltau}{\underline{\delta}}

\newcommand{\taubold}{\boldsymbol{\tau}}
\newcommand{\kappabold}{\boldsymbol{\kappa}}
\newcommand{\kappaboldu}{\underline{\boldsymbol{\kappa}}}

\newcommand{\kappau}{\underline{\kappa}}

\newcommand{\zetabold}{\boldsymbol{\zeta}}

\newcommand{\zetaboldu}{\underline{\zetabold}}

\newcommand{\mubold}{\boldsymbol{\mu}}

\newcommand{\wbold}{\boldsymbol{w}}
\newcommand{\psibold}{\boldsymbol{\psi}}
\newcommand{\fbold}{\boldsymbol{f}}
\newcommand{\gbold}{\boldsymbol{g}}
\newcommand{\ubold}{\boldsymbol{u}}
\newcommand{\nubold}{\boldsymbol{\nu}}

\newcommand{\varphihat}{\hat{\varphi}}
\newcommand{\Lu}{\underline{L}}
\newcommand{\tauu}{\underline{\tau}}
\newcommand{\su}{\underline{s}}
\newcommand{\ru}{\underline{r}}
\newcommand{\Ku}{\underline{K}}
\newcommand{\Kcech}{\widecheck{K}}

\newcommand{\bbracket}[1]{[\![#1]\!]}

\DeclareFontFamily{U}{tipa}{}
\DeclareFontShape{U}{tipa}{m}{sl}{
	<-8.5> tipasl8
	<8.5-9.5> tipasl9
	<9.5-11> tipasl10
	<11-> tipasl12
}{}
\DeclareSymbolFont{tipa}{U}{tipa}{m}{sl}
\DeclareMathSymbol{\SmallCapitalR}{\mathord}{tipa}{246}
\DeclareMathSymbol{\textPhi}{\mathord}{tipa}{70}
\newcommand{\rbalanced}{r_{mu}}

\newcommand{\Tcal}{\mathcal{T}}

\newcommand{\sgr}[2]{\mathrm{Isom}_{#1}^{#2} }

\newcommand{\shr}{Shr}

\newcommand{\JM}{\mathscr{J}_M}

\title[Self-shrinkers by stacking]{Self-shrinkers with any number of ends\\ in $\mathbb{R}^{3}$ by stacking $\mathbb{R}^{2}$}  

\author[G.~Shao]{Guanhua~Shao}
\address{Department of Mathematics, Rutgers University, Pistacaway, NJ 08854} 
\email{gs977@math.rutgers.edu}
	   
\author[J.~Zou]{Jiahua~Zou} 
\address{Department of Mathematics, Rutgers University, Pistacaway, NJ 08854} 
\email{jiahua.zou@rutgers.edu}

\begin{document}

\date{\today}


\begin{abstract}
	 For each half-integer $J$ and large enough integer $m$ we construct by PDE gluing methods a self-shrinker $\breve{M}[J,m]$ with $2J+1$ ends and genus $2J(m-1)$. $\breve{M}[J,m]$ resembles the stacking of $2J+1$ levels of the plane $\R^2$ in $\R^3$ that have been connected by $2Jm$ catenoidal bridges with $m$ bridges connecting each pair of adjacent levels. It observes the symmetry of an $m$-gonal prism (when $J$ is a half integer) or an $m$-gonal antiprism (when $J$ is an integer). The construction is based on the Linearised Doubling (LD) methodology which was first introduced by Kapouleas in the construction of minimal surface doublings of $\Spheqduo$ in $\Sph^3$.
\end{abstract}

\maketitle

\section{Introduction}
\label{S:intro} 

\subsection*{The historic framework and background} 
A hypersurface $\Sigma^n$ immersed in $\R^{n+1}$ is called an \emph{$f$-minimal hypersurface} if its mean curvature $H_{\Sigma}$ satisfies 
\begin{equation*}
	H_{\Sigma}+\nabla_{\Sigma} f\cdot \nu_{\Sigma}=0
\end{equation*}
where $\nu_{\Sigma}$ is a choice of unit normal, $f$ is a function defined on $\Sigma$. Here for the mean curvature we use the convention $H_{\Sigma}:=\Delta_{\Sigma} X_{\Sigma}\cdot \nu_{\Sigma}$, where $X_{\Sigma}$ is the embedding of $\Sigma$ and $\Delta_{\Sigma}$ is the Laplacian operator on $\Sigma$ induced from the Euclidean metric in $\R^{n+1}$. In particular, when 
\begin{equation*}
	f=\frac{\abs{x}^2}{4},
\end{equation*}
i.e. when $H_{\Sigma}$ satisfies 
\begin{equation*}
	H_{\Sigma}+\frac{1}{2} X_{\Sigma}\cdot \nu_{\Sigma}=0,
\end{equation*}
$\Sigma^n$ is called a \emph{self-shrinker}. It is well-known that for a self-shrinker $\Sigma$, the family of surfaces
\begin{equation*}
	(-\infty,0)\ni t\mapsto \sqrt{\abs{t}}\Sigma 
\end{equation*}
is a mean curvature flow.

Many interesting examples of self-shrinkers have been constructed by reduction to the analysis of ODE, e.g., \cite{angenent,KM,Riedler}; in $\R^3$ by minimax method, e.g., \cite{ketover:platonicsolids,ketover2024,BNS}, by studying the mean curvature flow \cite{IW,HMW}; or by gluing method, e.g., \cite{nguyenIII,kapouleas:kleene:moller,LDg} and in higher dimension, e.g., \cite{s3doubling}. However, only three classes of non-compact examples without a continuous group of symmetry in $\R^3$ are known: the desingularising constructions of the plane and the sphere in \cite{nguyenI,nguyenII,nguyenIII,kapouleas:kleene:moller,BNS}, the doubling constructions in \cite{ketover2024,IW} and the tripling constructions in \cite{HMW} of the plane. In particular, they all have one, two or three ends.

\subsection*{Brief discussion of the results and the staking construction}   
As a generalisation of the constructions in \cite{IW,ketover2024,HMW}, in this article we construct self-shrinkers in $\R^3$ by stacking the plane with any number of ends bigger or equal to $2$. More precisely, we prove the following theorem.
 \begin{theorem*}
	For any $J\in \frac{1}{2}\N$, if $m\in\N$ is large enough, then there is a $\Grp[m,J]$-invariant embedded self-shrinker $\breve{M}[m,J]$ in $\R^3$ (see \ref{def:grp}) of genus $2J(m-1)$ and $2J+1$ ends. The hypersurfaces $\breve{M}$ converge in the sense of varifolds as $m \to\infty$ to $(2J+1)\R^2$, and $\breve{M}$ has asymptotic cone $\cup_j\Ccal_j$, where $\Ccal_j$ is the cone described in \ref{thm}.
\end{theorem*}
Roughly speaking the surface $\Grp[m,J]$ has the shape of $2J+1$ levels of planes, and between each two adjacent levels, there are $m$ necks connecting them. The necks are distributed equally but alternately on the same circle (see \ref{def:L}).

One should expect that when $J=1/2$ and $m$ is large enough as required in all the constructions,  $\breve{M}[m,1/2]$ is the same self-shrinker as the one constructed in \cite{IW} by Ilmanen and White and the one constructed in \cite{ketover2024} by Ketover; and when $J=1$,  $\breve{M}[m,1]$ is the same self-shrinker as the one constructed in \cite{HMW} by Hoffman, Mart\'{\i}n and White.

Of particular interest for this construction are the \emph{stacking constructions} by PDE gluing methods as a generalisation of the doubling constructions. The first doubling constructions by PDE gluing methods produced minimal surface doublings of the Clifford torus \cite{kapouleas:clifford} and free boundary minimal surface doublings of the disk \cite{zolotareva}. Subsequently Kapouleas introduced in \cite{kapouleas:equator} a refinement of the general methodology which he called \emph{Linearised Doubling (LD)}. Using the LD methodology he constructed the first high genus minimal surface doublings of the equatorial sphere in round three-sphere. Since then the LD methodology has been successfully providing many new examples: minimal surfaces in round three-sphere \cite{kapouleas:ii:mcgrath,LDg,zouthesis}, minimal hypersurfaces in round four-sphere \cite{s3doubling}, self-shrinkers in $\R^3$ \cite{LDg} and $\R^4$ \cite{s3doubling}. In \cite{LDg} Kapouleas and McGrath also gave a very systematic study of doubling constructions in general.
 
The stacking constructions by PDE gluing methods were also introduced to study the same problems: they produced first minimal surface stackings of the Clifford torus \cite{wiygul:stacking} by Wiygul and then free boundary minimal surface stackings of the disk \cite{kapouleas:wiygul,CSW}.

In this paper, as a first example, we apply the LD methodology to the stacking constructions. We first find $2J+1$ LD solutions which solve the linearised self-shrinker equation (see \eqref{eq:jacobi}) and have logarithmic singularities, then connect their graphs by catenoids in Gaussian metric (see \ref{not:g}), and finally find the self-shrinkers by a perturbation. 

We want to point out that, although not pursued here, our method will also work to construct $f$-minimal surfaces in general of the same topology as long as the function $f$ satisfies the conditions needed in the construction.

We also mention that by the interesting work \cite{LZ} of Lee and Zhao, the self-shrinkers constructed in this paper will arise as blow-ups of the mean curvature flow of initially smooth, compact surfaces as they all have only conical ends.

\subsection*{Outline of strategy and main ideas}
As discussed above, we first construct the LD solutions (see \ref{def:LD}). However, instead of just one family of the LD solutions in in this article as in \cite{kapouleas:equator,kapouleas:ii:mcgrath,LDg,zouthesis,s3doubling}, in this article we have to make use of $J+\frac{1}{2}$ (when $J\notin \N$) or $J+1$ (when $J\in \N$) different families of the LD solutions $\varphi_j$. Fortunately, from the consideration of the balancing condition, their averages on the concentric circles, i.e. the \emph{rotationally invariant linearised doubling (RLD) solutions} (see \ref{lem:Phiavg}) differ from each other very little up to the coefficients. And as the cases studied before, we can control the LD solutions by the RLD solutions. Moreover, as the necks will be distributed along one circle between the levels, the RLD solutions in this article will only have one circle of singularity, which corresponds to the results in \cite{ketover2024,IW,HMW}. Therefore, this part is similar to the counterpart in \cite{kapouleas:equator}, where the catenoidal bridges are equidistributed along two parallel circle of $\Spheqduo$.

The LD solutions are then converted to \emph{matched linearised doubling (MLD) solutions} $\varphi_j+\underline{v}_j$, and here arises the main difficulty of this article. Unlike the doubling case, the balancing and unbalancing questions of the stacking constructions are more difficult to study. In general, we have $4$ parameters for each catenoidal bridges undetermined by the symmetry: the waist $\tau$, the height $h$, the location $r$ and the tilt $\kappa$ (see \ref{def:catebridge}). For each level $j$, the MLD solution $\varphi_j+\underline{v}_j$ are determined by these parameters of catenoids connecting this level from above and below. The matching conditions (see \ref{def:mismatch}) then provide a system of equations for these parameters which are called \emph{vertical balancing} (for $\tau$ and $h$) and \emph{horizontal balancing} (for $r$ and $\kappa$) as in \cite{kapouleas:equator}. The former in this article is similar to the one in \cite{wiygul:stacking,CSW}, while the latter is new. We then make use of the techniques from \cite[Section 9]{LDg}, which were further studied in \cite{zouthesis}, where a similar problem also rises from the lack of symmetry. 

Furthermore, to allow the perturbation, we also introduced $4J$ \emph{unbalancing} parameters $(\zetaboldu,\kappaboldu)$ at each gluing region. The parameters $r,\tau,h$ are then determined in \eqref{eq:r}, \eqref{eq:tau}, \eqref{eq:h}. The unbalancing parameters $(\zetaboldu,\kappaboldu)$ will create $4J$ mismatches $(\mubold,\mubold')$. We have to find their relationships and show that on the linear level this map is invertible (see \eqref{eq:muzeta}, \eqref{eq:muzetaprime} and \ref{lem:zetamu}). 

Unlike the minimal surfaces in $\R^3$, there are no rotationally invariant self-shrinkers like catenoid centred at any point. Thus as in \cite{LDg} we use the catenoids in the Gaussian metric $g_{\shr}$ (see \ref{not:g}) as the model and glue them with the graphs of MLD solutions to obtain the initial surfaces. Finally, we solve the linear problem of the linearised equation of self-shrinker from the technologies in \cite{kapouleas:kleene:moller,Chodosh:Schulze} and use Schauder's fixed point theorem to find the self-shrinker.

\subsection*{Organisation of the presentation}

The article consists of seven sections. 
The first section is the introduction. 
In Section \ref{S:not} we establish the notation. In Section \ref{S:RLD} we study the RLD and review the linear theory on $\R^2$ which is combined in section \ref{S:lin} with the linear theory on the catenoid to produce the global linear theory Proposition \ref{prop:lineareq}. 
In Section \ref{S:LD} we discuss the construction and estimates of the LD solutions by the RLD solution $\phi$. In Section \ref{S:MLD} we discuss the balancing conditions and construct the LD solutions $\varphi_j$ and the auxiliary functions $\underline{v}_j$.   
In Section \ref{S:init} we discuss the construction and properties of the initial surfaces. 
Finally in Section \ref{S:main} we state and prove the main Theorem. 

\subsection*{Acknowledgments}
The authors would like to thank Daniel Ketover for his support and for the helpful conversations about the topics of this article. JZ would like to thank Nicolaos Kapouleas for his interests and the helpful discussions.

\section{Elementary Geometry and Notation}\label{S:not}

\subsection*{General notation and conventions}

	  	\begin{notation}
	  	$ f(x) \lesssim_b g(x)$ indicates that there exists some absolute constant $C=C(b)>0$ that depends  only on $b$ so that $ f(x) \leq C g(x)$ for all $x$ in some specified domain. We will also use this notation for functions that depend on several variables or parameters. When the constant $C$ can be chosen absolutely, we will just write $f(x) \lesssim g(x)$.
	  	
	  	$ f(x) \sim_b g(x)$ indicates  that $ c(b)g(x) \leq f(x) \leq C(b)g(x)$ for some $C=C(b)>c=c(b)>0$ that depends only on $b$. When the constants $C$, $c$ can be chosen absolutely, we will just write $f(x) \sim g(x)$.
	  \end{notation}
	  
	  \begin{definition}
	  	\label{D:newweightedHolder}
	  	Assuming that $\Omega$ is a domain inside a manifold,
	  	$g$ is a Riemannian metric on the manifold, 
	  	$f,\rho:\Omega\to(0,\infty)$ are given functions, 
	  	$k\in \N$, 
	  	$\beta\in[0,1)$, 
	  	$u\in C^{k,\beta}_{loc}(\Omega)$ 
	  	or more generally $u$ is a $C^{k,\beta}_{loc}$ tensor field 
	  	(section of a vector bundle) on $\Omega$, 
	  	and that the injectivity radius in the manifold around each point $x$ in the metric $\rho^{-2}(x)g$
	  	is at least $1/10$,
	  	$\|u: C^{k,\beta} ( \Omega,\rho,g,f)\|$ is defined by
	  	$$
	  	\|u: C^{k,\beta} ( \Omega,\rho,g,f)\|:=
	  	\sup_{x\in\Omega}\frac{\,\|u:C^{k,\beta}(\Omega\cap B_x, \rho^{-2}(x)g)\|\,}{f(x) },
	  	$$
	  	where $B_x$ is a geodesic ball centered at $x$ and of radius $1/100$ in the metric $\rho^{-2}(x)g$.
	  	For simplicity, any of $\beta$, $\rho$, or $f$ may be omitted 
	  	when $\beta=0$, $\rho\equiv1$, or $f\equiv1$, respectively.
	  \end{definition}
	  
	  $f$ can be thought of as a ``weight'' function because $f(x)$ controls the size of $u$ in the vicinity of
	  the point $x$.
	  From the definition it follows that 
	  \begin{equation}
	  	\label{E:norm:derivative}
	  	\| \, \nabla u: C^{k-1,\beta}(\Omega,\rho,g,\rho^{-1}f)\|
	  	\le
	  	\|u: C^{k,\beta}(\Omega,\rho,g,f)\|, 
	  \end{equation}
	  and the multiplicative property 
	  \begin{equation}
	  	\label{E:norm:mult}
	  	\| \, u_1 u_2 \, : C^{k,\beta}(\Omega,g,\rho, f_1 f_2 \, )\|
	  	\lesssim_k	   
	  	\| \, u_1 \, : C^{k,\beta}(\Omega,g,\rho, f_1 \, )\|
	  	\,\,
	  	\| \, u_2 \, : C^{k,\beta}(\Omega,g,\rho, f_2 \, )\|.
	  \end{equation}

	  Cut-off functions will be used extensively,
	  and for this reason the following is adopted.
	  
	  \begin{definition}
	  	\label{DPsi} 
	  	A smooth function $\Psi:\R\to[0,1]$ is fixed with the following properties:
	  	\begin{enumerate}[(i)]
	  		\item 	$\Psi$ is non-decreasing.
	  		\item 	$\Psi\equiv1$ on $[1,\infty]$ and $\Psi\equiv0$ on $(-\infty,-1]$.
	  		\item  $\Psi-\frac12$ is an odd function.
	  	\end{enumerate}

	  Given now $a,b\in \R$ with $a\ne b$, the smooth function
	  $\psicut[a,b]:\R\to[0,1]$ is defined
	  by
	  \begin{equation}
	  	\label{Epsiab}
	  	\psicut[a,b]:=\Psi\circ L_{a,b},
	  \end{equation}
	  where $L_{a,b}:\R\to\R$ is the linear function defined by the requirements $L(a)=-3$ and $L(b)=3$.
	  
	  Clearly, $\psicut[a,b]$ has the following properties:
	  \begin{enumerate}[(i)]
	  	\item  $\psicut[a,b]$ is weakly monotone.
	  	\item  $\psicut[a,b]=1$ on a neighborhood of $b$ and 
	  	$\psicut[a,b]=0$ on a neighborhood of $a$.
	  	\item $\psicut[a,b]+\psicut[b,a]=1$ on $\R$.
	  \end{enumerate}

	  Suppose now we have two sections $f_0$, $f_1$ of some vector bundle over some domain $\Omega$. (A special case is when the vector bundle is trivial and $f_0, f_1$ are real-valued functions). Suppose we also have some real-valued function $d$ defined on $\Omega$. We define a new section
	  \begin{equation}
	  	\cutoff{a}{b}{d}(f_0,f_1):=f_1\psicut[a,b]\circ d +f_0\psicut[b,a]\circ d ,
	  \end{equation}
	  
	\end{definition}

	  \begin{notation}\label{not:manifold}
	  	For $(N^{n+1}, g)$ a Riemannian manifold, $\Sigma^n \subset N^{n+1}$ a two-sided hypersurface equipped with a
	  	(smooth) unit normal $\nu$, and $\Omega\subset \Sigma$, we introduce the following notation where any of $N$, $g$, $\Sigma$ or $\Omega$ may be omitted when clear from context.
	  	\begin{enumerate}[(i)]
	  		\item \label{item:neighbour} For $A\subset N$ we write $\dist_A^{N,g}$ for the distance function from $A$ with respect to $g$ and we define the tubular neighborhood of $A$ with radius $\delta>0$ by  $D_A^{N,g}(\delta):=\{p\in N:\dist_A^{N,g}(p)<\delta\}$. If $A$ is finite, we may just enumerate its points in both cases; for example, if $A=\{q\}$, we write $\dist^{N,g}_q(p)$. 
	  		\item We denote by $\exp^{N,g}$ the exponential map, by $\mathrm{dom}(\exp^{N,g}) \subset TN$ its maximal domain, and by $\mathrm{inj}^{N,g}$ the injectivity radius of $(N, g)$. Similarly, by $\exp_p^{N,g}$, $\mathrm{dom}(\exp_p^{N,g})$ and $\mathrm{inj}_p^{N,g}$ the same at $p\in N$.
	  		\item\label{item:graph} Given a function $f : \Sigma \to \R$ satisfying $\abs{f}(p) < \mathrm{inj}^{N,g}$,
	  		$\forall p \in \Omega$, we use the notation
	  		\begin{equation*}
	  			X^{N,g}_{\Omega,f}:=\exp^{N,g}\circ(f\nu)\circ I_\Omega^{N},\quad \Graph_{\Omega}^{N,g}(f):=X^{N,g}_{\Omega,f}(\Omega),
	  		\end{equation*}
	  		where $I_\Omega^{N}:\Omega\to N$ denotes the inclusion map of $\Omega$ in $N$. 
\item\label{item:sym} For $A\subset N$ we write $\sgr A{N,g}$ for the group of isometries of $(N,g)$ which preserve $A$ as a set. 
	  	\end{enumerate}
	  \end{notation}
	  
	  As in \cite[Definition A.1]{LDg}, we also define the Fermi exponential map.
	  \begin{definition}[Fermi exponential map]\label{def:fermi}
	  	 Given a hypersurface $\Sigma^n$ in a Riemannian manifold $(N^{n+1}, g)$ and a unit normal $\nu_p \in T_pN$ at some $p\in\Sigma $, for $\delta > 0$ we define
	  	 \begin{equation*}
	  	 	\hat{D}^{\Sigma, N, g}_p(\delta):=\{v+z\nu_p:v\in D^{T_p\Sigma,  g|_p}_O(\delta)\subset T_p\Sigma,z\in(-\delta,\delta)\}\subset T_pN.
	  	 \end{equation*}
	  	 For small enough $\delta$, the map $\exp^{\Sigma,N,g}_p:\hat{D}^{\Sigma, N, g}_p(\delta)\to N$, defined by
	  	 \begin{equation*}
	  	 	\exp^{\Sigma,N,g}_p(v+z\nu_p):=\exp^{N,g}_q(z\nu_v),\quad \forall v+z\nu_p\in \hat{D}^{\Sigma, N, g}_p(\delta)\text{ with }v\in T_p\Sigma,
	  	 \end{equation*}
	  	 where $q := \exp^{\Sigma,g}_p(v)$ and $\nu_v \in T_qN$ is the unit normal to $\Sigma$ at $q$ pointing to the same side of $D^{\Sigma,g|_{\Sigma}}_p(\delta)$ (which is two-sided) as $\nu_p$, is a diffeomorphism onto its image, which we will denote by $D^{\Sigma,N,g}_p\subset N$. We define the injectivity radius $\inj^{\Sigma,N,g}_p$ of $(\Sigma,N,g)$ at $p$ to be the supremum of such $\delta$’s. Finally, when $\delta<\inj^{\Sigma,N,g}_p$, we define the following on $D^{\Sigma,N,g}_p(\delta)$.
	  	 \begin{enumerate}[(i)]
	  	 	\item\label{item:proj}  $\Pi_{\Sigma}: D^{\Sigma,N,g}_p(\delta)\to \Sigma\cap   D^{\Sigma, N, g}_p(\delta)$ is the nearest point projection in $(D^{\Sigma,N,g}_p(\delta),g)$. Alternatively, $\Pi_{\Sigma}$ corresponds through $\exp^{\Sigma,N,g}_p$ to orthogonal projection to $T_p\Sigma$ in $(T_pN, g|_p)$.
	  	 	\item\label{item:zz} $\zz: D^{\Sigma,N,g}_p(\delta)\to (-\delta, \delta)$ is the signed distance from $\Sigma\cap   D^{\Sigma, N, g}_p(\delta)$ in $(D^{\Sigma, N, g}_p(\delta) g)$. Alternatively, $(\zz\circ\exp^{\Sigma,N,g}_p(v))\nu_p$ is the orthogonal projection of $v$ to $\langle\nu_p\rangle$ in $(T_pN, g|_p)$. 
	  	 	\item A foliation by the level sets $\Sigma_{z} := \zz^{-1}(z) \subset D^{\Sigma, N, g}_p(\delta)$ for $z \in (-\delta, \delta)$.
	  	 	\item Tensor fields $g_{\Sigma,\zz}$, $A_{\Sigma,\zz}$ 
	  	 	by requesting that on each level set $\Sigma_{z}$ they are equal to the
	  	 	first and second fundamental forms
	  	 	of $\Sigma_{z}$ respectively.
	  	 \end{enumerate}
	  \end{definition}
	  
	  \begin{notation}\label{not:g}
	  	We denote by $g_{Euc}$ the standard Euclidean metric on $\R^{n+1}$ and by $g_{\shr}$ the conformal metric $e^{2\omega} g_{Euc}$, where $\omega$ is
	  	\begin{equation}\label{eq:weight}
	  		\omega=-\frac{(\dist^{\R^{n+1},g_{Euc}}_O)^2}{4n}.
	  	\end{equation}
	  	
	  	We will also write $g$ instead of $g_{Euc}$ or omit $g$ in the notations defined in \ref{not:manifold} when it is clear from the context.
	  \end{notation}
	  \begin{notation}\label{not:multi}
	  	Given a function $f$ defined over some domain $\Omega$ and any function $u$ defined over $\Omega$ and $x\in\Omega$, we define the following notations.
	  	\begin{enumerate}[(i)]
	  		\item\label{item:multi} We denote by $\Tcal_f$ the \emph{multiplication operator by $f$}, i.e., $T_fu(x)=f(x)u(x)$.
	  		\item\label{item:df} We denote by $\dd_f$ the \emph{weighted differential by $f$}, i.e., $\dd_fu:=\Tcal_{e^{-f}}\dd(\Tcal_{e^{f}} u)=\dd u+u\dd f$. We use the similar notation to the other derivatives, e.g., $\partial_{f}u:=\Tcal_{e^{-f}}\partial(\Tcal_{e^{f}} u)$
	  	\end{enumerate}  
	  \end{notation}
	  
	  \begin{notation}\label{not:Fcal}
	  	For an oriented embedding hypersurface $\Sigma^n$ in $\R^{n+1}$ with embedding $X_{\Sigma}$, we use $g_{\Sigma}$, $A_{\Sigma}$, $\abs{A}^2_{\Sigma}$, $H_{\Sigma}$ and $\nu_{\Sigma}$ to denote the induced metric, the second fundamental form, the square of the second fundamental form, the mean curvature and the normal vector, respectively. We use $H^{\omega}_{\Sigma}$ to denote the \emph{weighted mean curvature} (cf. e.g., \cite{fminimal})
	  	\begin{equation*}
	  		H^{2\omega}_{\Sigma}:=H_{\Sigma}+\frac{1}{2}X_{\Sigma}\cdot\nu_{\Sigma}=e^{\omega}H_{\Sigma}^{\shr},
	  	\end{equation*}
	  	where $H_{\Sigma}^{\shr}$ is the mean curvature of $\Sigma$ in $(\R^{n+1},g_{\shr})$ (recall \ref{not:g}).
	  The Jacobi operators $\mathcal{L}_{\Sigma}$ in $(\R^{n+1},g_{Euc})$ and in $(\R^{n+1},g_{\shr})$ are defined by (from the calculation of \cite[Lemma 8.2]{s3doubling})
	  	\begin{align}\label{eq:jacobiH}
	  		&\mathcal{L}^{Euc}_\Sigma=\Delta_\Sigma+\abs{A}^2_{\Sigma},\\ &\mathcal{L}^{\shr}_\Sigma=e^{-2\omega}\left(\Delta_\Sigma+1+\frac{(2n-1)\abs{X_{\Sigma}\cdot\nu_{\Sigma}}^2-(n-1)\abs{X_{\Sigma}}^2}{4n^2}+\abs{A}^{2}_{\Sigma}-H_{\Sigma}\frac{X_{\Sigma}\cdot\nu_{\Sigma}}{n}  \right).\nonumber
	  	\end{align}
Finally, the linearised operator of $H^{2\omega}_{\Sigma}$ is defined by
\begin{align}\label{eq:jacobiF}
	\mathcal{L}_\Sigma=\Delta_\Sigma+\abs{A}^2_{\Sigma}-\frac{1}{2}X_{\Sigma}\cdot\nabla_{\Sigma}+\frac{1}{2}.
\end{align}
Notice that when $H^{2\omega}_\Sigma=0$, we have (for the calculation, see e.g., \cite[Appendix]{fminimal})
\begin{equation}\label{eq:jacobiHF}
	\mathcal{L}_\Sigma=\Tcal_{\exp(\omega)}\Lcal^{\shr}_{\Sigma} \Tcal_{\exp(\omega)}=-\Tcal_{\exp(-\omega)}\Hcal_{\Sigma} \Tcal_{\exp(\omega)},
	\end{equation}
	where $\Hcal$ is the Hamiltonian
	\begin{equation}\label{eq:Hamilton}
	\Hcal_{\Sigma}:=-\Delta_\Sigma-1+\frac{\abs{X_{\Sigma}\cdot\nu_{\Sigma}}^2+(n-1)\abs{X_{\Sigma}}^2}{4n^2}-\abs{A}^{2}_{\Sigma}.
	\end{equation}
	  \end{notation}
	  \subsection*{The coordinates and symmetries of the construction}
	  
	   Inspired by quantum mechanics, we will use $J$ to denote the total number of levels, as the \emph{total angular momentum quantum number}, which can be either integer (for an odd number of levels) or half-integer (for an even number of levels). 
	  
	  \begin{notation}\label{not:Jm}
	  	The constructions depend on a half-integer $J\in\frac{1}{2}\N$ and a large number $ m\in\N\setminus\{1\}$ which we now assume fixed, and $m$ is assumed to be as big as needed. 
	  \end{notation}
	  
	  We will use standard coordinates $(x_1, x_2, \zz)$ in $\R^3$, and we identify $\R^2\subset \R^3$ as
	  \begin{equation*}
	  	\R^2=\R^3\cap \{\zz=0\}.
	  \end{equation*}
	  We also define the cylindrical coordinates by the surjective map $\Theta:\R_+\cup\{0\}\times\R\to\R^2$ by 
	  \begin{align}\label{eq:Theta}
	  	\Theta(r, \theta)
	  	:=(r\cos\theta,r\sin\theta).
	  \end{align}
	  Apparently, $\Theta$ restricted to $\R_+\times\R$ is a covering map onto $\R^2\setminus\{O\}$. We also define the diffeomorphism $\Theta^{\Cylinder}:\R^2\to\R_+\times\R$ by
	  \begin{equation}\label{eq:ThetaCyl}
	  	\Theta^{\Cylinder}(s,\theta)=(e^s,\theta).
	  \end{equation}
	  Notice that $\Theta\circ\Theta^{\Cylinder}$ is conformal:
	  \begin{equation*}
	  	(\Theta\circ\Theta^{\Cylinder})^*g_{Euc}=e^{2s}(\ds^2+\dd\theta^2).
	  \end{equation*}
	  
	  We will also refer to 
	  \begin{equation} 
	  	\label{eq:circle}
	  		\uC[c]:=\R^2\cap\{x_1^2+x_2^2=c^2,\zz=0\}= \Theta(\{r=c,\zz=0\}).
	  	\end{equation}  
	  
	  Finally, for any interval $[a,b]\subset\R_+$ we define the circle in $\R^2$ by 
	  \begin{equation}\label{eq:ttorus}
	  	\Omega[a,b]:=\cup_{c\in [a,b]}\uC[c].
	  \end{equation}
	  
	  \begin{notation}\label{not:hgrp}
	  	We use $\mathfrak{O}(3)$ to denote the isometric group of $(\R^3,g_{Euc})$ with origin fixed, i.e., $\mathfrak{O}(3):=\mathrm{Isom}^{\R^3,g_{Euc}}_{O}$ (recall \ref{not:manifold}\ref{item:sym}).
	  	From the identification $\R^3\cong\R^2\times\R$, we can consider the embedding $\mathfrak{O}(2)\hookrightarrow \mathfrak{O}(3)$.
	  	We will then identify $\mathfrak{O}(2)$ with its image, i.e., $\mathfrak{O}(2):=\mathrm{Isom}^{\R^3,g_{Euc}}_{\R^2}$.
	  \end{notation}

	  We now define the reflections with respect to vertical planes $\hat{\sigma}_v[c]$, the reflection with respect to horizontal planes $\hat{\sigma}_h$, rotations $\hat{\CCC}[c],\hat{\CCC}_n$ and the reflections with respect to lines $\hat{\UUU}[c]$ in $\R_+\times \R$, where $c\in\R$ by
	  \begin{align}\label{eq:symhat}
	  	&\hat{\sigma}_v[c](r,\theta,\zz):=(r,2c-\theta,\zz),\quad\hat{\sigma}_h(r,\theta,\zz):=(r,\theta,-\zz),\\
	  	&\hat{\CCC}[c](r,\theta,\zz):=(r,\theta+c,\zz),\quad\hat{\CCC}_n:=\hat{\CCC}[2\pi/n],\nonumber\\
	  	&\hat{\UUU}[c](r,2c-\theta,\zz)=(r,2c-\theta,-\zz).\nonumber
	  \end{align}
	  We then define corresponding reflections with respect to vertical planes $\sigma_v[c]$, the reflection with respect to horizontal planes ${\sigma}_h$, rotations ${\CCC}[c],\CCC_n$ and the reflections with respect to lines ${U}[c]$ as elements in $\mathfrak{O}(3)$ by
	  \begin{align}\label{eq:sym}
	  	&\sigma_v[c](x_1,x_2,\zz):=(x_1 \cos 2c + x_2\sin 2c , x_1 \sin 2c - x_2 \cos 2c , \zz),\\
	  	&\sigma_h(x_1,x_2,\zz):=(x_1,x_2,-\zz),\nonumber\\
	  	&\CCC[c](x_1,x_2,\zz):=(x_1 \cos c - x_2 \sin c , x_1 \sin c + x_2 \cos c , \zz),\quad \CCC_n:=\CCC[2\pi/n],\nonumber\\
	  	&\UUU[c](x_1,x_2,\zz):=(x_1 \cos 2c + x_2\sin 2c , x_1 \sin 2c - x_2 \cos 2c , -\zz).\nonumber
	  \end{align}
	  We also define the reflections $\sigma[c]$, and rotations $\CCC[c],\CCC_n$ (by abuse of notation) as elements in $\mathfrak{O}(2)$ by
	  \begin{align}\label{eq:symduo}
	  	&\sigma[c](x_1,x_2):=(x_1 \cos 2c + x_2\sin 2c , x_1 \sin 2c - x_2 \cos 2c),\\
	  	&\CCC[c](x_1,x_2):=(x_1 \cos c - x_2 \sin c , x_1 \sin c + x_2 \cos c),\quad \CCC_n:=\CCC[2\pi/n].\nonumber
	  \end{align}
	  
	  We record the symmetries of $\Theta$ in the following lemma:
	  \begin{lemma}\label{lem:theta}
	  	We have the following.
	  	\begin{enumerate}[(i)]
	  		\item The group of deck transformations is generated by $\CCC_{2\pi}$.
	  		\item\label{item:thetaeqconj} $\sigma_{v}[c]\circ\Theta=\Theta\circ \hat{\sigma}_{v}[c]$, $\sigma_h\circ\Theta=\Theta\circ \hat{\sigma}_h$, $\CCC[c]\circ\Theta=\Theta\circ \hat{\CCC}[c]$ and $\UUU[c]\circ\Theta=\Theta\circ \hat{\UUU}[c]$.
	  	\end{enumerate}
	  \end{lemma}

	 We use the \emph{Schoenflies notation} for the \emph{point groups} in two and three dimensions.
	  \begin{definition}\label{def:grp}
	  	We define the group $\Grp=\Grp[m,J]\subset\mathfrak{O}(3)$ by
	  \begin{equation*}
	  	\Grp[m,J]=
	  	\begin{cases}
	  		&\mathbf{D}_{mh},\text{ when }J\notin \N\\
	  		&\mathbf{D}_{md},\text{ when }J\in \N.
	  	\end{cases}
	  \end{equation*}
	   where $\mathbf{D}_{mh}$ and $\mathbf{D}_{md}$ are subgroups in $\mathfrak{O}(3)$ defined by the following.
	  	\begin{enumerate}[(i)] 
	  		\item $\mathbf{D}_{mh}$ is generated by $\sigma_{v}[0],\sigma_{v}[\pi/m],\UUU[0]$.
	  		\item $\mathbf{D}_{md}$ is generated by $\sigma_{v}[0],\sigma_{v}[\pi/m],\UUU[\pi/(2m)]$.
	  	\end{enumerate} 
	  		We also define the groups $\mathbf{D}_m,\mathbf{D}_{2m}\subset\mathfrak{O}(2)$ by the following.
	  	\begin{enumerate}[(i)] 
	  			\item $\mathbf{D}_{m}$ is generated by $\sigma[0],\sigma[\pi/m]$.
	  			\item $\mathbf{D}_{2m}$ is generated by $\sigma[0],\sigma[\pi/(2m)]$.
	  		\end{enumerate} 
	  \end{definition}

	  \begin{remark}
	  	$\sigma_h\in \mathbf{D}_{mh}$ but $\sigma_h\notin \mathbf{D}_{md}$. Moreover, $\mathbf{D}_{mh}\cong D_m\times\Z_2$, and $\mathbf{D}_{md}\cong D_{2m}$, where $D_n$ is the abstract dihedral group with order $2n$.
	  \end{remark}
	  
	  \begin{notation}[Symmetric functions] 
	  	\label{N:G} 
	  	Given $\Omega\subset \R^3$ a surface
	  	invariant under the action of $\Grp[m,J]$, and a space of functions $\mathcal{X}\subset C^0(\Omega)$,
	  	we use a subscript ``sym'' to denote the subspace $\mathcal{X}_{\sym}\subset\mathcal{X}$ by
	  	\begin{equation*}
	  		\mathcal{X}_{\sym}:=\{u\in\mathcal{X}:u\circ g=\sgn(g)u, \forall g\in\Grp\},
	  	\end{equation*}
	  	where the signature of $g\in\Grp$ is defined by
	  	\begin{equation*}
	  		\sgn(g)=
	  		\begin{cases}
	  			1&\text{ when $g$ preserves the orientation of $\Omega$},\\
	  			-1&\text{ when $g$ reverses the orientation of $\Omega$}.
	  		\end{cases}
	  	\end{equation*} 
	  	
	  	Similarly, given $\Omega\subset \R^2$ 	invariant under the action of $\mathbf{D}_m$, and a space of functions $\mathcal{X}\subset C^0(\Omega)$, we use a subscript ``sym'' to denote the subspace $\mathcal{X}_{\sym}\subset\mathcal{X}$ consisting of the functions in $\mathcal{X}$ which are invariant under the action of $\mathbf{D}_m$.
	  \end{notation}

\section{Linear theory on $\R^2$}
\label{S:RLD}
Recall \eqref{eq:jacobiF}, the linearised operator on $\R^2$ is
\begin{equation}\label{eq:jacobi}
	\Lcal=\Lcal_{\R^2}:=\Delta-\frac{1}{2}X\cdot \nabla+\frac{1}{2}.
\end{equation}
And we have the corresponding Hamiltionian (recall \ref{eq:jacobiHR2})
\begin{equation}\label{eq:jacobiHR2}
	\Hcal=\Hcal_{\R^2}:-\Delta+\left(-1+\frac{r^2}{16}\right)=-\Tcal_{\exp(\omega)}\Lcal\Tcal_{\exp(-\omega)}.
\end{equation}

It will be easier later to state some estimates if we use a scaled metric on $\R^2$:
\begin{definition}
	We define the metric $\tilde{g}$ and coordinates $(\tilde{x}_1,\tilde{x}_2)$ on $\R^2$ by
	\begin{equation}\label{eq:tildeg}
		\tilde{g}:=m^2g_{Euc}=\dd \tilde{x}^2_1+\dd \tilde{x}^2_2,\quad (\tilde{x}_1,\tilde{x}_2):=m(x_1,x_2),
	\end{equation}
	and the corresponding operator by
	\begin{equation}\label{eq:tildeL}
		\tilde{\Lcal}:=\Delta_{\tilde{g}}+m^{-2}\left(-\frac{1}{2}\vec{\tilde{x}}\cdot \nabla_{\tilde{g}}+\frac{1}{2}\right)=m^{-2}\Lcal.
	\end{equation}
	
	Similarly, we define the coordinates $(\tilde{r},\tilde{\theta})$ and $(\tilde{s},\tilde{\theta})$ on $\R^2$ (recall \eqref{eq:Theta} and \eqref{eq:ThetaCyl}) by
	\begin{equation}\label{eq:tildetheta}
		(\tilde{r},\tilde{\theta})=m(r,\theta),\quad	(\tilde{s},\tilde{\theta})=m(s,\theta).
	\end{equation}
\end{definition}

\subsection*{Rotationally invariant solutions}
By a \emph{rotationally invariant function}, we mean a function on a domain of $\R^2$ which depends only on the coordinate $r$, i.e., invariant under the action of $\mathfrak{O}(2)$. When the solution $\phi$ is rotationally invariant, by calculating from \eqref{eq:jacobi}, the linearised equation $\Lcal\phi=0$ amounts to the ODE:
\begin{equation}\label{eq:jacobirotinvar}
	\frac{\dd^2\phi}{\dr^2}+\left(\frac{1}{r}-\frac{r}{2}\right)\frac{\dd\phi}{\dr}+\frac{1}{2}\phi=0.
\end{equation}
For the $k$-th Fourier coefficients of the equation, we will also consider the ODE:
\begin{equation}\label{eq:jacobik}
	\frac{\dd^2\phi}{\dr^2}+\left(\frac{1}{r}-\frac{r}{2}\right)\frac{\dd\phi}{\dr}+\left(\frac{1}{2}-k^2\right)\phi=0.
\end{equation}

The study of the ODE \eqref{eq:jacobirotinvar} has been done in \cite[Section 9]{ketover2024}. Here, we do a more thorough study for the application in this paper.
\begin{lemma}\label{lem:phimu}
	The space of solutions of the ODE \eqref{eq:jacobirotinvar} in $r$ on $\R^+$ is spanned by the two functions
	\begin{equation} 
		\label{eq:phiC}
		\begin{aligned}
			\phi_{m,k}:= M\left(k^2-\frac{1}{2},1,\frac{r^2}{4}\right),\quad\phi_{u,k}:=U\left(k^2-\frac{1}{2},1,\frac{r^2}{4}\right), 
		\end{aligned} 
	\end{equation} 
	where $M(\cdot,\cdot,\cdot)$ and $U(\cdot,\cdot,\cdot)$ are Kummer and Tricomi confluent hypergeometric functions (see e.g., \cite[Chapter 13]{NIST}).  Moreover, the following hold (we omit $k$ when $k=0$).
	\begin{enumerate}[(i)]
		\item\label{item:phimusing}  When $r\to\infty$,
		\begin{equation*}
			\phi_{m,k}\sim \frac{r^{2k^2-3}e^{r^2/4}}{2^{2k^2-3}\Gamma(k^2-1/2)},\quad \phi_{u,k}\sim \left(\frac{r}{2}\right)^{1-2k^2}.
		\end{equation*}
		Moreover, $\phi_{m}$ is smooth at $r=0$ and $\phi_m(0)=0$; $\phi_{u}\sim\frac{1}{2\sqrt{\pi}}\log r$ when $r\to 0$.
		\item\label{item:wronskian} The Wronskian $W[\phi_m,\phi_u]:=\phi_m\phi_u'-\phi_u\phi_m'$ satisfies
		\begin{equation}\label{eq:Wronskian}
			W[\phi_m,\phi_u](r)=\frac{e^{r^2/4}}{\sqrt{\pi}r}.
		\end{equation}
		\item\label{item:monotonicity} $\phi_m$ is strictly decreasing on $\R_+$, $\phi_u$ is strictly increasing on $\R_+$. Moreover, both $\phi_m$ and $\phi_u$ are strictly concave.
		\item\label{item:rroot} On $\R_+$, $\phi_{m}$ has a unique root $r_m$, and $\phi_{u}$ has a unique root $r_u$. Moreover, $r_u<r_m$.
		\item\label{item:hhatmonotone} The function $\frac{\phi_m'}{\phi_m}+\frac{\phi_u'}{\phi_u}-\frac{r}{2}$ is strictly decreasing from $+\infty$ to $-\infty$ on $(r_u,r_m)$.
		\item\label{item:rbalanced} There is a unique $\rbalanced\in \R^+$ such that (recall \ref{not:multi}\ref{item:df})
		\begin{equation}\label{eq:rbalanced}
			\frac{1}{\phi_m}\frac{\dd_{\omega}\phi_m}{\dr}(\rbalanced)+\frac{1}{\phi_u}\frac{\dd_{\omega}\phi_u}{\dr}(\rbalanced)=\frac{\phi_m'}{\phi_m}(\rbalanced)+\frac{\phi_u'}{\phi_u}(\rbalanced)-\frac{\rbalanced}{2}=0,
		\end{equation}
		and $\phi_{m}(\rbalanced)>0$, $\phi_{u}(\rbalanced)>0$.
	\end{enumerate}
\end{lemma}
\begin{proof}
	As in \cite[Section 9]{ketover2024}, the change of variable $\xi:=r^2/4$ in \eqref{eq:jacobirotinvar} leads to the Kummer’s equation (see e.g., \cite[(13.2.1)]{NIST}) $\xi\phi_{\xi\xi}+(1-\xi)\phi_{\xi}+(-k^2+1/2)\phi=0$, thus the space of solutions is spanned by $\phi_{m,k}$ and $\phi_{u,k}$. \ref{item:phimusing} follows by \cite[(13.2.6) and (13.2.23)]{NIST} when $r\to\infty$ and by \cite[(13.2.13) and (13.2.19)]{NIST} when $r\to 0$. \ref{item:wronskian} follows by \cite[(13.2.34)]{NIST}.
	
	The part in \ref{item:monotonicity} related to $\phi_m$ follows by differentiating the Taylor expansion expression of $M(-\frac{1}{2},1,\xi)$ (see e.g., \cite[13.2.2]{NIST}). From the differentiation formula \cite[(13.3.22)]{NIST} and the integral representation \cite[(13.4.4)]{NIST},
	\begin{equation*}
		U'(-1/2,2,\xi)=\frac{1}{2}U(1/2,2,\xi)=\frac{1}{2\sqrt{\pi}}\int_0^{\infty}e^{-\xi t}\sqrt{\frac{1+t}{t}}\dt>0.
	\end{equation*}
	As $\phi'_u(r)=rU'(-1/2,1,r^2/4)/2$, $\phi_u$ is strictly increasing. Moreover, by \cite[Corollary 3.2(ii)]{Mao2025}
	\begin{equation*}
		\frac{-U'(1/2,2,\xi)}{U(1/2,2,\xi)}>\frac{0.5(\xi+1.5)}{\xi(\xi+1)}>\frac{1}{2\xi}.
	\end{equation*}
	Thus $\phi''_u(r)=\frac{1}{2}U'(-1/2,1,r^2/4)+\frac{r^2}{4}U''(-1/2,1,r^2/4)=\frac{1}{4}U(1/2,1,r^2/4)+\frac{r^2}{8}U'(1/2,2,r^2/4)<0$. The part in \ref{item:monotonicity} related to $\phi_u$ then follows.
	
	The existence and uniqueness of $r_m$ and $r_u$ in \ref{item:rroot} simply follow from the asymptotics in \ref{item:phimusing} and monotonicities in \ref{item:monotonicity}. The positiveness of Wronskian in \ref{item:wronskian} shows that $\phi_m$ and $\phi_u$ can not both be negative. This implies $r_u<r_m$ in \ref{item:rroot}. Moreover, the condition that $\phi_m$ and $\phi_u(r)$ are both positive only holds when $r\in(r_u,r_m)$. The monotonicities and concavities in \ref{item:monotonicity} then imply \ref{item:hhatmonotone}. \ref{item:rbalanced} then follows.
\end{proof}

\begin{remark}\label{rk:rmu}
	Numerically, $r_m\approx 2.51$, $r_u\approx 0.88$, $\rbalanced\approx 1.52$.
\end{remark}

\subsection*{Linear theory on $\R^2$}
\begin{lemma}\label{lem:nokernel}
	$(\ker\Lcal)_{\sym}\subset L^2(\R^2,e^{2\omega})$ is trivial when $m\geq 2$. 
\end{lemma}

\begin{proof}
	By \eqref{eq:jacobi}, the operator $\Lcal$ is conjugate with the hamiltonian operator for harmonic oscillator. The space $\ker\Lcal$ is then the same and is spanned by the two coordinate functions $x_i$, $i=1,2$, restricted to $\R^2$ (see e.g., \cite[Lemma 8.1]{kapouleas:kleene:moller}). The conclusion then follows.
\end{proof}

\begin{lemma}[Weighted $L^2$-estimates, cf. {\cite[Lemma 8.1]{kapouleas:kleene:moller}} and {\cite[Theorem 3.14]{Chodosh:Schulze}}]
	Given $E\in L^2_{\sym}(\R^2,e^{2\omega})$ and $m\geq 2$, there is a unique $u\in H^2_{\sym}(\R^2,e^{2\omega})$ such that $\Lcal u=E$.
	
\end{lemma}

We will need to study the operator \eqref{eq:jacobi} over $\R^2$. As $\R^2$ is not compact, we have to develop the linear theory over a suitable Banach space inside $L^2(\R^2,e^{2\omega})$, which is equivalent to imposing some condition at infinity. From the asymptotics of each mode in \ref{lem:phimu}, we can see that a good choice is the set of functions which has at most the linear growth asymptotically. We will follow the definitions and make use of the results in \cite[Section 8]{kapouleas:kleene:moller} and \cite[Section 3]{Chodosh:Schulze}.

\begin{definition}[Homogeneously weighted H\"older spaces, {\cite[Definition 8.2]{kapouleas:kleene:moller} and \cite[Definition 3.1]{Chodosh:Schulze}}]
	For $\Omega\subset\R^2$, the \emph{homogeneously weighted spaces of H\"older functions for decay rate $k\in\mathbb{N}$} are defined as:
	\begin{equation*}
			C_{\mathrm{hom},-k}^{0,\beta}(\Omega):=\{f\in C^{0,\beta}_\mathrm{loc}(\Omega):\norm{f:C_{\mathrm{hom},-k}^{0,\beta}(\Omega)}<\infty\},
	\end{equation*}
	where the norm is defined by
	\begin{equation*}
			\norm{f:C_{\mathrm{hom},-k}^{0,\beta}(\Omega)}:=\sup_{x,y\in\Omega}\frac{1}{|x|^{-k-\beta}+|y|^{-k-\beta}}\frac{|f(x)-f(y)|}{|x-y|^\beta}+\sup_{x\in\Omega} |x|^k|f(x)|.
	\end{equation*}

	We then define
	\begin{align*}
		&C_{\mathrm{hom},-k}^{2,\beta}(\Omega):=\{f\in C_\mathrm{loc}^{0,\beta}(\Omega):\partial^{(l)}f\in C_{\mathrm{hom},-k}^{0,\beta}(\Omega),\: |(l)|\leq 2\},
	\end{align*}
	where $(l)$ ranges over all multi-indices, with norm given by
	\begin{equation*}
		\norm{f:C_{\mathrm{hom},-k}^{2,\beta}(\Omega)}:=\sum_{|(l)|\leq 2}\norm{\partial^{(l)} f:C_{\mathrm{hom},-k}^{0,\beta}(\Omega)}.
	\end{equation*}
\end{definition}

\begin{definition}[Anisotropically homogeneously weighted H\"o{}lder spaces, {\cite[Definition 8.3]{kapouleas:kleene:moller} and \cite[Definition 3.2]{Chodosh:Schulze}}] \label{def:anisospaces}
	The \emph{anisotropically homogeneously weighted H\"o{}lder spaces} are defined as
	\begin{equation*}
			C_{\mathrm{an},-1}^{2,\beta}(\Omega):=\{f\in C_{\mathrm{hom},-1}^{2,\beta}(\Omega): X\cdot\nabla f\in C_{\mathrm{hom},-1}^{0,\beta}(\Omega)\},
	\end{equation*}
	with the norm defined by
	\begin{equation*}
		\norm{f:C_{\mathrm{an},-1}^{2,\beta}(\Omega)}:=\norm{f:C_{\mathrm{hom},-1}^{2,\beta}(\Omega)}+\norm{X\cdot\nabla f:C_{\mathrm{hom},-1}^{0,\beta}(\Omega)}.
	\end{equation*}
\end{definition}

\begin{definition}[Cone H\"o{}lder spaces, {\cite[Definition 8.5]{kapouleas:kleene:moller} and \cite[Definition 3.3]{Chodosh:Schulze}}]\label{def:conespaces}
	For $\Omega\subset\R^2$ with $\R^2\setminus\Omega$ compact and $R:=\mathrm{diam}(\R^2\setminus\Omega)$, the (anisotropic homogeneous) cone H\"older space 
	are defined as
	\begin{align*}
		\mathcal{CS}^{0,\beta}(\Omega):=C_{\mathrm{hom},-1}^{0,\beta}(\Omega),\quad
		\mathcal{CS}^{2,\beta}(\Omega):=C^{2,\beta}(\Sph^1)\times C_{\mathrm{an},-1}^{2,\beta}(\Omega).
	\end{align*}
	An element $(c,f)\in\mathcal{CS}^{2,\beta}(\Omega)$ will be considered as a function on $\Omega$ given by
	\begin{equation*}
		u_{(c,f)}(r,\theta)=\cutoff{R+1}{2R+1}{r}(0,rc(\theta))+f(r,\theta),
	\end{equation*}
	with the norm defined by
	\begin{equation*}
		\norm{u:\mathcal{CS}^{2,\beta}(\Omega)}=
		\norm{(c,f):\mathcal{CS}^{2,\beta}(\Omega)}:=\norm{c:{C^{2,\beta}(\Sph^1)}}+\norm{f:C_{\mathrm{an},-1}^{2,\beta}(\Omega)}.
	\end{equation*}

\end{definition}

\begin{lemma}[Schauder estimates, {\cite[Theorem 8.9]{kapouleas:kleene:moller} and \cite[Corollary 3.17]{Chodosh:Schulze}}]\label{lem:linearRduo} 
	For any $E\in \mathcal{CS}^{0,\beta}_{\sym}(\R^2)$ there exists a unique pair $(c,f)\in\mathcal{CS}^{2,\beta}_{\sym}(\R^2)$ such that
	\[
	\Lcal u_{(c,f)}=E.
	\]
	Moreover,
	\begin{equation*}
		\norm{(c,f):\mathcal{CS}^{2,\beta}(\R^2)}\lesssim \norm{E:\mathcal{CS}^{0,\beta}(\R^2)}.
	\end{equation*}
\end{lemma}

  \section{LD solutions}
  \label{S:LD} 
	  \subsection*{Green's function and LD solutions}
	    We now discuss the Green’s function for $\Lcal$. 
	  \begin{lemma}\label{lem:green}
	  	There is a constant $\deltau>0$, such that $\forall p\in D_O(10\rbalanced)$ (recall \ref{lem:phimu}\ref{item:rbalanced}), there exists a function $G_p \in C^{\infty}(D_p(\deltau)\setminus\{p\})$  satisfying the following:
	  	\begin{enumerate}[(i)]
	  		\item\label{item:greenp} $\Lcal G_p= 0$ on $D_p(\deltau)\setminus\{p\}$.
	  		\item\label{item:greenest} $\norm{\Tcal_{\exp(\omega)}G_p-\log\dist^{g_{\shr}}_p:C^k(D_p(\deltau),\dist^{g}_p,g,(\dist^{g}_p)^2\log\dist^{g}_p)}\lesssim_k 1$.
	  		\item\label{item:greenuni}
	  		if $G'_p$ is another Green’s functions satisfying \ref{item:greenp} and \ref{item:greenest} for some $\deltau' > 0$, then the unique extension $G''\in C^{\infty}(D_p(\deltau'')) $ of $G_p-G'_p$ satisfies $G''(p)=0$ and $\dd_pG''=0$, where $\deltau''=\min\{\deltau,\deltau'\}$.
	  	\end{enumerate}
	  \end{lemma}
	  \begin{proof}
        By \cite[Lemma 3.3]{LDg}, there exists a Green's function $G^{\shr}_p$ to the operator $\Lcal_{\R^2}^{\shr}$ (recall \eqref{eq:jacobiH}) such that $\Lcal_{\R^2}^{\shr} G^{\shr}_p=0$ and $G^{\shr}_p$ satisfies the estimate in \ref{item:greenest} with $g_{\shr}$ instead of $g$ on some small neighbourhood of $p$ along with the property in \ref{item:greenuni}. The lemma then follows by defining $G_p:=e^{-\omega}G_{\shr}$ and using the equivalence of $\dist^{g}$ and $\dist^{g_{\shr}}$ on a neighbourhood of $p$.
	  \end{proof}
	  
	   Because of the symmetry imposed on the constructions, we concentrate now on LD solutions which are invariant under the action of $\mathbf{D}_m$ (recall \ref{def:grp}). Moreover, although the singular set is not generally contained on the same circle, by the restriction of the symmetry, it can be written as the disjoint union of two sets $L^0\sqcup L^1$, and it is contained on two circles (recall \eqref{eq:circle}) $\uC^0:=\uC[r^0]$ and $\uC^1:=\uC[r^1]$.
	   We will assume from now on (recall \ref{lem:phimu} and \ref{rk:rmu})
	   \begin{equation}
	   	r^0,r^{1}\in((r_u+\rbalanced)/2,(r_m+\rbalanced)/2).
	   \end{equation}
	  We then fix the positions of the points of singularity on $L^0$ and $L^1$ such that they are compatible with the coordinates.
	  \begin{definition}\label{def:L}
	  	We define finite sets $L^0,L^1\subset \R^2$ of points ($i=0,1$):
	  	 \begin{equation}
	  		L^{i}=L^{i}[r^{i};m]:=\mathbf{D}_mp_i,
	  	\end{equation}
	  	where $p_0=\Theta(r^{0},0)=(r^{0},0,0)$, $p_1=\Theta(r^{1},\pi/m)=(r^{1}\cos(\pi/m),r^{1}\sin(\pi/m),0)$.
	  \end{definition}
	  \begin{remark}\label{rk:D2m}
	  	When $r_0=r_1$, $L^{0}\sqcup L^{1}=\mathbf{D}_{2m}p_0$.
	  \end{remark}
	  \begin{definition}[LD solutions]\label{def:LD}
	  	 We call a $\mathbf{D}_m$-symmetric function $\varphi\in L^2(\R^2,e^{2\omega})$ a \emph{linearised doubling (LD) solution} on $\R^2$ when there exists two numbers $\tau_{\pm}\in \R_+$ and two $\mathbf{D}_m$-symmetric finite sets $L_{\pm}$ (one of them could be empty) satisfying the following.
	  	\renewcommand{\theenumi}{\roman{enumi}}
	  	\begin{enumerate}
	  		\item\label{item:LDCS} $\varphi\in C^{\infty}(\R^2\setminus (L_+\sqcup L_-))\cap\mathcal{CS}^{2,\beta}(\R^2\setminus (L_+\sqcup L_-))$ (recall \ref{def:conespaces}) for some $\beta\in(0,1)$ and $\Lcal\varphi=0$ on $\R^2\setminus (L_+\sqcup L_-)$. 
	  		\item\label{item:LDlog} $\forall p_{\pm}\in L_{\pm}$, the function $\Tcal_{\exp(\omega)}\varphi\mp\tau_{\pm}\log\dist^{g_{\shr}}_{p_{\pm}}$ is bounded on some neighborhood of $p_{\pm}$ in $\R^2$ (we use the convention $\tau_{\pm}=0$ when $L_{\pm}$ is empty). 
	  	\end{enumerate}	
	  \end{definition}

	   \begin{definition}[The constants $\delta$]\label{def:delta}
	   	We define a constant $\delta>0$ by
	   	\begin{equation}\label{eq:delta}
	   		\delta=\delta[m]:=1/(100m).
	   	\end{equation}
	   	Moreover, given $L:=L^0\sqcup L^1$ as in \ref{def:L}, we assume that $m$ is big enough so that the following are satisfied.
	  	\renewcommand{\theenumi}{\roman{enumi}}
	  	\begin{enumerate}
	  		\item $\forall p,p'\in L$ with $p\neq p'$, we have $D_p(9\delta)\cap D_{p'}(9\delta)=\emptyset$.
	  		\item $\forall p\in L$, $\delta<\mathrm{inj}_p^{\R^2,\R^3,g_{\shr}}$, where $\mathrm{inj}_p^{\R^2,\R^3,g_{\shr}}$ is the injective radius as in \ref{def:fermi}.
	  	\end{enumerate}	
	  \end{definition}

	  \begin{lemma}\label{lem:LD}
	 Given two numbers $\tauu:=(\tau_+,\tau_-)\in\R^2_+$, and two finite sets $\Lu=(L_+,L_-)$ as in \ref{def:LD}, there is a unique $\mathbf{D}_m$-symmetric LD solution $\varphi =\varphi(\tauu;\Lu;m)= \varphi[\tau_+,\tau_-;L_{+},L_{-};m]$ satisfying the conditions in \ref{def:LD}. 
Moreover the following hold. 
	  	\begin{enumerate}[(i)]
	  		\item  $\varphi$ depends linearly on $\tauu$.
	  		\item\label{item:varphiavgcont} $\varphi_{\avg} \in C^0(\R^2)\cap C^{\infty}(\R^2\setminus (\uC_0\sqcup \uC_1))$ and on $\R^2\setminus (\uC_0\sqcup \uC_1)$ it satisfies the ODE $\Lcal\varphi_{\avg}=0$.
	  		\item\label{item:extrasym} When $\tau_+=\tau_-$ and $\sigma[\pi/2m](L_+)=L_-$, $\varphi$ satisfies the extra odd symmetry: $\varphi\circ\sigma[\pi/2m]=-\varphi$ (recall \eqref{eq:symduo} and \ref{rk:D2m}).
	  	\end{enumerate}
	  	
	  \end{lemma}
	  \begin{proof}
	  	We denote $L=L_+\sqcup L_-$ in this proof. We define $\varphi_1 \in C^{\infty}_{\sym}(\R^2\setminus L)$ by requesting that it is supported on $D_{L}(2\delta)$ and $\varphi_1=\cutoff{\delta}{2\delta}{\dist_p}(\pm\tau_{\pm} G_p,0)$ (recall \ref{lem:green}\ref{item:greenp}) on $D_p(2\delta)$ for each $p \in L_{\pm}$. Note that the function $\Lcal\varphi_1\in C^{\infty}_{\sym}(\R^2)$ (by assigning 0 values on $L$) and it is supported on $\sqcup_{p\in L}D_p(2\delta)\setminus D_p(\delta)$. Because of the symmetries, by \ref{lem:linearRduo} and the standard regularity theory, there is a function $\varphi_2\in C^{\infty}_{\sym}(\R^2)\cap\mathcal{CS}^{2,\beta}_{\sym}(\R^2)$ such that $\Lcal\varphi_2=-\Lcal\varphi_1$. We can then define $\varphi := \varphi_1 + \varphi_2$. Uniqueness and linearity follow immediately. To prove \ref{item:varphiavgcont} we need to check that $\varphi$ is integrable on $\uC_0$, $\uC_1$ and that $\varphi_{avg}$ is also continuous there. But these follow easily by the logarithmic singularity behaviour of $G_p$ in \ref{lem:green}. Finally, \ref{item:extrasym} follows by the linearity and the uniqueness.
	  \end{proof}
	  
	  \begin{definition}\label{def:Phi}
	  	For $i=0,1$, $L^0,L^1$ as in \ref{def:L}, by making use of \ref{lem:LD}, we define the LD solution $\Phi^{i}=\Phi[L^i] := \varphi[1,0;L^i,\emptyset;m] \in C^{\infty}_{\sym}(\R^2\setminus L^i)$.
	  \end{definition}

	  \subsection*{Estimates on $\Phi_{\avg}$}
	  	In this and two following subsections we assume $\Phi$ is either one of $\Phi^{i}$, $L$ is either one of $L^i$, $\ru$ is either one of $r^i$, $i=0,1$ (recall \ref{def:Phi}), and $\underline{\uC}$ is either one of $\uC_i$. 
	  
	In order to estimate the mismatch (defined in \ref{def:mismatch}) of linear combinations of function $\Phi$'s, we have to study the behaviour of $\Phi$ around its singularities. Overall, the idea is to make use of its average $\Phi_{avg}$, which is much easier to understand.  
	  	
	  	\begin{lemma}\label{lem:Phiavg}
	  		$\phi:=\Phi_{avg}$ is given by (recall \ref{lem:phimu} and \ref{lem:LD}\ref{item:varphiavgcont})
	  		\begin{equation}\label{eq:phi}
	  			\phi=
	  				\begin{cases} 
	  					\phi_1e^{\ru^2/8}\frac{\phi_{u}}{\phi_{u}(\ru)}, &\text{ when  } r\geq \ru,\\
	  					\phi_1e^{\ru^2/8}\frac{\phi_{m}}{\phi_{m}(\ru)}, &\text{ when } r \leq \ru,
	  				\end{cases}
	  		\end{equation}
	  		where $\phi_1=\phi_1[\ru]$ is given by
	  		\begin{equation}\label{eq:phiunum}
	  			\phi_1[\ru]=\sqrt{\pi}m\phi_{m}(\ru)\phi_{u}(\ru)e^{-\ru^2/4}.
	  		\end{equation}
	  	\end{lemma}
	  	\begin{proof}
	  		Because of the asymptotics of $\varphi$ at infinity \ref{def:LD}\ref{item:LDCS} and \ref{lem:phimu}\ref{item:phimusing}, it is clear that $\phi=A_m\phi_m/\phi_m(\ru)$ when $r\leq \ru$ and $\phi=A_u\phi_u/\phi_u(\ru)$ when $r\geq \ru$ for some constants $A_m,A_u$. The continuity from \ref{lem:LD}\ref{item:varphiavgcont} implies that $A_m=A_u$. For $0 < \epsilon_1 \ll \epsilon_2$, we now consider the domain $\Omega_{\epsilon_1, \epsilon_2}:=D_{\underline{\uC}}(\epsilon_2)\setminus D_{L}(\epsilon_1) $. By integrating $\Lcal\Phi=0$ on $\Omega_{\epsilon_1, \epsilon_2}$ and integrating by parts, we obtain
	  		\begin{equation*}
	  			\oint_{\partial\Omega_{\epsilon_1, \epsilon_2}}\left(\nabla\Phi-\frac{1}{2}\Phi X\right)\cdot\vec{\eta}+\frac{3}{2}\int_{\Omega_{\epsilon_1, \epsilon_2}}\Phi=0,
	  		\end{equation*}
	  		where $\vec{\eta}$ is the outward conormal. By taking the limit as $\epsilon_1\to 0$ first and then as $\epsilon_2\to 0$, we obtain by using the logarithmic behaviour of $\Phi$ near $L$ from \ref{def:LD}\ref{item:LDlog}, that
	  		\begin{equation*}
	  			2\pi \ru\left(A_u\frac{\phi_u'}{\phi_u}(\ru)-A_m\frac{\phi_m'}{\phi_m}(\ru)\right) =2\pi m e^{\ru^2/8}.
	  		\end{equation*}
	  		Therefore,
	  		\begin{equation*}
	  			A_m=A_u=\frac{m\phi_m(\ru)\phi_u(\ru)e^{\ru^2/8}}{\ru W[\phi_m,\phi_u](\ru)}.
	  		\end{equation*}
	  		The result then follows by \eqref{eq:Wronskian}.
	  	\end{proof}
	  	
	  	\begin{definition}\label{def:phiju}
	  		For $a,b,c\in\mathbb{R}$, we define the rotationally invariant solutions $\underline{\phi}[a,b]=\underline{\phi}[a,b;\ru]$ and (with singularity) $\underline{j}[c]=\underline{j}[c;\ru]$ by requesting the following initial data (recall \ref{not:multi}\ref{item:df})
	  		\begin{equation*}
	  			\underline{\phi}(\ru)=e^{\ru^2/8}a,\quad \partial_{\omega} \underline{\phi}(\ru)=e^{\ru^2/8}b;\quad \underline{j}(\ru)=0,\quad \partial_{\omega,+} \underline{j}(\ru)=-\partial_{\omega,-} \underline{j}(\ru)=e^{\ru^2/8}c,
	  		\end{equation*}
	  		and the ODEs $\Lcal\underline{\phi}=0$ on $\R^2\setminus\{O\}$ and $\Lcal\underline{j}=0$ on $\R^2\setminus(\{O\}\cup \underline{\uC})$.
	  	\end{definition}
	  	
	  	\begin{notation}\label{not:hhat}
	  		We use $\hat{h}$ to denote the function (recall \eqref{eq:Wronskian})
	  		\begin{equation*}
	  			\hat{h}:=\frac{\phi_m\phi_u'+\phi_u\phi_m'-r\phi_m\phi_u/2}{2r W[\phi_m,\phi_u]}=\frac{\sqrt{\pi}}{2}e^{-r^2/4}\phi_m\phi_u\left(\frac{\phi_m'}{\phi_m}+\frac{\phi_u'}{\phi_u}-\frac{r}{2}\right).
	  		\end{equation*}	
	  		Notice by \ref{lem:phimu}\ref{item:hhatmonotone}-\ref{item:rbalanced}, $\hat{h}(\rbalanced)=0$ and $\hat{h}$ is strictly decreasing in a neighbourhood of $\rbalanced$. We use $\hat{h}^{-1}$ to denote its inverse function in this neighbourhood.
	  	\end{notation}

	  	\begin{corollary}
	  		On $\R^2\setminus(\{0\}\cup \underline{\uC})$ we have
	  		\begin{equation}\label{eq:Phiavg}
	  			\phi=\underline{\phi}[\phi_1[\ru],m\hat{h}(\ru);\ru]+\underline{j}[m/(2\ru);\ru].
	  		\end{equation}
	  	\end{corollary}
	  	\begin{proof}
	  		By \eqref{eq:phiunum} and \ref{not:hhat}, $2m\hat{h}(\ru)=\phi_1\left(\frac{\phi_m'}{\phi_m}(\ru)+\frac{\phi_u'}{\phi_u}(\ru)-\frac{\ru}{2}\right)$, $m/\ru=\phi_1\left(-\frac{\phi_m'}{\phi_m}(\ru)+\frac{\phi_u'}{\phi_u}(\ru)\right)$.
	  		It then follows by \eqref{eq:phi} and \ref{def:phiju}.
	  	\end{proof}
	  	
	  	  	\begin{notation}\label{not:phihat}
	  		We use $\hat{\phi}=\hat\phi[\ru]$ to denote the function (recall \eqref{eq:phiunum} and \eqref{eq:Phiavg})
	  		\begin{align*}
	  			\hat{\phi}:=&\frac{\phi-\cutoff{2/m}{3/m}{\dist_{\underline{\uC}}}(\underline{j}[m/(2\ru)],0)}{\phi_1}\\
	  			=&
	  				\underline{\phi}\left[1,\frac{e^{\ru^2/4}\hat{h}(\ru)}{\sqrt{\pi}\phi_m(\ru)\phi_u(\ru)}\right]+\cutoff{2/m}{3/m}{\dist_{\underline{\uC}}}\left(0,\frac{e^{\ru^2/4}\underline{j}[1/(2\ru)]}{\sqrt{\pi}\phi_m(\ru)\phi_u(\ru)}\right).
	  		\end{align*}
	  	\end{notation}
	  	
	  	\begin{definition}\label{def:hatrs}
	  	We define a scaled shifted coordinate $\hat{r}=\hat{r}[m]:=m(r-\ru)$. We then define the function $\hat{s}=\hat{s}[m]:=m\log\frac{r}{\ru}=m\log\left(1+\frac{\hat{r}}{m\ru}\right)$, equivalently (recall \eqref{eq:ThetaCyl}),
	  	\begin{equation*}
	  		\Theta^{\Cylinder}(\hat{s}/m+\su,\theta)=(r,\theta),
	  	\end{equation*}
	  	where $\su$ is defined such that $\Theta^{\Cylinder}(\su,\theta)=(\ru,\theta)$.
	  \end{definition}
	  \begin{definition}
	  	For any constant $C\geq 1$ independent of $m$, we define the reflection operator (with respect to $\ru$) $\Rcal=\Rcal_{\ru}:C^0(D_{\underline{\uC}}(C/m))\to C^0(D_{\underline{\uC}}(C/m))$ by $\Rcal u(r,\theta):=u(2\ru-r,\theta)$. We then define the antisymmetric operator (with respect to $\ru$) $\Acal=\Acal_{\ru}:C^0(D_{\underline{\uC}}(C/m))\to C^0(D_{\underline{\uC}}(C/m))$ by $\Acal u:=u-\Rcal u$.
	  	\end{definition}

	  		\begin{lemma}\label{lem:phiju}
	  			The following hold (recall \eqref{eq:tildeg} and \ref{def:hatrs}) for $x\in(0,3]$.
	  			\begin{enumerate}[(i)]
	  				\item\label{item:phiu} $\norm{\Tcal_{\exp(\omega)}\underline{\phi}[1,0]-1:C^k(D_{\underline{\uC}}(x/m)\setminus\underline{\uC},\tilde{g})}\lesssim_k x^2/m^2$.
	  				\item\label{item:phiub} $\norm{\Tcal_{\exp(\omega)}\underline{\phi}[0,m]-\hat{s}:C^k(D_{\underline{\uC}}(x/m)\setminus\underline{\uC},\tilde{g})}\lesssim_k x^2/m^2$.
	  				\item \label{item:ju} $\norm{\Tcal_{\exp(\omega)}\underline{j}[m]-\abs{\hat{s}}:C^k(D_{\underline{\uC}}(x/m)\setminus\underline{\uC},\tilde{g})}\lesssim_k x^2/m^2$.
	  				\item\label{item:phiuA}   $\norm{\Acal\Tcal_{\exp(\omega)}\underline{\phi}[1,0]:C^k(D_{\underline{\uC}}(3/m)\setminus\underline{\uC},\tilde{g})}\lesssim_k 1/m^3$.
	  				\item \label{item:juA} $\norm{\Acal\Tcal_{\exp(\omega)}\underline{j}[m]:C^k(D_{\underline{\uC}}(3/m)\setminus\underline{\uC},\tilde{g})}\lesssim_k 1/m^3$.  
	  			\end{enumerate}
	  		\end{lemma}
	  		\begin{proof}
	  			The proof is similar to \cite[lemma 5.14]{kapouleas:equator} and \cite[Lemma 2.23]{kapouleas:ii:mcgrath} by noticing that the functions inside the norms in \ref{item:phiu}-\ref{item:ju} have zero initial conditions and satisfy the ODE (recall \eqref{eq:jacobiHR2})
	  			\begin{equation*}
	  				\frac{\dd^2 u}{\dd \hat{r}^2}+\frac{1}{m\ru+\hat{r}}\frac{\dd u}{\dd \hat{r}}+\frac{1}{m^2}\left(1-\frac{(m\ru+\hat{r})^2}{16m^2}\right)u=\frac{E}{m^2},
	  			\end{equation*}
	  			where all the $C^k$ norms of $E$ is uniformly bounded in the domain. \ref{item:phiu}-\ref{item:ju} then follow by standard variation of constants. Similarly, the functions inside the norm in \ref{item:phiuA}-\ref{item:juA} has zero initial condition and satisfy the ODE
	  			\begin{equation*}
	  				\frac{\dd^2 \Acal u}{\dd \hat{r}^2}+\frac{1}{m\ru+\hat{r}}\frac{\dd \Acal u}{\dd \hat{r}}+\frac{1}{m^2}\left(1-\frac{(m\ru+\hat{r})^2}{16m^2}\right)\Acal u=-\frac{2m\ru}{m^2\ru^2-\hat{r}^2}\frac{\dd \Rcal u}{\dd \hat{r}}-\frac{\ru\hat{r}}{4m}\Rcal u.
	  			\end{equation*}
	  			The estimates in \ref{item:phiu} and \ref{item:ju} now control the right hand side with an extra factor $m^{-1}$. \ref{item:phiuA}-\ref{item:juA} then follow by standard variation of constants again.
	  		\end{proof}
	  	 \subsection*{Estimates on $\Phi^{0}$ and $\Phi^{1}$}
	  		Now we decompose $\Phi$. Our goal is to understand $\Phi''$ in \eqref{eq:Phiprime} around each singularity of $\Phi$. As stated above, we compare it with the average $\hat{\Phi}$ and estimate the error $\Phi'$. Notice that we remove the jump of $\phi=\Phi_{\avg}$ in the definition of $\hat{\Phi}$ \eqref{eq:Phihat} so that $\Phi'$ is smooth.

	  	\begin{definition}\label{def:decompose}
	  		We define $\hat{G}\in C^{\infty}_{\sym}(\R^2\setminus L)$ by
	  		\begin{equation}\label{eq:Ghat}
	  			\hat{G}:=\cutoff{2\delta}{3\delta}{\dist_{p}}(G_p-\underline{\phi}[\log\delta,0],0)
	  		\end{equation} 
	  		on $D_L(3\delta)$ (recall \ref{lem:green}) and $0$ otherwise. We define the rotationally invariant function $\hat{\Phi}\in C^{\infty}(\R^2)$ by (recall \ref{not:phihat})
	  		\begin{equation}\label{eq:Phihat}
	  			\hat{\Phi}:=\phi-\cutoff{2/m}{3/m}{\dist_{\underline{\uC}}}(\underline{j}[m/(2\ru)],0)=\phi_1\hat{\phi}.
	  		\end{equation}
	  		We then define the functions $\Phi'',\Phi'\in C^{\infty}_{\sym}(\R^2)$ and $E'\in C^{\infty}_{\sym}(\R^2)$  by
	  		\begin{equation}\label{eq:Phiprime}
	  			\Phi'':=\Phi-\hat{G},\quad \Phi':=\Phi-\hat{G}-\hat{\Phi},\quad E':=\Lcal\Phi'.
	  		\end{equation}
	  	\end{definition}
	
	  	\begin{lemma}\label{lem:GE}
	  		The following hold (recall \ref{eq:tildeg}).
	  		\begin{enumerate}[(i)]
	  		\item\label{item:Ghat} $\hat{G}$ is supported on $D_L (3\delta)$, and for $\delta'\in (0,\delta]$, $\norm{\hat{G}:C^k(\R^2\setminus D_L (\delta'),\rr,g)}\lesssim_k 1+\log(\delta/\delta')$, where $\rr:=\min\{\dist_L,\delta\}$.
	  		\item\label{item:GhatA} $\norm{\Acal\Tcal_{\exp(\omega)}\hat{G}:C^k(\R^2\setminus D_L (\delta),\tilde{g})}\lesssim_k 1/m$.
	  		\item\label{item:Eprime} $E'$ is supported on $D_{\uC} (3/m)\setminus D_L (2\delta)$, and $\norm{m^{-2}E':C^k(\R^2 ,\tilde{g})}\lesssim_k 1$.
	  		\item\label{item:EprimeA} $\norm{m^{-2}\Acal\Tcal_{\exp(\omega)}E':C^k(\R^2 ,\tilde{g})}\lesssim_k 1/m$.
	    	\end{enumerate}
	  	\end{lemma}
	  	\begin{proof}
	  		\ref{item:Ghat} follows by \ref{lem:green}\ref{item:greenest}, \eqref{eq:delta} and the definition \eqref{eq:Ghat}. By \eqref{eq:Phiprime} and \eqref{eq:Phihat} (recall \eqref{eq:tildeL}),
	  		\begin{equation*}
	  			m^{-2}E'=-\tilde{\Lcal}\hat{G}+\tilde{\Lcal}\cutoff{2/m}{3/m}{\dist_{\underline{\uC}}}(\underline{j}[m/(2\ru)],0).
	  		\end{equation*}
	  		The first term vanishes on $D_L(\delta)$, and on $\R^2\setminus  D_L(\delta)$ it is controlled by \ref{item:Ghat}. The second term vanishes on $\R^2\setminus D_{\underline{\uC}}^{\R^2}(3/m)$, and on $D_{\underline{\uC}}^{\R^2}(3/m)$, it is controlled by \ref{lem:phiju}\ref{item:ju}. \ref{item:Eprime} then follows. \ref{item:GhatA} follows as \cite[Lemma 4.9(ii)]{kapouleas:ii:mcgrath} or \cite[Lemma 8.23(ii)]{LDg} by the estimate $\norm{\Acal \dist_{p}^{g_{\shr}}:C^k(\supp\hat{G},\tilde{g})}\lesssim_k 1/m$, \ref{lem:green}\ref{item:greenest} and \ref{lem:phiju}\ref{item:phiuA}. \ref{item:EprimeA} follows as \cite[Lemma 5.21(iv)]{kapouleas:equator}, \cite[Lemma 4.14(ii)]{kapouleas:ii:mcgrath} or \cite[Lemma 8.23(iv)]{LDg} by \ref{item:GhatA} and \ref{lem:phiju}\ref{item:juA}.
	  	\end{proof}
	  	\begin{lemma}\label{lem:Phiosc}
	  		The following hold for some constant $C$ independent of $m$.
	  		\begin{enumerate}[(i)]
	  			\item\label{item:Phiosc} $\norm{\Tcal_{\exp(\omega)}\Phi_{\osc}':C^k(\R^2,e^{-C\hat{s} }),\tilde{g}}\lesssim_k 1$.
	  			\item\label{item:PhioscA} $\norm{\Acal\Tcal_{\exp(\omega)}\Phi_{\osc}':C^k(D_{\underline{\uC}}(3/m),\tilde{g})}\lesssim_k 1/m$.
	  			\item\label{item:PhioscCS} $\norm{\Phi_{\osc}':C^k(\R^2\setminus D_O(2\rbalanced),m^2r^{-C(m^2-1) },g)}\lesssim_k 1$.
	  		\end{enumerate}
	  	\end{lemma}
	  	\begin{proof}
	  		From the definitions \ref{def:decompose} (recall \eqref{eq:jacobiHR2})
	  		\begin{equation*}
	  			\Hcal \Tcal_{\exp(\omega)}\Phi'=\Tcal_{\exp(\omega)} E'
	  		\end{equation*}
	  		By \ref{def:LD}\ref{item:LDCS}, $\Phi'_{\osc}\in\mathcal{CS}^{2,\beta}(\R^2)$, which gives the behaviour of $\Phi'_{\osc}$ at infinity. \ref{item:Phiosc} now follows as \cite[Lemma 5.23(iii)]{kapouleas:equator} or \cite[Lemma 8.28(ii)]{LDg} along with \ref{lem:GE}\ref{item:Eprime} by noticing that $E'$ is supported in $D_{\underline{\uC}}(3/m)$. Also notice that when $r>4$, the potential term in \eqref{eq:jacobiHR2} has a preferred sign for applying the comparison of ODE. As \cite[Lemma 5.23(iv)]{kapouleas:equator}, \cite[Lemma 4.15(iii)]{kapouleas:ii:mcgrath} or \cite[Lemma 8.28(iii)]{LDg} $, u=\Tcal_{\exp(\omega)}\Phi'_{\osc}$ satisfies
	  		\begin{equation*}
	  			\left(\Delta_{\tilde{g}}+\frac{1}{m^2}\left(1-\frac{r^2}{16}\right) \right)\Acal u=	\frac{1}{m^2} \Acal \Tcal_{\exp(\omega)} E'_{\osc}-\frac{2m\ru}{m^2\ru^2-\hat{r}^2}\frac{\dd \Rcal u}{\dd \hat{r}}-\frac{\ru\hat{r}}{4m}\Rcal u,
	  		\end{equation*}
	  		with zero initial condition. \ref{item:PhioscA} then follows along with \ref{item:Phiosc} and \ref{lem:GE}\ref{item:EprimeA}.
	  		
	  		\ref{item:PhioscCS} follows similarly as \ref{item:Phiosc}. From the support of $E'$ \ref{lem:GE}\ref{item:Eprime}, when $r>\ru+3/m$ the $k$-th Fourier coefficients of $\Phi'_{\osc}$ satisfies the ODE \eqref{eq:jacobik}.
	  		As $\Phi'_{\osc}\in\mathcal{CS}^{2,\beta}(\R^2)$ and by \ref{lem:phimu}\ref{item:phimusing}, $\Phi'_{\osc,k}$ is a multiple of $U\left(-k^2+\frac{1}{2},1,\frac{r^2}{4}\right)$ and has asymptotics $r^{1-2k^2}$ with coefficients controlled by \ref{lem:GE}\ref{item:Eprime}. From the symmetry, $k$ must be a multiple of $m$, \ref{item:PhioscCS} then follows by the standard Fourier analysis.
	  	\end{proof}
	  	
	  	\begin{lemma}\label{lem:Phiprime}
	  		The following hold.
	  	\begin{enumerate}[(i)]
	  		\item\label{item:Phiprime} $\norm{\Tcal_{\exp(\omega)}\Phi':C^k(\R^2 ,\tilde{g})}\lesssim_k 1$.
	  		\item\label{item:PhiprimeA} $\norm{\Acal\Tcal_{\exp(\omega)}\Phi':C^k(D_{\underline{\uC}}(3/m),\tilde{g})}\lesssim_k 1/m$.
	  		\item\label{item:PhiprimeCS} $\norm{\Phi':C^k(\R^2\setminus D_O(2\rbalanced),m^2r^{-C(m^2-1) },g)}\lesssim_k 1$.
	  	\end{enumerate}
	  	\end{lemma}
	  	\begin{proof}
	  		By \ref{lem:Phiosc}\ref{item:Phiosc} and \ref{lem:Phiosc}\ref{item:PhioscA}, we only need to show the bound for $\Phi'_{avg}$. By \eqref{eq:Phiprime} and \eqref{eq:Phihat},
	  		\begin{equation*}
	  			\Phi'_{avg}=-\hat{G}_{\avg}+\cutoff{2/m}{3/m}{\dist_{\underline{\uC}}}(\underline{j}[m/(2\ru)],0),
	  		\end{equation*}
	  		which vanishes on $\R^2\setminus D_{\underline{\uC}}(3/m)$ completely. \ref{item:PhiprimeCS} now follows by \ref{lem:Phiosc}\ref{item:PhioscCS}. While on $D_{\underline{\uC}}(3/m)$, $\tilde{\Lcal}\Phi'_{avg}=m^{-2}E'_{\avg}$ provides the ODE
	  		\begin{equation*}
	  			\frac{\dd^2\Phi'_{avg}}{\dd\tilde{r}^2}-\left(\frac{1}{m\tilde{r}}-\frac{\tilde{r}}{2m^2}\right)\frac{\dd\Phi'_{avg}}{\dd\tilde{r}}+\frac{\Phi'_{avg}}{2m^2}=\frac{E'_{\avg}}{m^2}.
	  		\end{equation*}
	  		The estimate of $\Phi'_{\avg}$ then follows by the variation of parameters along with \ref{lem:GE}\ref{item:Eprime}. The estimate of $\Acal\Tcal_{\exp(\omega)}\Phi'_{\avg}$ follows by \ref{lem:phiju}\ref{item:juA} and \ref{lem:GE}\ref{item:GhatA}.
	  	\end{proof}
	  	  	\begin{lemma}\label{lem:Phinorm}
	  	  		The following hold for $\Phi$.
	  	  		\begin{enumerate}[(i)]
	  	  			\item\label{item:Phinorm} For $\delta'\in (0,\delta]$, $\rr:=\min\{\dist_L,\delta\}$, $\norm{\Phi:C^{k,\beta}(\R^2\setminus \sqcup_{q\in L}D_q^{\R^2}(\delta'),\rr,g)}\lesssim_k m+\log(\delta/\delta')$.
	  	  		\item\label{item:PhiCS} $\norm{\Phi:\mathcal{CS}^{2,\alpha}(\R^2\setminus D_O(2\rbalanced))}\lesssim m$ and
	  	  		\begin{equation}\label{eq:Phicone}
	  	  			\lim_{r\to\infty}\Phi/r=\sqrt{\pi}m\phi_{m}(\ru)e^{-\ru^2/8}/2.
	  	  		\end{equation}
	  	  		\end{enumerate}
	  	  			\end{lemma}
	  	  	\begin{proof}
	  	  		 \ref{item:Phinorm} follows by using \ref{lem:phiju}\ref{item:phiu}-\ref{item:ju}, \eqref{eq:Phiavg}, \eqref{eq:phiunum}, \ref{lem:GE}\ref{item:Ghat} and \ref{lem:Phiprime} to control the terms in the decomposition \ref{def:decompose}.
	  	  		 
	  	  		 By \ref{def:decompose}, on $\R^2\setminus D_O(2\rbalanced))$, $\Phi=\phi+\Phi'_{\osc}$.  \eqref{eq:Phicone} follows by \eqref{eq:Phiavg}, \eqref{eq:phiunum}, \ref{lem:Phiprime}\ref{item:PhiprimeCS} and \ref{lem:phimu}\ref{item:phimusing}. \ref{item:PhiCS} then follows by the definition \ref{def:conespaces}.
	  	  	\end{proof}
	  	\subsection*{Dislocation estimates on $\Phi^0$ and $\Phi^1$}
	  	 We have studied $\Phi^i$ around its singularities. However, the LD solutions of which the mismatches will be estimates are combinations of $\Phi^0$ and $\Phi^1$ (see \eqref{eq:varphizeta}, \eqref{eq:Phipm}, \eqref{eq:Lpm} and \eqref{eq:L}). In another word, we have to understand the $\Phi^0$ around the singularities of $\Phi^1$ and \emph{vice versa} (see the proof of \ref{lem:miss}, especially the estimates for $\mu'_{j,\pm}$). In this case, the estimates in \ref{lem:Phiprime} turn out to be insufficient. In this subsection, we treat this problem by the study in \cite[Section 9]{LDg}, where a similar situation happens in the doubling construction as the lack of symmetry. However, the case we encounter in this paper corresponds to the case $m_i=2m$ in \cite[Definition 9.1 and Lemma 9.16]{LDg}, which was not fully studied there. The general treatment of this problem was done in JZ's thesis  in \cite[Section 5.2.5]{zouthesis}, and we are in a much simpler situation here.
	  	 
	\begin{lemma}\label{lem:Gm}
		Consider $\R^2$ with coordinates $(\hat{s},\tilde{\theta})$ (recall \eqref{eq:tildetheta} and \ref{def:hatrs}), on any compact set of $\R_+\times\R \setminus\{(0,2\pi k):k\in\Z\}$ (recall \eqref{eq:Theta}, \eqref{eq:ThetaCyl} and  \eqref{eq:Phiavg}),
		\begin{equation*}
			\left(\Tcal_{\exp(\omega)}\left(\Phi-\underline{\phi}[\phi_1,m\hat{h}(\ru)]\right)\right)\circ\Theta\circ\Theta^{\Cylinder}(\hat{s}/m+\su,\tilde{\theta}/m) \to G_{\infty}(\hat{s},\tilde{\theta}),
		\end{equation*}
		smoothly when $m\to\infty$, where
		\begin{equation*}
			G_{\infty}(\hat{s},\tilde{\theta}):=\frac{1}{2}\log\left(4\sin^2\frac{\tilde{\theta}}{2}+4\sinh^2\frac{\hat{s}}{2}\right)=\frac{1}{2}\log\left(\sin^2\frac{\tilde{\theta}}{2}+\sinh^2\frac{\hat{s}}{2}\right)+\log 2,
		\end{equation*}
		is the Green's function on the cylinder.
	\end{lemma}
	\begin{proof}
		The same as \cite[Lemma 9.26]{LDg}. The function $-\Tcal_{\exp(\omega)}\left(\underline{\phi}[\phi_1,m\hat{h}(\ru)]\right)$ is the function $\phi_m$ there. Moreover, here $G_{\infty,\avg}(0,0)=\partial_+G_{\infty,\avg}(0,0)-\partial_-G_{\infty,\avg}(0,0)=0$ by \cite[remark 9.27]{LDg}.
	\end{proof}
	  	
	  	\begin{lemma}\label{lem:MPhi}
	  		For $i=0,1$, consider the four functions $\Phi^i=\Phi^i[L^i]$, $\tilde{\Phi}^i=\Phi^i[\tilde{L}^i]$ as in \ref{def:Phi}, where the four point sets $L^i:=L^i[\ru;m]$, $\tilde{L}^i:=L^i[\tilde{\ru};m]$ are as in \ref{def:L}. When $\Delta {\ru}:=m(\tilde{\ru}-\ru)$ is small enough in terms of absolute value, the following holds as $m\to \infty$, where $q\in L^{1}$, $\tilde{q}\in \tilde{L}^{1}$ when $i=0$ and $q\in L^{0}$, $\tilde{q}\in \tilde{L}^{0}$ when $i=1$.
	  		\begin{align*}
	  			\frac{1}{m}\frac{\partial_{\omega}\tilde{\Phi}^i}{\partial r}\left(\tilde{q}\right)=&\frac{1}{m}\frac{\partial_{\omega}\Phi^i}{\partial r}\left(q\right)-\frac{\Delta \ru}{\ru^2} \left(\frac{1}{4}+o(1)\right)+O(\Delta \ru^2).
	  		\end{align*}
	  		
	  	\end{lemma}
	  	\begin{proof}
	  		The term $-\Tcal_{\exp(\omega)}\left(\underline{\phi}[\phi_1,m\hat{h}(\ru)]\right)$ in \ref{lem:Gm} can be controlled by \ref{not:hhat} and \ref{def:phiju}. The conclusion the follows by \ref{lem:Gm} and an expansion from the calculations 
	  		\begin{align*}
	  			\frac{\partial G_{\infty}}{\partial \hat{s}}\left(0,\pi\right)=0,\quad
	  			\frac{\partial^2 G_{\infty}}{\partial \hat{s}^2}\left(0,\pi \right)=\frac{1}{4}. 	
	  		\end{align*}
	  		Notice that the conformal change \eqref{eq:ThetaCyl} provides extra coefficient $e^{-2\su}=\ru^{-2}$.
	  	\end{proof}
	  	
	  		\begin{lemma}\label{lem:Phihat}
	  	For the two functions $\Phi=\Phi^i$ and $\tilde{\Phi}=\tilde{\Phi}^i$ and $\Delta {\ru}:=m(\tilde{\ru}-\ru)$ small enough as in \ref{lem:MPhi}, the following hold on $\R^2$.
	  		\begin{align*}
	  		\abs{\hat{\tilde{\Phi}}-\hat{\Phi}}\lesssim \Delta {\ru}\cdot (r+1).
	  		\end{align*}
	  		where $\hat{\tilde{\Phi}}$ and $\hat{\Phi}$ is the corresponding (rotationally invariant) term in the decomposition \ref{def:decompose}.
	  	\end{lemma}
	  	\begin{proof}
	  		By \ref{def:decompose} and \eqref{eq:Phiavg}, when $r\in (\ru-4/m,\ru+4/m)$, it is equivalent to controlling the differences of the two functions
	  		\begin{equation*}
	  			\underline{\phi}[\phi_1[\ru],m\hat{h}(\ru);\ru],\quad \underline{j}[m/(2\ru);\ru],
	  		\end{equation*}
	  		when $\ru$ varies. By \ref{def:phiju} and \eqref{eq:phiunum} this follows by the continuous dependence of the solution of ODE to the initial conditions. For the other domain, this follows by \eqref{eq:phi} and the asymptotic of $\phi_m$ in \ref{lem:phimu}\ref{item:phimusing}.
	  	\end{proof}

	  	\section{MLD solutions}\label{S:MLD}
	  	\subsection*{Mismatch and obstruction spaces $\skernelv$, $\skernel$}
	  	
	  	\begin{definition}[Spaces of affine functions]\label{def:affine}
	  		Given $p\in\R^2$, let $\Vcal[p]\subset C^{\infty}(T_p\R^2)$ be the space of \emph{affine functions on} $T_p\R^2$. Given a function $v$ which is defined on a neighborhood of $p$ in $\R^2$ and is differentiable at $p$, we define $\Ecalu_pv:=e^{\omega(p)}(v(p)+\dd_{\omega}|_pv)\in\Vcal[p]$ (recall \ref{not:multi}\ref{item:df}). $\forall \kappau\in\Vcal[p]$, let $\kappau=\kappa^{\perp}+\kappa$ be the unique decomposition with $\kappa^{\perp}\in\R$ and $\kappa\in T^*_p\R^2$ and let $\abs{\kappau}:=\abs{\kappa^{\perp}}+\abs{\kappa}$. we also define $\Vcal[L]:=\oplus_{p\in L}\Vcal[p]$ for any finite set $L\in\R^2$.
	  	\end{definition}
	  	
	  	\begin{lemma}\label{lem:varphi}
	  		For the $\mathbf{D}_m$-symmetric LD solution $\varphi=\varphi[\tauu;\Lu;m]$ as in \ref{def:LD}, $\forall p\in L_{\pm}$ there exists a function $\hat{\varphi}_p\in C^{\infty}(D_p(2\delta))$ such that the following hold.
	  		\renewcommand{\theenumi}{\roman{enumi}}
	  		\begin{enumerate}
	  			\item \label{item:varphihat} $\varphi=\varphihat_p\pm\tau_{\pm} G_p$ on $D_p(2\delta)\setminus\{p\}$ (recall \ref{lem:green}\ref{item:greenp}).
	  			\item\label{item:Ep} $\Ecalu_p\varphihat_p:T_p\R^2\to \Vcal[p]$ is independent of the choice of $\delta$ and depends only on $\varphi$.
	  		\end{enumerate}	
	  	\end{lemma}
	  	\begin{proof}
	  		\ref{item:varphihat} follows from \ref{def:LD}\ref{item:LDlog} and \ref{lem:green}\ref{item:greenest} 
	  		and \ref{item:varphihat} then serves as the definition of $\varphihat_p$. \ref{item:Ep} follows then from \ref{item:varphihat} and \ref{lem:green}\ref{item:greenuni}. 
	  	\end{proof}
	  	
	  	\begin{definition}[Mismatch of LD solutions]\label{def:mismatch}
	  		For the $\mathbf{D}_m$-symmetry LD solution $\varphi=\varphi[\tauu;\Lu;m]$ as in \ref{lem:LD}, and a pair of two real numbers $\hu:=(h_+,h_-)$ (we use the convention $h_{\pm}=0$ when $L_{\pm}$ is empty), we define the \emph{mismatch of $\varphi$}, $\Mcal^{\hu}_{\Lu}\varphi\in\Vcal[\Lu]$, by $\Mcal^{\hu}_{\Lu}\varphi:=\oplus_{p\in L_+\sqcup L_-}\Mcal^{h_{\pm}}_p\varphi$, where $\Mcal^{h_{\pm}}_p\varphi\in\Vcal[p]$ for $p\in L_{\pm}$ is defined by (recall \eqref{eq:w})
	  		\begin{align}\label{eq:mismatch}
	  			\Mcal^{h_{\pm}}_p\varphi&:=\Ecalu_p\varphihat_p\pm\tau_{\pm}\log\frac{\tau_{\pm}}{2}-h_{\pm}.
	  		\end{align} 
	  	\end{definition}
	  	
	  	\begin{remark}
	  			From a Taylor expansion of $\varphihat_p$ combined with \ref{lem:green}\ref{item:greenest}, $\Mcal_p^{h_{\pm}}$ can be equivalently defined by requesting that for small $v\in T_p\R^2$
	  			\begin{equation*}
	  				\left(\Tcal_{\exp(\omega)}\varphi\right)\circ\exp_p^{g_{\shr}}(v)=\pm\tau_{\pm}\log\frac{2\abs{v}_{g_{\shr}|_p}}{\tau_{\pm}}+h_{\pm}+(\Mcal^{h_{\pm}}_p\varphi)(v)+O(\abs{v}^2\log\abs{v}),
	  			\end{equation*}
	  			or by \ref{lem:logconformal},
	  		\begin{equation*}
	  			\left(\Tcal_{\exp(\omega)}\varphi\right)\circ\exp_p^{g}(v)=\pm\tau_{\pm}\log\frac{2\abs{v}}{\tau_{\pm}}+h_{\pm}\pm\tau_{\pm}\frac{r^2(p)}{8}\mp\tau_{\pm}\frac{r(p)}{8}\dr(v)+(\Mcal^{h_{\pm}}_p\varphi)(v)+O(\abs{v}^2\log\abs{v}).
	  		\end{equation*}
	  		Notice that $\Tcal_{\exp(\omega)}\varphi$ is an LD solution with respect to the metric $g_{\shr}$ in the sense of \cite[Definition 3.10]{LDg}.
	  	\end{remark}
	  	
	  	\begin{definition}\label{def:VW}
	  		Given a $\mathbf{D}_m$-symmetric finite set $L\subset\uC[c]$ (recall \eqref{eq:circle}), We define the functions $V=V[L]$, $W=W[L]$, $V'=V'[L]$, $W'=W'[L]\in C^{\infty}_{\sym}(\R^2)$ by (recall \ref{def:phiju})
	  		\begin{align}\label{eq:VW}
	  			V[L]:=\cutoff{\delta}{2\delta}{\dist_{L}}(\underline{\phi}[1,0;c],0),\quad W[L]:=\Lcal V\\
	  			V'[L]:=\cutoff{\delta}{2\delta}{\dist_{L}}(\underline{\phi}[0,1;c],0),\quad W'[L]:=\Lcal V,\nonumber
	  		\end{align}
	  		on $D_{L}(2\delta)$ and $0$ otherwise.
	  		
	  		For a collection of $\mathbf{D}_m$-symmetric finite sets $L_i$ such that $L_i\subset\uC[c_i]$ and $L:=\sqcup_i L_i$, we then define spaces $\skernelv[L],  \skernel[L]\subset C^{\infty}_{\sym}(\R^2)$ by
	  		\begin{align*}
	  			\skernelv[L]:=\oplus_i\mathrm{span}\{V[L_i],V'[L_i]\},&\quad \skernel[L]:=\oplus_i\mathrm{span}\{W[L_i],W'[L_i]\}.
	  		\end{align*}
	  	\end{definition}
	  	\begin{lemma}[Obstruction spaces]\label{lem:obstr}
	  		The following hold for $\skernelv=\skernelv[L]$.
	  		\begin{enumerate}[(i)]
	  			\item\label{item:vsupp} The functions in $\skernelv$ are supported on $D_{L}(4\delta)$.
	  			\item\label{item:wsupp} The functions in $\skernel$ are supported on $\sqcup_{p\in L}D_p(4\delta)\setminus D_p(\delta/4)$.
	  			\item\label{item:V} For $V,V'\in \skernelv$, $\norm{V:C^k(\R^2 ,\tilde{g})}\lesssim_k 1$, $\norm{V':C^k(\R^2 ,\tilde{g})}\lesssim_k 1$.
	  		\end{enumerate}	
	  	\end{lemma}
	  	
	  	\begin{proof}
	  		\ref{item:vsupp}, \ref{item:wsupp} and \ref{item:V} follow by the definitions \ref{def:phiju} and  \ref{def:VW}. 
	  	\end{proof}
	  	\begin{lemma}\label{item:LDLpm}
	  		For the $\mathbf{D}_m$-symmetry LD solution $\varphi=\varphi[\tauu;\Lu;m]$ as in \ref{lem:LD}, suppose that either $L_+=L^0$ and $L_-=L^1$, or $L_-=L^0$ and $L_+=L^1$ for some $L^0=L^0[r_0]$, $L^1=L^1[r_1]$ as in \ref{def:L}, when $L_{\pm}$ are both non-empty; $L_{+}$ is the same as either $L^0$ or $L^1$ if $L_{-}$ is empty; and $L_{-}$ is the same as either $L^0$ or $L^1$ if $L_{+}$ is empty, then the following hold with $L=L_+\sqcup L_-$.
	  		\begin{enumerate}[(i)]
	  		\item\label{item:varphisym} $\oplus_{p\in L_+\sqcup L_-}\Ecalu_p C^{\infty}_{\sym}(\R^2)$ is a four-dimensional subspace of $\Vcal[L]=\Vcal[L_+]\oplus \Vcal[L_-]$ when $\L_{\pm}\neq \emptyset$; and a two-dimensional subspace of $\Vcal[L]$ when one of $L_{\pm}$ is empty. We denote this space by $\Vcal_{\sym}[L]$.
	  		\item\label{item:EL} For $\skernelv[L]=\skernelv[L_+]\oplus\skernelv[L_-]$, $\Ecal_{\Lu}:\skernelv[L]\to\Vcal_{\sym}[L]$ defined by $\Ecal_{\Lu}(v):=\sum_{p\in L_{\pm}}\Ecalu_p v$ is a linear isomorphism, and $\norm{\Ecal_{\Lu}^{-1}}\lesssim m^{2+\beta}$, where $\norm{\Ecal_{\Lu}^{-1}}$ is the operator norm of $\Ecal_{\Lu}^{-1}:\Vcal_{\sym}[L]\to\skernelv[L]$ with respect to the $C^{2,\beta}(\R^2,g)$ norm on the target and the maximum norm on the domain subject to the metric $g$ on $\R^2$.
	  	\end{enumerate}
	  	\end{lemma}
	  	\begin{proof}
	  		\ref{item:varphisym} follows from the symmetry imposed on $\varphi$. \ref{item:EL} follows from \ref{def:affine} and \ref{lem:obstr}\ref{item:V}.
	  	\end{proof}
	  	\subsection*{The family of LD solutions}
	  	We now start the construction of the LD solution. 
	  	
	  	Recall \ref{lem:LD}, in general, there are four parameters for each LD solution: the radii ($r_{\ell}$'s) of the circles where $L^0$ and $L^1$ lie, which will be the positions of catenoidal bridges; and the strengths ($\tau_{\ell}$'s) of the singularities respectively, which will be the waists of catenoidal bridges. 
	  	
	  	However, these parameters are not independent. We will connect the two adjacent levels of the graphs of the LD solutions by $m$ catenoidal bridges (see \ref{def:initial}). Therefore, the singularities of the adjacent levels have to relate each other: the $L^+$ of level $j+1$ and the $L^-$ of level $j$ not only have to be both either $L^0$ or $L^1$ (recall \ref{item:LDLpm}), but the circles where they lie have to be very close to each other. Actually, in this section, we will set them to be the same (see \eqref{eq:Lpm} and \eqref{eq:L}), and postpone this dislocation in the next section from the tilt $\kappa$ of catenoid (see \ref{def:catebridge}).
	  	
	  	Overall, the symmetry in \ref{N:G} then leads to the following space of symmetric parameters. The full space of parameters will be defined in \ref{def:Pcal}.
	  	\begin{definition}[The space of symmetric parameters]\label{def:zeta}
	  		We define \emph{the space of symmetric parameters} (recall \ref{not:Jm})
	  		\begin{equation*}
	  			\Pcal_s:=\begin{cases}
	  				&\R^{2J+1},\text{ when }J \text{ is a half integer},\\
	  				&\R^{2J},\text{ when }J \text{ is an integer}.
	  			\end{cases}
	  		\end{equation*}
	  		 The continuous parameters of the LD solutions are
	  		 \begin{equation*}
	  		 	\zetaboldu=(\zeta,\zetabold,\zeta',\zetabold'):=
	  		 	\begin{cases}
	  		 		&(\zeta,(\zeta_i)_{i=1/2}^{J-1},\zeta',(\zeta'_i)_{i=1/2}^{J-1})\in\BPcals,\text{ when }J \text{ is a half integer},\\
	  		 		&(\zeta,(\zeta_i)_{i=1}^{J-1},\zeta',(\zeta'_i)_{i=1}^{J-1})\in\BPcals,\text{ when }J \text{ is an integer},
	  		 	\end{cases}
	  		 \end{equation*} 
	  		  where 
	  		\begin{equation}\label{eq:zeta}
	  			\BPcals:=\{	\zetaboldu=(\zeta,\zetabold,\zeta',\zetabold')\in\Pcal_s:\abs{\zeta}+ \abs{\zetabold}+\abs{\zeta'}+ \abs{\zetabold'}\leq \cu\},
	  		\end{equation}
	  			where $\cu$ is a constant which can be taken as large as needed in terms of absolute constants, independent of $m$ fixed and will be fixed later (cf. the proof of \ref{thm}).
	  	\end{definition}
	  	
We now give the LD solutions. As explained above, they are determined by the radii of the circles $r_{\ell}$'s and the strengths $\tau_{\ell}$'s. The $r_{\ell}$'s all have to be very close to the radius $\rbalanced$ in \ref{lem:phimu}\ref{item:rbalanced}, which is the root of $\hat{h}$ (recall \ref{not:hhat}), so that the averages of LD solutions are all almost "balanced" at the jumps (the "horizontal balancing" in \cite{kapouleas:equator}). The dislocation analysis from \ref{lem:MPhi} leads to the expression \eqref{eq:r} (see the estimates of $\mu'_{j,\pm}$ in \ref{lem:miss}).

The motivation of the formul{\ae} \eqref{eq:tau} and \eqref{eq:h} for $\tau_{\ell}$'s and $h_{\ell}$'s is more complicated. The mismatch \eqref{eq:mismatch} relates both $\tau_{\ell}$'s and parameters $h_{\ell}$'s, which will correspond to the waists and heights of the catenoidal bridges. However, without putting extra dislocations (the $\kappa^{\perp}$ in \ref{def:catebridge}), both $\tau_{\ell}$'s and $h_{\ell}$'s will be determined from the system of equations by vanishing the mismatches in \eqref{eq:mismatch} (the "vertical balancing" in \cite{kapouleas:equator}). Indeed, if one writes such equations and linearises them in a certain way, the system will be reduced to a linear eigenvalue problem in \ref{lem:linearalg}. A similar system has been studied by Wiygul in \cite[(3.11)-(3.17)]{wiygul:stacking}. We omit this part here and check both balancings in \ref{lem:miss}. One can also compare \eqref{eq:tau} with \cite[(2.14),(3.12),(3.14),(3.16),Lemma 3.18]{wiygul:stacking} and \cite[(2.14)-(2.15)]{CSW}, and compare \eqref{eq:h} with \cite[(2.16)-(2.17)]{wiygul:stacking} and \cite[(2.17),(2.19)]{CSW}.

As stated before \ref{not:Jm}, $J$ is an analogy of \emph{total angular momentum quantum number}. We then use $j=-J,-J+1,\dots,J$ to label the $2J+1$ levels or LD solutions, as an analogy  of the \emph{total angular momentum projection quantum numbers}. As for a spin-$\frac{1}{2}$ system, we use $\ell=-J+1/2,-J+3/2,\dots,J-1/2$ to label the intervals between the levels, as an analogy of the possible \emph{magnetic quantum numbers}. 

\begin{definition}
	We define the \emph{signature function} $\sgn:\frac{1}{2}\Z\to \{\pm 1\}$ by
	\begin{equation*}
		\sgn(\ell):=
		\begin{cases}
			&\ell\mod 2,\text{ when }\ell\in\Z,\\
			&(\ell+1/2)\mod 2,\text{ when }\ell\notin\Z.
		\end{cases}
	\end{equation*}
\end{definition} 
\begin{remark}
	The signature function determines the positions of catenoidal bridges. When $\sgn(j+1/2)=i$ for $i=0,1$, we will put a catenoidal bridge connect the level $j$ and the level $j+1$ at points belong to some $L^i[r_{j+1/2}]$; similarly when $\sgn(j-1/2)=i$ for $i=0,1$, we will put a catenoidal bridge connecting level $j$ and level $j-1$ at points belong to some $L^i[r_{j-1/2}]$. This relation gives \eqref{eq:taupm}, \eqref{eq:Lpm} and \eqref{eq:hpm}.
	
	While the choice of the signature of an integer is natural, the choice of the signature of a half integer is arbitrary. If it were defined as $(\ell-1/2)\mod 2$, then the shrinker obtained would be the shrinker $\sigma_h(\breve{M})$ (recall \eqref{eq:sym} and \ref{def:grp}), where $\breve{M}$ is the shrinker constructed in \ref{thm}.
\end{remark}	  		
 	  		
\begin{definition}[LD solutions {$\varphi_j\bbracket{\zetaboldu}$}]\label{def:varphi}
For $\zetaboldu\in \BPcals$ as in \ref{def:zeta}, $\ell=-J+1/2,-J+3/2,\cdots,J-1/2$, we define the parameters $\tau_{\ell}$, $r_{\ell}$, $h_{\ell}$, $\phi_J$ and the finite point sets $L_{\ell}$ (recall \ref{def:L}, \eqref{eq:phiunum}, \ref{not:hhat} and \ref{def:VW}) by 
\begin{equation}\label{eq:r}
	r_{\ell}=r_{\ell}\bbracket{\zetaboldu}=r_{\ell}\bbracket{\zetaboldu;m}:=
		\hat{h}^{-1}(\zeta'/m)-4r_{J-1/2}^2\sum_{i=\abs{\ell}+1/2}^{J-1}\zeta'_{i}/m^2,
\end{equation}
\begin{equation}\label{eq:L}
	L_{\ell}=L_{\ell}\bbracket{\zetaboldu}=L_{\ell}\bbracket{\zetaboldu;m}:=L^{\sgn(\ell)}[r_{\ell}],
\end{equation}
\begin{equation}\label{eq:phiJ}
	\phi_J=	\phi_J\bbracket{\zetaboldu}=\phi_J\bbracket{\zetaboldu;m}:=\phi_1[r_{J-1/2}],
\end{equation}
\begin{align}\label{eq:tau}
	\tau_{\ell}=\tau_{\ell}\bbracket{\zetaboldu}=\tau_{\ell}\bbracket{\zetaboldu;m}:=\frac{\cos\frac{\pi \ell}{2J+1}}{m}\exp\left((\cos\frac{\pi }{2J+1}-1)\phi_J+\zeta+\frac{\sum_{0<i\leq\abs{\ell}-1/2}\zeta_i}{\phi_J}\right),
\end{align}	
\begin{equation}\label{eq:h}
	h_{\ell}=h_{\ell}\bbracket{\zetaboldu}=h_{\ell}\bbracket{\zetaboldu;m}:=
	\begin{cases}
			&\pm \tau_{J-3/2}\phi_J/2
		, \ell=\pm(J-1/2),\\
		&(\tau_{\ell-1}-\tau_{\ell+1})\phi_J/2
		, \ell=-J+3/2,\dots,J-3/2,
	\end{cases}
\end{equation}	 	  	
For $j=-J,-J+1,\cdots,J$,
we then define the LD solution $\varphi_j$ and the functions $\underline{v}_j\in\skernelv[L_j]$ by
using \ref{lem:LD}
\begin{equation}\label{eq:varphizeta}
	\varphi_j=\varphi_j\bbracket{\zetaboldu}=\varphi_j\bbracket{\zetaboldu;m}:=\varphi[\tauu_j;\Lu_j;m]=\tau_{j,+}\Phi_{j,+}-\tau_{j,-}\Phi_{j,-},\text{ where, }
\end{equation}
\begin{equation}\label{eq:Phipm}
	\Phi_{j,\pm}:=\varphi[1,0;L_{j,\pm},\emptyset;m],
\end{equation}
\begin{equation}\label{eq:taupm}
	\tauu_{j}=\tauu_{j}\bbracket{\zetaboldu}=\tauu_{j}\bbracket{\zetaboldu;m}=(\tau_{j,+}\bbracket{\zetaboldu;m},\tau_{j,-}\bbracket{\zetaboldu;m}):=
	\begin{cases}
		&(\tau_{J-1/2},0), j=J,\\
		&(\tau_{j-1/2},\tau_{j+1/2}), -J+1\leq j \leq J-1,\\
		&(0,\tau_{-J+1/2}), j=-J,
	\end{cases}
\end{equation}

\begin{equation}\label{eq:Lpm}
	\Lu_{j}=\Lu_{j}\bbracket{\zetaboldu}=	\Lu_{j}\bbracket{\zetaboldu;m}=(L_{j,+}\bbracket{\zetaboldu;m},L_{j,-}\bbracket{\zetaboldu;m}):=
	\begin{cases}
		&(L_{J-1/2},\emptyset), j=J,\\
		&(L_{j-1/2},L_{j+1/2}), -J+1\leq j \leq J-1,\\
		&(\emptyset,L_{-J+1/2}), j=-J,
	\end{cases}
\end{equation}
\begin{equation}\label{eq:vunder}
	 \underline{v}_j:= -\Ecal_{\Lu_j}^{-1}\Mcal^{\hu_j}_{\Lu_j}\varphi_j\in\skernelv[L_{j}],\text{ where }L_j:=L_{j,+}\sqcup L_{j,-}=L_{j+1/2}\sqcup L_{j-1/2}.
\end{equation}
\end{definition}

\begin{definition}\label{def:deltaprime}
	For $\tau_{j,\pm}$ as in \eqref{eq:varphizeta}, we define $\delta'_{j,\pm}$ by
	\begin{equation}\label{eq:deltaprime}
		\delta'_{j,\pm}:=\tau_{j,\pm}^{\alpha},
	\end{equation}
	where we now fix some small $\alpha\in(0,1)$, which we will assume is as small in absolute	terms as needed. We use the notation $\tau_{\min}:=\min_{\ell} \tau_{\ell}$, $\tau_{\max}:=\max_{\ell} \tau_{\ell}$ and $\deltamin':=\min_{j,\pm}\delta'_{j,\pm}=\tau_{\min}^{\alpha}$.
\end{definition}

	  	\begin{remark}\label{rk:taur}
	  		By \eqref{eq:phiunum}, \eqref{eq:zeta}, \eqref{eq:tau} and \eqref{eq:r}, for $\ell,\ell'\in\{-J+1/2,\dots,J-1/2\}$,
	  		\begin{equation*}
	  			\abs{\tau_{\ell}/\tau_{\ell'}-1}\lesssim \cu/m,\quad \abs{r_{\ell}-r_{\ell'}}\lesssim \cu/m^2.
	  		\end{equation*}
	  	\end{remark}
	  	\begin{lemma}[Matching equation and matching estimate for {$\varphi_j\bbracket{\zetaboldu}$}]\label{lem:miss}
	  		Given $\varphi_j=\varphi_j\bbracket{\zetaboldu}$ as in \ref{def:varphi} with $\zetaboldu\in\BPcals$, for (recall \ref{def:mismatch}) $\mu_{\pm,j}$,  $\mu_{\pm,j}'$ defined by (recall \eqref{eq:tau})
	  		\begin{equation}\label{eq:mu}
	  		\Vcal_{\sym}[\Lu_{j}]	\owns\Mcal^{\hu_j}_{\Lu_j}(\varphi_j)=
	  		\begin{cases}
	  			&\tau_{J,+}\left(\mu_{J,+}+\mu'_{J,+}\dr\right), j=J,\\
	  			&\Big(\tau_{j,+}\left(\mu_{j,+}+\mu'_{j,+}\dr\right),\\
	  			&\tau_{j,-}\left(\mu_{j,-}+\mu'_{j,-}\dr\right)\Big), 0 \leq j \leq J-1 ,
	  		\end{cases}
	  		\end{equation}
	  		where (recall \eqref{eq:h})
	  		\begin{equation}\label{eq:hpm}
	  			\hu_{j}=
	  			\begin{cases}
	  				&(h_{J-1/2},0), j=J,\\
	  				&(h_{j-1/2},h_{j+1/2}), -J+1\leq j \leq J-1,\\
	  				&(0,h_{-J+1/2}), j=-J.\\
	  			\end{cases}
	  		\end{equation}
	  		We then have up to a constant only depending on $N$, with $\zeta_{0}:=0, \zeta_{-\frac{1}{2}}:=-\zeta_{\frac{1}{2}}$, $\zeta_{J}:=0$,
	  		\begin{equation}\label{eq:muzeta}
	  	   	\mu_{j,\pm}=O(1)\pm\zeta+
	  	   	\begin{cases}
	  	   		&\mp\frac{\cos\frac{\pi (j+1\mp1/2)}{2J+1}\zeta_{j+1/2\mp 1/2}-\cos\frac{\pi(j-1\mp 1/2)}{2J+1}\zeta_{j-1/2\mp1/2}}{2\cos\frac{\pi(j\mp 1/2)}{2J+1}},
	  	   		1/2\leq j\leq J  
	  	   		,\\
	  	   		&\mp  \frac{\cos\frac{3\pi/2}{2J+1}}{2\cos\frac{\pi/2}{2J+1}}\zeta_{1}, j=0, J\in\N,
	  	   	\end{cases}
	  	   \end{equation}

	  		\begin{equation}\label{eq:muzetaprime}
	  			\mu'_{j,\pm}=O(1)\pm
	  			\begin{cases}
	  				 &\zeta', j=J\text{ for }+,\\
	  				 &\pm\zeta'_{j}, 1/2\leq j\leq J-1,\\
	  				 &0, J\in \N, j=0.
	  			\end{cases}
	  		\end{equation}
	  	\end{lemma}	
	  	\begin{proof}
	  		For any $p\in L_{j,\pm}$, by \ref{lem:varphi}\ref{item:varphihat}, \eqref{eq:Phiavg}, \eqref{eq:varphizeta} and \ref{def:decompose}, inside $D_p(2\delta)$ with $q\in L_{j,\mp}$
	  		\begin{align}\label{eq:hatphi}
	  			&\frac{\hat{\varphi}_{j,p}}{\tau_{j,\pm}}=\pm\underline{\phi}[\phi_1[r(p)]-\log\delta,m\hat{h}(r(p));r(p)]+\Phi'_{j,\pm}\\
	  			\mp&\frac{\tau_{j,\mp}}{\tau_{j,\pm}}
	  		 \left(\underline{\phi}[\phi_1[r(q)],m\hat{h}(r(q));r(q)]+\Phi'_{j,\mp})\right)\nonumber,
	  		\end{align}
	  		where $\Phi'_{j,\pm}$ denotes the corresponding decomposition in \ref{def:decompose} for $\Phi_{j,\pm}$ in \eqref{eq:Phipm}. By \ref{def:mismatch}, \ref{def:phiju}, \ref{lem:phiju}\ref{item:phiu}-\ref{item:phiub}, \ref{lem:Phiprime}\ref{item:Phiprime}, \eqref{eq:phiunum} and \ref{def:delta}, \ref{rk:taur}
	  		\begin{align*}
	  			\mu_{j,\pm}=\pm&\left(\phi_J-\log\delta+O\left(\frac{\cu}{m}\right)+O(1)\right)\pm\log\frac{\tau_{j,\pm}}{2}-\frac{h_{j,\pm}}{\tau_{j,\pm}}\\ \mp&\frac{\tau_{j,\mp}}{\tau_{j,\pm}}\left(\phi_J+O\left(\frac{\cu^2}{m^2}\right)+O\left(\frac{\cu}{m}\right)+O(1)\right)\\
	  		=\pm&\left(\phi_J+\log m\right)\mp\frac{\tau_{j,\mp}}{\tau_{j,\pm}}\phi_J\pm\log{\tau_{j,\pm}}-\frac{h_{j,\pm}}{\tau_{j,\pm}} +O(1).
	  	\end{align*}
	  	Now plug in \eqref{eq:taupm}, \eqref{eq:h} and \eqref{eq:hpm} for $\tau_{j,\pm}$ and $h_{j,\pm}$, when $j\neq J$ and $j\neq J-1$ for $-$
	  	\begin{align*}
	  		\mu_{j,\pm}
	  		=&\pm\left(\phi_J+\log m+\log{\tau_{j\mp\frac{1}{2}}}\right)\mp\frac{\tau_{j\pm\frac{1}{2}}}{\tau_{j\mp\frac{1}{2}}}\phi_J-\frac{\tau_{j-1\mp\frac{1}{2}}-\tau_{j+1\mp\frac{1}{2}}}{2\tau_{j\mp\frac{1}{2}}}\phi_J+O(1)\\
	  		=&\pm\left(\phi_J+\log m+\log{\tau_{j\mp\frac{1}{2}}}\right)\mp\frac{\tau_{j-1\mp\frac{1}{2}}+\tau_{j+1\mp\frac{1}{2}}}{2\tau_{j\mp\frac{1}{2}}}\phi_J+O(1),
	  	\end{align*}
	  	and when $j=J$ and $j= J-1$ for $-$
	  		\begin{align*}
	  		\mu_{J,+}
	  		=&\phi_J+\log m+\log{\tau_{J-\frac{1}{2}}}-\frac{\tau_{J-\frac{3}{2}}}{2\tau_{J-\frac{1}{2}}}\phi_J+O(1),\\
	  		\mu_{J-1,-}
	  		=&-\phi_J-\log m-\log{\tau_{J-\frac{3}{2}}}+\frac{\tau_{J-\frac{3}{2}}}{\tau_{J-\frac{1}{2}}}\phi_J-\frac{\tau_{J-\frac{3}{2}}}{2\tau_{J-\frac{1}{2}}}\phi_J+O(1)\\
	  		=&-\phi_J-\log m-\log{\tau_{J-\frac{3}{2}}}+\frac{\tau_{J-\frac{3}{2}}}{2\tau_{J-\frac{1}{2}}}\phi_J+O(1),
	  	\end{align*}
	  		Plug in \eqref{eq:tau} for $\tau_{\ell}$, \eqref{eq:muzeta} now follows by applying \ref{cor:linearalg}.
	  
	  	 Similarly, by \ref{def:mismatch}, \ref{def:phiju}, \ref{lem:Phiprime}\ref{item:PhiprimeA} and \ref{lem:MPhi}, \ref{rk:taur}, when $j\neq J$
	  	\begin{align*}
	  			\mu'_{j,\pm}=& \pm\left(m\hat{h}(r(p))+O\left(1\right)\right)
	  			 \mp\left(1+O\left(\frac{\cu}{m}\right)\right)\left(m\hat{h}(r(p))-\frac{m^2(r(q)-r(p))}{r(p)^2}\left(\frac{1}{4}+o(1)\right)+O\left(\frac{\cu^2}{m}\right)\right),
	  		\end{align*}
	  		and
	  		\begin{align*}
	  			\mu'_{J,+}=& m\hat{h}(r(p))+O\left(1\right).
	  		\end{align*}
	  Now plug in \eqref{eq:Lpm} for $r(p)$ and $r(q)$
	  		\begin{align*}
	  		\mu'_{j,\pm}=& \pm m^2(r_{j\pm 1/2}-r_{j\mp 1/2})/(4r_{j\mp 1/2}^2)+O(1),\text{ when }j\neq J,\\
	  		\mu'_{J,+}=&  m\hat{h}(r_{J-1/2})+O\left(1\right).
	  	\end{align*}	
	  		Plug in \eqref{eq:r} for $r_{\ell}$, \eqref{eq:muzetaprime} now follows.
	  	\end{proof}

	  	\begin{lemma}[Estimates for {$\varphi_j\bbracket{\zetaboldu}$}]\label{lem:varphiest}
	  		Given $\varphi_j\bbracket{\zetaboldu}$ and $\tau_{\ell}\bbracket{\zetaboldu}$ as in \ref{def:varphi}, the following hold when $m$ is large enough (recall \ref{def:deltaprime}).  
	  		\begin{enumerate}[(i)]
	  			\item\label{item:hatvarphi} $\forall p\in L_{j,\pm}$, $\delta^{-2}\norm{\hat{\varphi}_{j,p}+\underline{v}_j:C^{3,\beta}(\partial D_p(\delta),g)}\lesssim \tau_{\min}^{1-\alpha/9}$.
	  			\item\label{item:solboundn} $\norm{\varphi_j+\underline{v}_j:C^{k,\beta}(\R^2\setminus \sqcup_{\pm} D_{L_{j,\pm}}(\delta_{j,\pm}'),\rr,g)}\lesssim_k m\tau_{\min}$, where $\rr:=\min\{\dist_{L_j},\delta\}$.
	  				\item\label{item:solCS} $\norm{\varphi_j+\underline{v}_j:\mathcal{CS}^{2,\alpha}(\R^2\setminus D_O(2\rbalanced))}\lesssim m\tau_{\min}$ (recall \ref{def:conespaces}).
	  			\item\label{item:solembed} On $\R^2\setminus \sqcup_{\pm}D_{L_{j,\pm}}(\delta'_{j,\pm})$ we have
	  				\begin{align*}
	  					&\tau_{\max}+ h_{J-1/2}\hat{\phi}\leq \pm(\varphi_{\pm J}+\underline{v}_{\pm J}),\\
	  					&\tau_{\max}+h_{j-1/2}\hat{\phi}\leq \varphi_j+\underline{v}_j\leq -\tau_{\max}+h_{j+1/2}\hat{\phi}, J+1\leq j\leq J-1.\\
	  				\end{align*}
	  		\end{enumerate}
	  	\end{lemma}
	  	\begin{proof}
	  	From \eqref{eq:hatphi}, \eqref{eq:tau}, \ref{lem:phiju}\ref{item:phiu}-\ref{item:phiub}, \eqref{eq:phiunum} and \ref{lem:Phiprime}\ref{item:Phiprime}, the estimate in \ref{item:hatvarphi} is satisfied by $\hat{\varphi}_{j,p}$. Moreover, from \ref{lem:Phinorm}\ref{item:Phinorm}, \eqref{eq:tau} and \eqref{eq:varphizeta}, the estimate in \ref{item:solboundn} is satisfied by $\varphi_{j}$.
	  	
	  	On the other hand, by definition \eqref{eq:vunder}, on $D_{L_{\pm}}(4\delta)$,
	  	\begin{equation}\label{eq:vunderj}
	  		\underline{v}_j=-\tau_{j,\pm}\mu_{j,\pm} V[L_{j,\pm}]-\tau_{j,\pm}\mu'_{j,\pm}V'[L_{j,\pm}],
	  	\end{equation}
	  	and $\underline{v}_j=0$ on the other region by \ref{lem:obstr}\ref{item:vsupp}.
    	From \eqref{eq:tau}, \eqref{eq:mu}, \eqref{eq:zeta} and \ref{lem:obstr}\ref{item:V}, 
    	\begin{equation}\label{eq:vundernorm}
    		\norm{\underline{v}_j:C^k(\R^2,\tilde{g})}\lesssim_k \cu\tau_{\min}.
    	\end{equation}
    	The estimates in \ref{item:hatvarphi}-\ref{item:solboundn} are now also satisfied by $\underline{v}_j$.  \ref{item:solCS} follows by \eqref{eq:tau} and \ref{lem:Phinorm}\ref{item:PhiCS} as $\underline{v}_j=0$ on the domain. 
	  

	  Finally by using \eqref{eq:tau}, \ref{lem:GE}\ref{item:Ghat}, \ref{lem:Phiprime}\ref{item:Phiprime}, \ref{lem:Phiprime}\ref{item:PhiprimeCS} and \eqref{eq:vundernorm} to control the terms in \ref{def:decompose}, on $\R^2\setminus \sqcup_{\pm}D_{L_{j,\pm}}(\delta'_{j,\pm})$, we have
	  \begin{align*}
	  	\varphi_j+\underline{v}_j&=\tau_{j,+}(\hat{\Phi}_{j,+}+\hat{G}_{j,+}+\Phi'_{j,+})-\tau_{j,-}(\hat{\Phi}_{j,-}+\hat{G}_{j,-}+\Phi'_{j,-})+\underline{v}_j\\
	  	&=\tau_{j,+}\hat{\Phi}_{j,+}-\tau_{j,-}\hat{\Phi}_{j,-}+\tau_{\min}O(\alpha m+\cu),
	  \end{align*}
	  	where subscript ${j,\pm}$ denotes the corresponding decomposition in \ref{def:decompose} for $\Phi_{j,\pm}$ in \eqref{eq:Phipm}. By \ref{lem:Phihat} and \ref{rk:taur}, we can replace $\hat{\Phi}_{j,\pm}$ by $\hat{\Phi}_{J,+}$ with error controlled, and then by \eqref{eq:taupm} and \eqref{eq:h} when $j\neq \pm J$
	  	\begin{align*}
	  		\varphi_j+\underline{v}_j
	  		&=(\tau_{j-1/2}-\tau_{j+1/2})\hat{\Phi}_{J,+}+\tau_{\min}O(\alpha m+\cu)\\
	  		&=\left(\frac{h_{j+1/2}}{\phi_J}+\frac{\tau_{j-1/2}-2\tau_{j+1/2}+\tau_{j+3/2}}{2}\right)\hat{\Phi}_{J,+}+\tau_{\min}O(\alpha m+\cu)\\
	  		&=\left(\frac{h_{j-1/2}}{\phi_J}-\frac{\tau_{j-3/2}-2\tau_{j-1/2}+\tau_{j+1/2}}{2}\right)\hat{\Phi}_{J,+}+\tau_{\min}O(\alpha m+\cu)
	  	\end{align*}
	  	with $j\neq J-1$ for the second equation and $j\neq -J+1$ for the third equation. Similarly,
	  	\begin{align*}
	  		\varphi_J+\underline{v}_J
	  		&=\left(\frac{h_{J-1/2}}{\phi_J}-\frac{\tau_{J-3/2}-2\tau_{J-1/2}}{2}\right)\hat{\Phi}_{J,+}+\tau_{\min}O(\alpha m+\cu),\\
	  		\varphi_{J-1}+\underline{v}_{J-1}
	  		&=\left(\frac{h_{J-1/2}}{\phi_J}+\frac{\tau_{J-3/2}-2\tau_{J-1/2}}{2}\right)\hat{\Phi}_{J,+}+\tau_{\min}O(\alpha m+\cu).
	  	\end{align*}
	  	Plugging in \eqref{eq:tau} by using \eqref{eq:Phihat}, \eqref{eq:phiJ} and \ref{cor:linearalg} again,
	  	\begin{align*}
	  		\varphi_j+\underline{v}_j
	  		&=h_{j+1/2}\hat{\phi}+\left(\cos\frac{\pi}{2J+1}-1\right)\tau_{j+1/2}\hat{\Phi}_{J,+}+\tau_{\min}O(\alpha m+\cu)\\
	  		&=h_{j-1/2}\hat{\phi}-\left(\cos\frac{\pi}{2J+1}-1\right)\tau_{j-1/2}\hat{\Phi}_{J,+}+\tau_{\min}O(\alpha m+\cu),
	  	\end{align*}
	  	with $j\neq J$ for the first equation and $j\neq -J$ for the second equation. By \ref{def:decompose}, \ref{def:phiju} and \eqref{eq:phiunum}, it is easy to see $m\lesssim\hat{\Phi}$. \ref{item:solembed} then follows by taking $m$ big enough and $\alpha$ small enough and using \ref{rk:taur}. 
	\end{proof}

	  \section{The initial surfaces}
\label{S:init}
	  
	   \subsection*{Catenoidal bridges and their mean curvature}
	  \label{S:cat} 
	  As stated in the introduction, the catenoidal bridges in our construction will have four parameters: the position $r$, the waist $\tau$, the height $h$ and the tilt $\kappa$. The first two parameters have been fixed by \eqref{eq:r} and \eqref{eq:tau} in the last section. On the other hand, we can still input a parameter $\kappa^{\perp}$ for the height to \eqref{eq:h} and a parameter $\kappa$ for the tilt. This is similar to \cite[Section 2]{LDg}. The difference here is that we have bigger unbalancing parameters $\kappau=(\kappa,\kappa^{\perp})$.
	   \begin{definition}[Tilting rotations $\RRR_{\kappa}$, {\cite[Definition 1.11]{LDg}}]\label{def:tilt}
	  	Let $\kappa:E^2 \to (E^2 )^{\perp}$ be a linear map, where $E^2$ is a two-dimensional subspace of a three-dimensional Euclidean vector space $E^3$ and $(E^2 )^{\perp}$ denotes the orthogonal complement of $E^2$ in $E^3$. By choosing a unit normal vector to $E^2$ , we can identify $\kappa$ with an element of $(E^2 )^*$ . We define $\RRR_{\kappa}$ to be the rotation of $E^3$ characterized by $\RRR_{\kappa} (P ) = \Graph_P \kappa$ for $P \subset E^2$ a half-plane with $\partial P = \ker \kappa$ when $\kappa\neq 0$, or the identity $\Id_{E^3}$ when $\kappa=0$.
	  	
	  	Given also a function $u : \Omega \to \R$ on $\Omega \subset E^2$ such that $\RRR_{\kappa} (\Graph^{E^3}_{\Omega} u)$ is graphical over $E^2$ , we define $\mathrm{Tilt}_{\kappa}(u) : \Omega' \to \R$, with $\Omega'\subset E^2$ a shift of $\Omega$, by requesting $\RRR_{\kappa} (\Graph^{E^3}_{\Omega} u)=\Graph^{E^3}_{\Omega'} \mathrm{Tilt}_{\kappa}(u)$.
	  \end{definition}
	  
	  \begin{definition}\label{def:cyl}
	  	Given $p\in\R^2$, we choose for $(T_p\R^3,g_{\shr}|_p)$ Cartesian coordinates $(\tilde{x},\tilde{y},\tilde{z}):T_p\R^3\to\R^3$ satisfying $(\tilde{x},\tilde{y},\tilde{z})\circ\nu_{\R^2,g_{\shr}}(p)=(0,0,1)$. Let $\Cylinder:=\R\times\Sph^1$ be the standard cylinder, $\chi$ the standard product metric on it and $(s,\vartheta)$ be the standard coordinates on $\Cylinder$. We also define $\Cylinder[\tau,\varrho]$ by
	  	\begin{equation*}
	  		\Cylinder[\tau,\varrho]:=\{(s,\vartheta)\in\Cylinder:\tau\cosh s<\varrho\}.
	  	\end{equation*}
	  	
	  	Given $\tau\in\R_{+},h\in\R$, we define a catenoid $\K[p,\tau,h]\subset T_p\R^3$ of size $\tau$ and height $h$, and its parametrization $X_{\K}=X_{\K}[p,\tau,h]:\Cylinder\to\K[p,\tau,h]$ by taking
	  	\begin{align}\label{eq:rhozz}
	  		&\rho(s):=\tau\cosh s,\quad \zz(s):=\tau s+h\\
	  		&(\tilde{x},\tilde{y},\tilde{z})\circ X_{\K}(s,\vartheta)=(\rho(s) \cos\vartheta,\rho(s)\sin\vartheta,\zz(s)).\nonumber
	  	\end{align}
	  	Alternatively, 
	  	\begin{equation*}
	  		(\tilde{x},\tilde{y},\tilde{z})^{-1}\{(\rr\cos\vartheta,\rr\cos\vartheta,\varphi_{cat}(\rr):(\rr,\vartheta)\in [\tau,\infty)\times\R\}\subset T_p\R^3
	  	\end{equation*}
	  	is the part above the waist of $\K[p,\tau]$, where the function $\varphi_{cat}=\varphi_{cat}[\tau,h]:[\tau,\infty)\to\R$ is defined by
	  	\begin{equation*}
	  		\varphi_{cat}[\tau,h]:=\tau\arcosh\frac{\rr}{\tau}+h. 
	  	\end{equation*}
	  \end{definition}

	  \begin{definition}[Tilted catenoidal bridges in $T_p\R^3$ and catenoidal bridges in $\R^3$]\label{def:catebridge}
	  	Given $p\in \R^2$,  $x \in [0,4]$ (where
	  	$x$ may be omitted when $x = 0$) $\tau>0$ and $\kappau=\kappa^{\perp}+\kappa\in\Vcal[p]$, we define an elevated and tilted by $\kappau$ \emph{model catenoid} in $T_p\R^3$ of size $\tau$ and a corresponding \emph{catenoidal bridge} in $\R^3$ as the following in the notation of \ref{def:cyl}. where $b$ is a large constant to be chosen later independently of the $\tau$ and $\kappau$ parameters (see the proof of \ref{prop:lineareq}).
	  	\begin{align*}
	  		&\K[p,\tau,h,\kappau]:=X_{\K}[p,\tau,h,\kappau](\Cylinder)\subset T_p\R^3,\\
	  		\text{where (recall \ref{def:tilt}) }&X_{\K}[p,\tau,h,\kappau]:=\RRR_{\tau\kappa}\circ X_{\K}[p,\tau,h]+\tau\kappa^{\perp}\nu_{\Sigma}(p):\Cylinder\to T_p\R^3,\\
	  		&\Kcech_x[p,\tau,h,\kappau]:=X_{\Kcech}[p,\tau,h,\kappau](\Cylinder[\tau,6\tau^{\alpha}/(1+x)])\subset \R^3,\\
	  		&\Ku_x[p,\tau,h,\kappau]:=X_{\Kcech}[p,\tau,h,\kappau](\Cylinder[\tau,b(1+x)\tau])\subset \R^3,\\
	  		\text{where }&X_{\Kcech}[p,\tau,h,\kappau]:=\exp_p^{\R^2,\R^3,g_{\shr}}\circ X_{\K}[p,\tau,h,\kappau]:\Cylinder\to \R^3.
	  	\end{align*}	 
	  	Finally, using the above maps, we take the coordinates $(s,\vartheta)$ on the cylinder as in \ref{def:cyl} to be coordinates on $\K[p,\tau,h,\kappau]$ and $\Kcech[p,\tau,h,\kappau]$.
	  \end{definition}
	  
	   \begin{definition}[Graphs of tilted catenoidal bridges]
	  	Given $\kappau\in\Vcal[p]$, we define $\varphi^{\pm}_{cat}[p,\tau,h,\kappau]:\Omega'\to\R$ smoothly depending on the parameters by using \ref{lem:graph} and requiring that (recall \ref{not:manifold})
	  	\begin{equation*}
	  		\Graph_{\Omega'}^{\R^3,g}(\varphi^{\pm}_{cat}[p,\tau,h,\kappau] )\subset\Graph_{\Omega}^{\R^3,g_{\shr}}\left(\pm\mathrm{Tilt}_{\pm\tau\kappa}(\varphi_{cat}[\tau,\pm h]\circ\dist_O)+\tau\kappa^{\perp}\right)\subset \Kcech[p,\tau,h,\kappau]
	  	\end{equation*}
	  	 and that $\varphi^{\pm}_{cat}[p,\tau,0,0]$ is always positive, where (only here and the lemma following) we use $\Omega:=D_p^{\R^2,g}(5\tau^{\alpha})\setminus D_p^{\R^2,g}(2\tau)$ and $\Omega':=D_p^{\R^2,g}(4\tau^{\alpha})\setminus D_p^{\R^2,g}(3\tau)$. Equivalently,
	  	\begin{equation*}
	  		(x_1,x_2,\varphi^{\pm}_{cat}[p,\tau,h,\kappau])\in \Kcech[p,\tau,h,\kappau] \text{ for }p=(x_1,x_2,0)\in\Omega'.
	  	\end{equation*}
	  \end{definition}
	  \begin{lemma}[Tilted catenoid asymptotics]\label{lem:varphicat}
	  	For $\rr:=\dist^{g_{\shr}}_{p}$, assuming $-\log\tau\sim m$, $h\sim m\tau $, $\kappau\sim 1$, the following hold.
	  	\begin{equation*}
	  		\norm{\Tcal_{\exp(\omega)}\varphi^{\pm}_{cat}[p,\tau,h,\kappau]\mp \tau\log\frac{2\rr}{\tau}-h-\tau\kappau\circ(\exp^{\R^2,g_{\shr}}_p)^{-1}:C^k(D_p(4\tau^{\alpha})\setminus D_p(3\tau),\rr,g,\rr^{-2})}\lesssim_k m^3\tau^3.
	  	\end{equation*}
	  \end{lemma}
	  \begin{proof}
	  	As in \cite[Lemma 2.12, remark 2.15]{LDg}, $u^{\pm}=\pm\mathrm{Tilt}_{\pm\tau\kappa}(\varphi_{cat}[\tau,\pm h]\circ\dist_O)+\tau\kappa^{\perp}$ satisfies the estimate
	  	\begin{equation*}
	  		\norm{u^{\pm}\mp \tau\log\frac{2\dist_O^{g_{\shr}|_p}}{\tau}- h-\tau\kappau:C^k(D^{T_p\R^2}_p(4\tau^{\alpha})\setminus D^{T_p\R^2}_p(2\tau),\dist_O^{g_{\shr}|_p},g_{\shr}|_p,(\dist_O^{g_{\shr}|_p})^{-2})}\lesssim_k \tau^3.
	  	\end{equation*}
	  	The estimate then follows by \ref{lem:graph}\ref{item:fprime}.
	  \end{proof}
	  \begin{lemma}
	  	For $\Kcech=\Kcech[p,\tau,h,\kappau]$, $\norm{\rho^2H:C^k(\Kcech,\chi,\rho^2(\abs{\zz}+\tau)) }\lesssim_k 1$.
	  \end{lemma}
	  \begin{proof}
	  	Notice that $\hat{g}:=X_{\Kcech}^*g_{\shr}$ on $T_p\R^3$ is conformal with $g_{\shr}|_p$, the quantity $\beta$ in \cite[Lemma C.9]{LDg} vanishes by \cite{LDg}[remark C.6]. The result then follows by \cite[Lemma C.9]{LDg} and \cite[Lemma 2.28]{LDg} along with the fact that $\R^2$ is totally geodesic in $(\R^3,g_{\shr})$.
	  \end{proof}
	  \subsection*{The initial surfaces and their regions}
	 \begin{definition}[The space of parameters]\label{def:Pcal}
	 		We define the space of (total) parameters (recall \ref{not:Jm}, \ref{item:LDLpm}\ref{item:varphisym} and \ref{def:zeta})
	 	\begin{equation*}
	 		\Pcal:=\Pcal_s\oplus\oplus_{\ell>0}^{J-1/2}\Vcal_{\sym}[L_\ell]\cong \R^{4J}.
	 	\end{equation*}
	 	We also define the space $\BPcal$ by (recall \eqref{eq:zeta})
	 	\begin{equation*}
	 		\BPcal:=\{(\zetaboldu,\kappaboldu=\{\kappau_{\ell}\}^{J-1/2}_{\ell>0})\in\Pcal:\abs{\zetaboldu}\leq \cu, \abs{\kappaboldu}\leq \cu\}.
	 	\end{equation*}
	 \end{definition}
	    \begin{notation}\label{not:kappa}
	  	From \ref{def:tilt}, when $\ell>0$, we can represent $\kappau_{\ell}$ by $(\kappa_{\ell},\kappa^{\perp}_{\ell})\in\R^2$. We then use the notation $\kappabold=\{\kappau_{\ell}\}_{\ell>0}$, $\kappabold^{\perp}=\{\kappau_{\ell}^{\perp}\}_{\ell>0}$.  We can also extend the definition of $\kappau_{\ell}$ to all $\ell=-J+1/2,-J+3/2,\dots,J-1/2$ by $\kappau_{\ell}:=-\kappau_{-\ell}$ when $\ell<0$ and $\kappau_0:=(0,0)$ when $J$ is a half integer. We also use the notation $	\kappau_j:=(\tau_{j,+}\kappau_{j,+},\tau_{j,-}\kappau_{j,-})$ and (cf., \ref{eq:taupm} and \ref{eq:hpm}) 
	  	\begin{equation*}
	  		(\kappau_{j,+},\kappau_{j,-}):=
	  		\begin{cases}
	  			&(\kappau_{J-1/2},0), j=J,\\
	  			&(\kappau_{j-1/2},\kappau_{j+1/2}), -J+1\leq j \leq J-1,\\
	  			&(0,\kappau_{-J+1/2}), j=-J.
	  		\end{cases}
	  	\end{equation*}
	  	Similarly, we can represent $\kappau_{j,\pm}$ by $(\kappa_{j,\pm},\kappa^{\perp}_{j,\pm})$. 
	  \end{notation}
	  \begin{definition}[Initial surfaces]\label{def:initial}
	  	Given $(\zetaboldu,\kappaboldu)\in \BPcal$, $\varphi=\varphi_j\bbracket{\zetaboldu}$ and $\underline{v}_j$ as in \ref{def:varphi},
	  	we define the smooth initial surface (recall \ref{not:manifold}\ref{item:graph} and \ref{def:catebridge})
	  	\begin{equation*}
	  		M=M[\zetaboldu,\kappaboldu]:= \sqcup_j\Graph_{\Omega_j}^{\R^3,\R^2,g}(\varphi^{gl}_j)\cup\sqcup_{\ell} \sqcup_{p\in L_{\ell}}\Kcech[p,\tau_{\ell},h_{\ell},\kappau_{\ell}],
	  	\end{equation*}
	  	where $\Omega_j:=\R^2\setminus \sqcup_{\pm} D_{L_{j,\pm}}(9\tau_{j,\pm})$ and the functions $\varphi^{gl}_j:=\varphi^{gl}_j\bbracket{\zetaboldu,\kappaboldu}:\Omega_j\to\R$ are defined by $\varphi_{j}+\underline{v}_{j}+\Ecal^{-1}_{\Lu_j}\kappau_j$ (recall \ref{item:LDLpm}\ref{item:EL}) on $\R^2\setminus  \sqcup_{\pm} D_{L_{j,\pm}}(3\delta'_{j,\pm})$ and on $\sqcup_{p\in L_{j,\pm}} D_p(3\delta'_{j,\pm})\setminus D_p(9\tau_{j,\pm})$
	  	\begin{equation}\label{eq:phigl}
	  		\varphi^{gl}_j:=
	  			\Psi[2\delta'_{j,\pm},3\delta'_{j,\pm};\dist_p](\varphi^{\pm}_{cat}[p,\tau_{j,\pm},h_{j,\pm},\kappau_{j,\pm}],\varphi_{j}+\underline{v}_{j}+\Ecal^{-1}_{\Lu_j}\kappau_j).
	  	\end{equation}
	  \end{definition}
	
	  \begin{lemma}[The gluing region]\label{lem:glue}
	  	For $M=M[\zetaboldu,\kappaboldu]$ as in \ref{def:initial}, $g=g_{Euc}$ and $\forall p\in L_{j,\pm}$, the following hold.
	  	\begin{enumerate}[(i)]
	  		\item\label{item:gluecat} $\norm{\varphi^{gl}_j-\varphi^{\pm}_{cat}[p,\tau_{j,\pm},h_{j,\pm},\kappa_{j,\pm}]:C^{2,\beta}(D_p(4\delta'_{j,\pm})\setminus D_p(\delta'_{j,\pm}),\delta'_{j,\pm},g)}\lesssim \tau_{j,\pm}^{1+15\alpha/8}$.
	  		\item\label{item:gluevarphi} $\norm{\varphi^{gl}_j:C^{3,\beta}(D_p(4\delta'_{j,\pm})\setminus D_p(\delta'_{j,\pm}),\delta'_{j,\pm},g)}\lesssim m\tau_{j,\pm}$.
	  		\item\label{item:glueH} $\norm{(\delta'_{j,\pm})^2 (H^{2\omega})':C^{2,\beta}(D_p(3\delta'_{j,\pm})\setminus D_p(2\delta'_{j,\pm}),\delta'_{j,\pm},g)}\lesssim \tau_{j,\pm}^{1+15\alpha/8}$, where $(H^{2\omega})'=H^{2\omega}_{\Omega_j}\circ (\Pi_{\R^2})^{-1}$ (recall \ref{not:Fcal}).
	  	\end{enumerate}
	  \end{lemma}
	  \begin{proof}
	  	In this proof we use the notation omitting the subscript $j$.
	  	
	  	We have  for each $p \in L_{\pm}$ on $D_p(4\delta'_{\pm})\setminus D_p(\delta'_{\pm})$ and $\underline{\phi}[1,0]=\underline{\phi}[1,0;r(p)]$ (recall \ref{def:phiju})
	  	\begin{align}\label{eq:varphiplusminus}
	  		&\varphi^{gl}=\pm\tau_{\pm} G_p+\left(\mp\tau_{\pm}\log\frac{\tau_{\pm}}{2}+h_{\pm}\right) \underline{\phi}[1,0]+\Ecal^{-1}_{\Lu}\kappau+\cutoff{2\delta'_{\pm}}{3\delta'_{\pm}}{\dist_p}(\underline{\varphi}_{\pm},\bar{\varphi}_{\pm}),\nonumber\\
	  	\text{where }&\underline{\varphi}_{\pm}:=\varphi^{\pm}_{cat}[p,\tau_{\pm},h_{\pm},\kappau_{\pm}]\mp\tau_{\pm} G_p+\left(\pm\tau_{\pm}\log\frac{\tau_{\pm}}{2}-h_{\pm}\right) \underline{\phi}[1,0]-\Ecal^{-1}_{\Lu}\kappau\\
	  	&\bar{\varphi}_{\pm}:=\hat{\varphi}_p+\left(\pm\tau_{\pm}\log\frac{\tau_{\pm}}{2}-h_{\pm}\right)\underline{\phi}[1,0]+\underline{v}_{\pm}.\nonumber
	  	\end{align} 
	  	Thus by \ref{lem:green}\ref{item:greenest}, \eqref{eq:tau} and \eqref{eq:h}, on $D_p(4\delta'_{\pm})\setminus D_p(\delta'_{\pm})$,
	  	\begin{align}\label{eq:phiglestn}
	  		\norm{\varphi^{gl}}\lesssim m\tau_{\pm}+\norm{\underline{\varphi}_{\pm}}+\norm{\bar{\varphi}_{\pm}},
	  	\end{align}
	  	 where in this proof when we do not specify the norm, we mean the $C^{3,\beta}(D_p(4\delta'_{\pm})\setminus D_p(\delta'_{\pm}),\delta'_{\pm},g)$ unless specified otherwise.
	  	
	  	Note that on $D_p(4\delta'_{\pm})\setminus D_p(\delta'_{\pm})$, 
	  	 \begin{align*}
	  	 	\varphi^{gl}-\varphi^{\pm}_{cat}[p,\tau_{\pm},h_{\pm},\kappau_{\pm}]&=\cutoff{2\delta'_{\pm}}{3\delta'_{\pm}}{\dist_p}(0,\bar{\varphi}_{\pm}-\underline{\varphi}_{\pm}),\\
	  	 	\Lcal\varphi^{gl}&=\Lcal\cutoff{2\delta'_{\pm}}{3\delta'_{\pm}}{\dist_p}(\underline{\varphi}_{\pm},\bar{\varphi}_{\pm}).
	  	 \end{align*}
	  	 Using these, we have 
	  	 \begin{align}
	  	 	\norm{\varphi^{gl}-\varphi^{\pm}_{cat}[p,\tau_{\pm},h_{\pm},\kappau_{\pm}]}&\lesssim \norm{\underline{\varphi}_{\pm}}+\norm{\bar{\varphi}_{\pm}},\label{eq:phiglcat}\\
	  	 	\norm{(\delta'_{\pm})^2\Lcal\varphi^{gl}:C^{0,\beta}(D_p(4\delta'_{\pm})\setminus D_p(\delta'_{\pm}),\delta'_{\pm},g)}&\lesssim \norm{\underline{\varphi}_{\pm}}+\norm{\bar{\varphi}_{\pm}}.\label{eq:Lphigl}
	  	 \end{align}

	  	  By  \eqref{eq:varphiplusminus}, 
	  	 \begin{align*}
	  	 	&\norm{\underline{\varphi}_{\pm}}\lesssim	\norm{\Tcal_{\exp(\omega)}\underline{\varphi}_{\pm}} \lesssim \norm{\tau_{\pm}\Tcal_{\exp(\omega)} G_p-\tau_{\pm}\log\dist_p^{g_{\shr}}}\\
	  	 	+&\norm{\Tcal_{\exp(\omega)}\varphi^{\pm}_{cat}[p,\tau_{\pm},h_{\pm},\kappau_{\pm}]\mp\tau_{\pm}\log\frac{2\dist^{g_{\shr}}_p}{\tau_{\pm}}-h_{\pm}-\kappau_{\pm}\circ(\exp^{\R^2,g_{\shr}}_p)^{-1}}\nonumber\\
	  	 	+&\norm{\left(\pm\tau_{\pm}\log\frac{\tau_{\pm}}{2}-h_{\pm}-\kappa_{\pm}^{\perp}\right)(1-\Tcal_{\exp(\omega)}\underline{\phi}[1,0])}+\norm{\kappa_{\pm}\circ(\exp^{\R^2,g_{\shr}}_p)^{-1}-\kappa_{\pm}\Tcal_{\exp(\omega)}\underline{\phi}[0,1]}\nonumber.
	  	 \end{align*}
	  	 Thus by using \ref{lem:green}\ref{item:greenest}, \ref{lem:varphicat}, \ref{lem:phiju}\ref{item:phiu}-\ref{item:phiub},
	  	 \ref{def:VW} to control each terms and using \eqref{not:kappa}, \eqref{eq:tau}, \eqref{eq:h}, \eqref{eq:deltaprime} to control the parameters
	  	 \begin{equation}\label{eq:phiminusestn}
	  	 	\norm{\underline{\varphi}_{\pm}}\lesssim \tau_{\pm}\delta'^2_{\pm}\abs{\log\delta'_{\pm}}+m^3\tau_{\pm}^3(\delta'_{\pm})^{-2}+ \tau_{\pm}(\delta'^2_{\pm})\abs{\log\tau_{\pm}}+\tau_{\pm}(\delta'_{\pm})^2\lesssim m\tau_{\pm}^{1+2\alpha}+m^3\tau_{\pm}^{2-2\alpha}.
	  	 \end{equation}
	  	 
	  	 Because $\Lcal\bar{\varphi}_{\pm} = 0$ on $D_p(4\delta'_{\pm})\setminus D_p(\delta'_{\pm})$ and $\Ecalu_{p}\bar{\varphi}_{\pm}=0$ (recall \ref{def:mismatch} and \eqref{eq:varphiplusminus}), it follows from standard linear theory and \ref{lem:varphiest}\ref{item:hatvarphi} that
	  	 \begin{equation}\label{eq:phiplusestn}
	  	 	\norm{\bar{\varphi}_{\pm}}\lesssim\left(\frac{\delta'_{\pm}}{\delta}\right)^2\norm{\bar{\varphi}_{\pm}:C^{3,\beta}(\partial D_p^{\R^2}(\delta),\delta,g)}\lesssim  (\delta'_{\pm})^{2}\tau_{\pm}^{1-\alpha/9}\lesssim\tau_{\pm}^{1+17\alpha/9}. 
	  	 \end{equation}
	  	  \ref{item:gluecat} now follows by \eqref{eq:phiglcat}, \eqref{eq:phiminusestn} and \eqref{eq:phiplusestn}, and \ref{item:gluevarphi} now follows by \eqref{eq:phiglestn}, \eqref{eq:phiminusestn} and \eqref{eq:phiplusestn}.

	  	  By expanding $(H^{2\omega})'$ in linear and higher order terms (by using $H^{2\omega}_{\Omega_j}=e^{\omega}H_{\Omega_j}^{g_{\shr}}$ as in \ref{not:Fcal} or as in e.g., \cite[(10.2)]{kapouleas:kleene:moller} and \cite[Lemma 2.4]{Wang:JAMS}) we have
	  	 \begin{equation*}
	  	 	(\delta'_{\pm})^2 (H^{2\omega})'=	(\delta'_{\pm})^2\Lcal\varphi^{gl}+\delta'_{\pm}\tilde{Q}_{(\delta'_{\pm})^{-1}\varphi^{gl}}.
	  	 \end{equation*}
	  	 Thus 
	  	 \begin{equation} 
\label{eq:thus} 
	  	 	\norm{\delta'_{\pm}\tilde{Q}_{(\delta'_{\pm})^{-1}\varphi^{gl}}:C^{0,\beta}(D_p(4\delta'_{\pm})\setminus D_p(\delta'_{\pm}),\delta'_{\pm},g)}
	  	 	\lesssim(\delta'_{\pm})^{-2}\norm{\varphi^{gl}}^2.
	  	 \end{equation}
	  	  \ref{item:glueH} follows along with \eqref{eq:Lphigl}.
	  \end{proof}
	  
	  \begin{lemma}\label{lem:phigl}
	  	$M$ defined in \ref{def:initial} is a $\Grp[m,J]$-invariant embedded surface (recall \ref{def:grp}) and moreover the following hold.
	  	\begin{enumerate}[(i)]
	  		\item\label{item:phiglembed} On $\R^2\setminus \sqcup_{\pm}D_{L_{j,\pm}}(\delta'_{j,\pm})$ we have 
	  		\begin{align*}
	  			&\frac{8}{9}\tau_{\max}+ h_{J-1/2}\hat{\phi}\leq \pm\varphi^{gl}_{\pm J},\\
	  			&\frac{8}{9}\tau_{\max}+h_{j-1/2}\hat{\phi}\leq \varphi^{gl}_j \leq -\frac{8}{9}\tau_{\max}+h_{j+1/2}\hat{\phi},\quad J+1\leq j\leq J-1.\\
	  		\end{align*}
	  		\item\label{item:phiglbound} $\norm{\varphi^{gl}_j:C^{3,\beta}(\R^2\setminus \sqcup_{\pm}D_{L_{j,\pm}}(\delta'_{j,\pm}),g)}\lesssim \frac{9}{8}\tau_{\min}^{8/9}$.
	  		\item\label{item:phiglCS} $\norm{\varphi^{gl}_j:\mathcal{CS}^{2,\alpha}(\R^2\setminus D_O(2\rbalanced))}\lesssim m\tau_{\min}$ (recall \ref{def:conespaces}).
	  			\item\label{item:phiglcone}
	  			\begin{equation*}
	  				\lim_{r\to\infty}\varphi^{gl}_j/r=
	  				\begin{cases}
	  					\pm\frac{\sqrt{\pi}m}{2}{\tau}_{J- 1/2}\phi_m[{r}_{J- 1/2}]e^{{r}_{J- 1/2}^2/8},\text{ when }j=\pm J,\\
	  					\frac{\sqrt{\pi}m}{2}\left({\tau}_{j- 1/2}\phi_m[{r}_{j- 1/2}]e^{{r}_{j- 1/2}^2/8}-{\tau}_{j+ 1/2}\phi_m[{r}_{j+ 1/2}]e^{{r}_{j+ 1/2}^2/8}\right),\text{ when }j\neq\pm J,0,\\
	  					0\text{ when }j=0 \text{ and }J\in\N.
	  				\end{cases}
	  			\end{equation*}
	  	\end{enumerate}
	  \end{lemma}
	  \begin{proof}
	  	The symmetry follows by the construction. \ref{item:phiglembed} on $\R^2\setminus \sqcup_{\pm}D_{L_{j,\pm}}(3\delta'_{j,\pm})$ follows from \ref{lem:varphiest}\ref{item:solembed} and \ref{def:initial}, and on each $D_p(4\delta'_{j,\pm})\setminus D_p(\delta'_{j,\pm})$ with $p\in L_{j,\pm}$ follows from \ref{lem:glue}\ref{item:gluecat} and \ref{lem:varphicat}. The embeddedness of $M$ then follows by \ref{item:phiglembed} and by comparing the rest of $M$ with catenoids. \ref{item:phiglbound} on $\R^2\setminus\sqcup_{\pm} D_{L_{j,\pm}}(3\delta'_{j,\pm})$ follows from \ref{lem:varphiest}\ref{item:solboundn} and \ref{def:initial}, and on each $D_p(4\delta'_{j,\pm})\setminus D_p(\delta'_{j,\pm})$ with $p\in L_{j,\pm}$ follows from  \ref{lem:glue}\ref{item:gluevarphi}.  \ref{item:phiglCS} follows by \ref{lem:varphiest}\ref{item:solCS} and \ref{def:initial}, \ref{item:phiglcone} follows by \eqref{eq:varphizeta}, \eqref{eq:Phicone} and \eqref{eq:taupm}.
	  \end{proof}
	  \begin{definition}\label{def:region}
	  	We define the regions on $M$ by (recall \ref{def:catebridge})
	  	\begin{align}\label{eq:KcheckM}
	  		&\K_M:=\sqcup_{\ell}\sqcup_{p\in L_{\ell}}\K[p,\tau_{\ell},h_{\ell},\kappau_{\ell}],\\
	  		&\Kcech_x[M]:=\sqcup_{\ell}\sqcup_{p\in L_{\ell}}\Kcech_x[p,\tau_{\ell},h_{\ell},\kappau_{\ell}],\nonumber\\\nonumber
	  	&\Ku_x[M]:=\sqcup_{p\in L}\Ku_x[p,\tau_{\ell},h_{\ell},\kappau_{\ell}].\nonumber
	  	\end{align}
	  	We define the map $\Pi_{\K}:\Kcech[M]\to\K_M$ by taking $\Pi_{\K}:=(\exp_p^{\R^2,\R^3,g_{\shr}})^{-1}$ on each $\Kcech[p,\tau_{\ell},h_{\ell},\kappau_{\ell}]$. We also define $\underline{\K}[p,\ell]:=\Pi_{\K}(\Ku[p,\tau_{\ell},h_{\ell},\kappau_{\ell}])$ for each $p\in L_{\ell}$. Finally, we define the region $\tilde{S}'_{x,j}\subset \R^2$ and $\tilde{S}_{x,j}\subset M$ by
	  	\begin{align}\label{eq:tildeSprime}
	  		&\tilde{S}'_{x,j}:=\R^2\setminus \sqcup_{\pm} D_{L_{j,\pm}}((1+x)b\tau_{j,\pm}),\\
	  		&\tilde{S}_{x,j}:=\Pi_{\R^2}^{-1}(\tilde{S'}_{x,j})\cap \Graph_{\Omega_j}^{\R^3,\R^2,g}(\varphi^{gl}_j).\nonumber
	  	\end{align}
	  	We also omit $x$ if $x=0$.
	  \end{definition}
	 
	  \begin{notation}\label{not:JM}
	  	If for $j=-J,-J+1,\dots,J$, $f_j$ is a function supported on $\tilde{S}'_j$, we use the notation $\fbold=\{f_j\}$. 	We then define $\JM(\fbold)$ to be the function on $M$ supported on $\sqcup_j\tilde{S}_{j}$ defined by $f_j\circ\Pi_{\R^2}$ on $\tilde{S}_{j}$ and $0$ on the other region. 
	  	
	  	If there are two collections of functions $\fbold=\{f_j\}$, $\gbold=\{g_j\}$, we also use $\fbold\gbold$ to denote the collection of functions $\{f_jg_j\}$. Similarly, we use the notation $\Lcal\fbold$ to denote the collection of functions $\{\Lcal f_j\}$.
	  \end{notation}

	   \begin{definition}\label{def:VM}
	  	We define the vector space $\Vcal[M]\subset\oplus_{j}\Vcal_{\sym}[L_j]$ by (recall \ref{item:LDLpm}\ref{item:varphisym}) the following. An element $\{\nu_{J,+}+\nu'_{J,+}\dr,\{(\nu_{j,+}+\nu'_{j,+}\dr,\nu_{j,-}+\nu'_{j,-}\dr)\}_{j={-J+1}}^{J-1},\nu_{J,-}+\nu'_{J,-}\dr\}\in \Vcal[M]$ if and only if
	  	\begin{align*}
	  		\begin{cases}
	  			&\nu_{j,+}=-\nu_{-j,-},\nu'_{j,+}=-\nu'_{-j,-},\text{ when }j\in\N,\\
	  			&\nu_{j,+}=\nu_{-j,-},\nu'_{j,+}=\nu'_{-j,-},\text{ when }j\notin\N.\\
	  		\end{cases}
	  	\end{align*}
	  	By \ref{item:LDLpm}\ref{item:varphisym}, we have $\Vcal[M]\cong \R^{4J}$. We will then use $(\nubold,\nubold')=(\{\nu_{j,\pm}\},\{\nu'_{j,\pm}\})\in\R^{4J}$ to denote the element in $\Vcal[M]$. We also use $(\taubold\nubold,\taubold\nubold')$ to denote $(\{\tau_{j,\pm}\nu_{j,\pm}\},\{\tau_{j,\pm}\nu'_{j,\pm}\})\in \Vcal[M]$.
	  	
	  	We also define the vector spaces $\skernelv[M]\subset\oplus_{j}\skernelv[L_j]$, $\skernel[M]\subset\oplus_{j}\skernel[L_j]$ (recall \ref{def:VW}) by the following.  An element $\{f_{J,+},\{(f_{j,+},f_{j,-})\}_{j={-J+1}}^{J-1},f_{J,-}\}\in \skernelv[M]$ if and only if
	  	\begin{align*}
	  		\begin{cases}
	  			&f_{j,+}=-f_{-j,-},\text{ when }j\in\N,\\
	  			&f_{j,+}=f_{-j,-},\text{ when }j\notin\N.\\
	  		\end{cases}
	  	\end{align*}
	  	Similarly, we define $\skernel[M]$, or equivalently, $\skernel[M]=\Lcal \skernelv[M]$. By \ref{N:G}, we have $\JM\skernelv[M]\subset C^{\infty}_{\sym}(M)$ and $\JM\skernel[M]\subset C^{\infty}_{\sym}(M)$.
	  
	  We then define the isomorphism $\Ecal_M: \skernelv[M]\to \Vcal[M]$ by $\Ecal_M\fbold=\{\Ecal_{\Lu_j}f_j\}$ (recall \ref{def:VW}). We remark that
	  \begin{equation*}
	  	\Lcal\circ\Ecal_M^{-1}(\nubold,\nubold')=\{\sum_{\pm}\tau_{j,\pm}(\nu_{j,\pm}W_{j,\pm}+\nu'_{j,\pm}W'_{j,\pm})\},
	  \end{equation*}
	  where $W_{j,\pm}=W[L_{j,\pm}]$, $W'_{j,\pm}=W'[L_{j,\pm}]$.
	  \end{definition}
	  
	  \begin{definition}
	  	We define the functions $w_j\in\skernel[L_j]$ by (recall \ref{def:VW})
	  	\begin{equation}\label{eq:w}
	  		w_j:=
	  		\begin{cases}
	  			&\Lcal\left(\underline{v}_j+\Ecal_{\Lu_j}^{-1}\kappau_j\right),\text{ when }j\in\N,\\
	  			&\frac{j}{\abs{j}}\Lcal\left(\underline{v}_j+\Ecal_{\Lu_j}^{-1}\kappau_j\right),\text{ when }j\notin\N.\\
	  		\end{cases}
	  	\end{equation}
	  	Notice by \ref{lem:obstr}\ref{item:wsupp}, we can use the notation $\wbold=\{w_j\}$. Moreover, by \ref{N:G} and \eqref{eq:vunder}, we have $\wbold\in\skernel[M]$ and $\JM(\wbold)\in C^{\infty}_{\sym}(M)$.
	  \end{definition}
	  
	  \begin{lemma}\label{lem:zetamu}
	  	There exists a linear map (recall \ref{def:Pcal}) $Z=Z_{\zetaboldu,\kappaboldu}:\Vcal[M]\to \Pcal$ such that for $(\zetaboldu,\kappaboldu)\in\BPcal$
	  	\begin{equation*}
	  		\abs{(\zetaboldu,\kappaboldu)-Z(\mubold,\mubold')}\lesssim 1,
	  	\end{equation*}
	  	where $(\mubold,\mubold')\in\Vcal[M]$ are defined such that  $	\Lcal\circ\Ecal_M^{-1}(\taubold\mubold,\taubold\mubold')=-\wbold$ (recall \eqref{eq:w}). 
	  \end{lemma}
	  \begin{proof}
	  	First notice that by \ref{lem:miss}, \eqref{eq:w} and \eqref{not:kappa}
	  	\begin{equation}\label{eq:mubold}
	  		(\mubold,\mubold')=(\{\mu_{j,\pm}-\kappa_{j\mp 1/2}^{\perp}\},\{\mu'_{j,\pm}-\kappa_{j\mp 1/2}\}),
	  	\end{equation}
	  	where $\mu_{j,\pm}$, $\mu'_{j,\pm}$ are as the same in \ref{lem:miss}.
	  	
	  	Now consider the map $A:\Pcal\to\Vcal[M]$, $A:(\zetaboldu,\kappaboldu)\mapsto (\nubold,\nubold')$ defined by (again we define $\zeta_0=0$, $\zeta_{-1/2}=-\zeta_{1/2}$, $\zeta_{J}=0$)
	  	\begin{equation*}
	  		\nu_{j,\pm}=\pm\zeta+
	  		\begin{cases}
	  			&\mp\frac{\cos\frac{\pi (j+1\mp1/2)}{2J+1}\zeta_{j+1/2\mp 1/2}-\cos\frac{\pi(j-1\mp 1/2)}{2J+1}\zeta_{j-1/2\mp1/2}}{2\cos\frac{\pi(j\mp 1/2)}{2J+1}}-\kappa_{j\mp 1/2}^{\perp},
	  			1/2\leq j\leq J, ,\\
	  			&\mp  \frac{\cos\frac{3\pi/2}{2J+1}}{2\cos\frac{\pi/2}{2J+1}}\zeta_{1}\mp\kappa_{1/2}^{\perp}, j=0,\text{ when } J\in\N,
	  		\end{cases}
	  	\end{equation*}
	  	and
	  	\begin{equation*}
	  		\nu'_{j,\pm}=
	  		\begin{cases}
	  			&\zeta'-\kappa_{J-1/2}, j=J\text{ for }+,\\
	  			&\pm\zeta'_{j}-\kappa_{j\mp1/2}, 1/2\leq j\leq J-1,\\
	  			&\mp \kappa_{1/2}, j=0,\text{ when } J\in\N.
	  		\end{cases}
	  	\end{equation*}
	  	From \eqref{eq:muzeta}, \eqref{eq:muzetaprime} and \eqref{eq:mubold}, $\abs{A(\zetaboldu,\kappaboldu)-(\mubold,\mubold')}\lesssim 1$. The problem is then the same as showing the invertibility of the linear map $A$. 
	  	
	  	For the component $(\zeta',\zetabold',\kappabold)\mapsto \nubold'$, the invertibility is obvious. For the component $(\zeta,\zetabold,\kappabold^{\perp})\mapsto \nubold$, notice that for $1 \leq j \leq J$,  $\nu_{j,+}+\nu_{j-1,-}=-2\kappa_{j-1/2}^{\perp}$ and for $1/2 \leq j \leq J$
	  	\begin{equation*}
	  		\nu_{j,+}-\nu_{j-1,-}=2\zeta-\frac{\cos\frac{\pi (j+1/2)}{2J+1}\zeta_{j}-\cos\frac{\pi(j-3/2)}{2J+1}\zeta_{j-1}}{\cos\frac{\pi(j- 1/2)}{2J+1}}.
	  	\end{equation*}
	  	The right hand side can be written as a matrix multiplying with the vector $(\zeta,\zetabold)$ with non-zero elements in the diagonal, one sub-diagonal, and one column. The invertibility simply follows by Gaussian elimination.
	  \end{proof}
	  \section{The linearised equation on the initial surfaces}
\label{S:lin} 

	  \subsection*{Global norms and the mean curvature on the initial surfaces}
	  
	  \begin{definition}\label{def:norm}
	  	For $k \in \{0,2\}$, $\hat{\beta} \in (0, 1)$, $\hat{\gamma} \in \mathbb{R}$, we define the following norms. For $L$ a finite set in $\R^2$ and $\Omega$ a domain in $\R^3\cap \{\zz=h\}$ (recall \ref{def:conespaces}), we define
	  	\begin{equation*}
	  		\norm{u}_{k,\hat{\beta},\hat{\gamma};\Omega,L}:=\norm{u:C^{k,\hat{\beta}}(\Omega\cap\Pi_{\R^2}^{-1}(D_O(2\rbalanced)),\rr,g,\rr^{\hat{\gamma}})}+\norm{u\circ\Pi_{\R^2}^{-1}:\mathcal{CS}^{2,\hat{\beta}}(\Pi_{\R^2}(\Omega)\setminus D_O(2\rbalanced))},
	  	\end{equation*}
	  	where $\rr:=\dist_L\circ \Pi_{\R^2}$ and $g$ is the standard metric on $\R^3$; for $\Omega$ a domain in $M$ we define
	  		\begin{align*}
	  		\norm{u}_{k,\hat{\beta},\hat{\gamma};\Omega}&:=\norm{u:C^{k,\hat{\beta}}(\Omega\cap \Pi_{\R^2}^{-1}(D_O(2\rbalanced)),\rr,g,\rr^{\hat{\gamma}})}\\
	  		&+\sum_j\norm{u|_{\tilde{S}_j}\circ(\Pi_{\R^2}|_{\tilde{S}_j})^{-1}:\mathcal{CS}^{k,\hat{\beta}}(\Pi_{\R^2}(\tilde{S}_j\cap\Omega)\setminus D_O(2\rbalanced))},
	  	\end{align*}
	  	where $\rr:=\dist_{L_j}\circ\Pi_{\R^2}$ on $\Graph_{\Omega_j}^{\R^3,\R^2,g}(\varphi^{gl}_j)$ and $\rr:=\dist_{L_{\ell}}\circ\Pi_{\R^2}$ on $\Kcech[p,\tau_{\ell},h_{\ell},\kappau_{\ell}]$ (recall \ref{def:initial}), and $g$ is the metric on $M$ induced by the standard metric on $\R^3$; for $\Omega$ a domain in $\K_M$  (recall \ref{def:region}) we define
	  	\begin{equation*}
	  		\norm{u}_{k,\hat{\beta},\hat{\gamma};\Omega}:=\norm{u:C^{k,\hat{\beta}}(\Omega,\rr,g,\rr^{\hat{\gamma}})},
	  	\end{equation*}
	  	where $\rr := \rho(s)$ (recall \eqref{eq:rhozz}) and $g$ is the metric induced by each Euclidean metric $g_{\R^3}|_p$ on $T_p\R^3$ $\forall p \in L_{\ell}$ on $ \K[p,\tau_{\ell},h_{\ell},\kappau_{\ell}]$.
	  \end{definition}
	  
	    \begin{convention}\label{conv:gamma}
	  	From now on we assume that $b$ (recall \ref{def:catebridge}) is as large as needed in absolute terms and $\cu$ (recall \ref{def:Pcal}), and $m$ is as big and thus $\tau_{\ell}$ is as small as needed in absolute terms and $b$. We also fix some $\beta\in(0,1/100)$ and $\gamma\in (1,2)$. We will suppress the dependence of various constants on $\beta$. 
	  \end{convention} 
	  
	  \begin{notation}\label{not:Bq}
	  	Throughout this subsection we use the notation $\R^2_{j,\pm}$ to denote the affine plane with height $h_{j,\pm}$ (recall \eqref{eq:hpm}):
	  	\begin{equation*}
	  		\R^2_{j,\pm}:=\R^3\cap\{\zz=h_{j,\pm}\}.
	  	\end{equation*}
	  	Let $q\in \tilde{S'}_j\cap D_p(\delta)$ for $p\in L_{j,\pm}$, we define the metric $\tilde{g}_q :=\rr(q)^{-2} g$ on $\R^3$, where $\rr(q):=\dist_{L_{j,\pm}}(q)$, $g=g_{Euc}$. In this metric $\tilde{S}_{j}$ is locally the graph of $\varphi_{:q}$ over $\R^2_{j,\pm}$, where $\varphi_{:q}:=\rr(q)^{-1} (\varphi^{gl}_j-h_{j,\pm})$. We also define the disk $\widecheck{B}_q:=D^{\R^2,\tilde{g}_q}_q(1/10)$. Finally, we remark that the norm defined in \ref{def:norm} for $\widecheck{B}_q$ is the same on $\R^2_{j,\pm}\cap \Pi^{-1}_{\R^2}(\widecheck{B}_q)$: for $u\in C^{k,\hat{\beta}}(\R^2_{j,\pm})$
	  	\begin{equation}\label{eq:normheq}
	  		\norm{u\circ \Pi_{\R^2}^{-1}}_{k,\hat{\beta},\hat{\gamma};\widecheck{B}_q,L_{j,\pm}}=\norm{u}_{k,\hat{\beta},\hat{\gamma};\R^2_{j,\pm}\cap\Pi^{-1}_{\R^2}(\widecheck{B}_q),L_{j,\pm}}.
	  	\end{equation}\label{eq:Lheq}
	  	And for the operator $\Lcal_{\R^2_{j,\pm}}$ (recall \eqref{eq:jacobiF}): 
	  	\begin{equation}
	  		\Lcal_{\R^2_{j,\pm}}u=\Lcal_{\R^2}(u\circ\Pi_{\R^2}).
	  	\end{equation}
	  \end{notation}

	  \begin{lemma}\label{lem:equivnorm}
	  	For $k=0,2$ and $\hat{\gamma}\in\R$ the following hold.
	  	\begin{enumerate}[(i)]
	  		\item\label{item:equivnormK} If $\widecheck{\Omega}$ is a domain in $\Pi_{\K}(\Kcech[M])$ (recall \eqref{eq:KcheckM}), $\Omega:=\Pi_{\K}^{-1}(\widecheck{\Omega})\subset \Kcech[M]\subset M$ and $f\in C^{k,\beta}(\widecheck{\Omega})$, then
	  		\begin{equation*}
	  			\norm{f\circ\Pi_{\K}}_{k,\beta,\hat{\gamma};\Omega}\sim_k 	\norm{f}_{k,\beta,\hat{\gamma};\widecheck{\Omega}}.
	  		\end{equation*}
	  		\item\label{item:equivnormS} If $\Omega'$ is a domain in $\tilde{S'}_j$ (recall \ref{eq:tildeSprime}), $\Omega:=(\Pi_{\R^2}|_{\tilde{S}_j})^{-1}(\Omega')$ and $f\in C^{k,\beta}(\Omega')$, then
	  		\begin{equation*}
	  			\norm{f\circ\Pi_{\R^2}}_{k,\beta,\hat{\gamma};\Omega}\sim_k 	\norm{f}_{k,\beta,\hat{\gamma};\Omega',\sqcup_{\pm}L_{j,\pm}}.
	  		\end{equation*}
	  	\end{enumerate}
	  \end{lemma}
	\begin{proof}
		 \ref{item:equivnormK} is the same as \cite[Lemma 4.3(i)]{LDg} if we use $g_{\shr}$ instead of $g_{Euc}$ in the definition \ref{def:norm}. However, the two metric is equivalent in the region.
		
		\ref{item:equivnormS} is similar to \cite[Lemma 4.3(ii)]{LDg} in $D_O(2\rbalanced)$. However, as the $C^0$-norm of $\phi_j^{gl}$ is of order $m\tau_{j,\pm}$ because of the height $h_{j,\pm}$ in $\widecheck{B}_q$ (see \ref{lem:varphicat} and \ref{lem:glue}\ref{item:gluecat}), we first compare the norms:
			\begin{equation*}
			\norm{f\circ\Pi_{\R^2}}_{k,\beta,\hat{\gamma};\tilde{S}_j\cap\Pi^{-1}_{\R^2}(\widecheck{B}_q)}\sim_k 	\norm{f}_{k,\beta,\hat{\gamma};\R^2_{j,\pm}\cap\Pi^{-1}_{\R^2}(\widecheck{B}_q),\sqcup_{\pm}L_{j,\pm}}.
		\end{equation*}
		Then we make use of \eqref{eq:normheq} to obtain \ref{item:equivnormS}. In the other region it simply follows by the definition \ref{def:norm}.
	\end{proof}

	  \begin{lemma}\label{lem:equivop}
	  	For $\hat{\gamma}\in\R$, the following hold.
	  	\begin{enumerate}[(i)]
	  		\item \label{item:equivopK}If $u\in C^{2,\beta}(\Pi_{\K}(\Kcech[M]))$, then (recall \eqref{eq:jacobiH})
	  		\begin{equation*}
	  			\norm{\Tcal_{\exp(-\omega)}\Lcal_M(\Tcal_{\exp(-\omega)}u\circ\Pi_{\K})-(\Lcal^{Euc}_{\K}u)\circ\Pi_{\K}}_{0,\beta,\hat{\gamma}-2;\Kcech[M]}\lesssim (\delta_{\min}')^2	\norm{u}_{2,\beta,\hat{\gamma};\Pi_{\K}(\Kcech[M])}.
	  		\end{equation*}
	  		\item\label{item:equivopS} If $u\in C^{2,\beta}(\tilde{S}'_j)$ and $\epsilon_1\in [0,1/2]$, then 
	  		\begin{equation*}
	  			\norm{\Lcal_M(u\circ\Pi_{\R^2})-(\Lcal u)\circ\Pi_{\R^2}}_{0,\beta,\hat{\gamma}-2;\tilde{S}_j}\lesssim b^{\epsilon_1-1}\log b \tau^{\epsilon_1}_{\max}	\norm{u}_{2,\beta,\hat{\gamma+\epsilon_1};\tilde{S'}_j}.
	  		\end{equation*}
	  	\end{enumerate}
	  \end{lemma}

	  \begin{proof}
	  	By \cite[Lemma 4.7(i)]{LDg},
	  	\begin{equation*}
	  		\norm{\Lcal^{\shr}_M(u\circ\Pi_{\K})-(\Lcal^{Euc}_{\K}u)\circ\Pi_{\K}}_{0,\beta,\hat{\gamma}-2;\Kcech[M]}\lesssim (\delta'_{\min})^2	\norm{u}_{2,\beta,\hat{\gamma};\Pi_{\K}(\Kcech[M])}.
	  	\end{equation*}
	  		\ref{item:equivopK} then follows by \eqref{eq:jacobiHF} and the uniform bound of $e^{\omega}$.
	  		\ref{item:equivopS} is similar to \cite[Lemma 4.7(ii)]{LDg} in $D_O(2\rbalanced)$. However, again as the $C^0$-norm of $\varphi_j^{gl}$ is of order $m\tau_{j,\pm}$ because of the height $h_{j,\pm}$ in $\widecheck{B}_q$, we first compare the operators for $u\in C^{k,\beta}(\R^2_{j,\pm}\cap\Pi^{-1}_{\R^2}(\widecheck{B}_q))$:
	  		\begin{equation*}
	  		\norm{\Lcal_M(u\circ\Pi_{\R^2_{j,\pm}})-(\Lcal_{\R^2_{j,\pm}} u)\circ\Pi_{\R^2_{j,\pm}}}_{0,\beta,\hat{\gamma}-2;\tilde{S}_j\cap\Pi^{-1}_{\R^2}(\widecheck{B}_q)}\lesssim b^{\epsilon_1-1}\log b \tau^{\epsilon_1}_{j,\pm}	\norm{u}_{2,\beta,\hat{\gamma}+\epsilon_1;\R^2_{j,\pm}\cap\Pi^{-1}_{\R^2}(\widecheck{B}_q)}.
	  		\end{equation*}
	  		Then we make use of \eqref{eq:Lheq} to obtain \ref{item:equivnormS}. Finally, outside $D_O(2\rbalanced)$, \ref{item:equivopS} follows by \ref{lem:phigl}\ref{item:phiglCS}.
    \end{proof}
\begin{lemma}[Mean curvature on the initial surfaces]\label{lem:H}
	We have for $H^{2\omega}=H^{2\omega}_M$ (recall \ref{not:Fcal}) the global estimate (recall \ref{not:JM})
	\\ 
	\begin{equation}\label{eq:H}
		\norm{H^{2\omega}-\JM(\wbold)}_{0,\beta,\gamma-2;M}\lesssim \tau_{\max}^{1+\alpha/3}.
	\end{equation} 
\end{lemma}
\begin{proof}
	We can again use \ref{not:Fcal} as $H^{2\omega}=e^{\omega}H_{M}^{g_{\shr}}$ and
	on $\Kcech[M]$ (recall \ref{def:region}), the estimate is the same as in \cite[Lemma 4.6]{LDg} in $g_{\shr}$. We then again use the equivalence of the two metrics. On the gluing region, the estimate follows by \ref{lem:glue}\ref{item:glueH}. On the exterior of the gluing region and in $\Pi_{\R^2}^{-1}(D_O(2\rbalanced))$ the estimate is again the same as in \cite[Lemma 4.6]{LDg}, where instead of the formula for $H'$, we use the formula on each $\tilde{S}'_j$ (recall \ref{not:Bq})
	\begin{equation*}
		\rr^2 (H^{2\omega})'=\rr^2 w_j+\rr\tilde{Q}_{\varphi_{:q}},
	\end{equation*}
	where $(H^{2\omega})'$ is the pull back of $H^{2\omega}$ on $\R^2$. Finally, outside of $\Pi_{\R^2}^{-1}(D_O(2\rbalanced))$, the estimate follows as \cite[Proposition 10.1]{kapouleas:kleene:moller} or by \cite[Lemma 2.4]{Wang:JAMS} along with \ref{lem:phigl}\ref{item:phiglCS}.
\end{proof}
     Using now the same rescaling we prove a global estimate for the nonlinear terms of the weighted mean curvature of the graph over the initial surfaces as follows.
	  	\begin{lemma}\label{lem:nonlinear}
	  		Let $M$ as in \ref{def:initial} and $\textup{ф}\in C^{2,\beta}(M)$ satisfies $\norm{\textup{ф}}_{2,\beta,\gamma;M}\leq \tau_{\max}^{1+\alpha/4}$, then $M_{\textup{ф}}:=\Graph_M^{\R^3,g}(\textup{ф})$ is well defined, is embedded, and if $H^{2\omega}_{\textup{ф}}$ is the weighted mean curvature of $M_{\textup{ф}}$ pulled back to $M$ by $\Graph_M^{\R^3,g}(\textup{ф})$ and $H^{2\omega}$ is the weighted mean curvature of $M$, then
	  		\begin{equation*}
	  			\norm{H^{2\omega}_{\textup{ф}}-H^{2\omega}-\mathcal{L}_M\textup{ф}}_{0,\beta,\gamma-2;M}\lesssim  \norm{\textup{ф}}_{2,\beta,\gamma;M}^2
	  		\end{equation*}
	  	\end{lemma}

	  \begin{proof}
	  	On $\Kcech[M]$ as in \cite[Lemma 5.1]{LDg}
	  	\begin{equation*}
	  		\norm{H^{\shr}_{\textup{ф}}-H^{\shr}_M-\mathcal{L}_M^{\shr}\Tcal_{\exp(\omega)}\textup{ф}}_{0,\beta,\gamma-2;M}\lesssim  \norm{\Tcal_{\exp(\omega)}\textup{ф}}_{2,\beta,\gamma;M}^2,
	  	\end{equation*}
	  	where $H^{\shr}_{\textup{ф}}=H^{\shr}_{M_{\textup{ф}}}$ (recall \ref{not:Fcal}). By the uniform control of $\omega$ and \eqref{eq:jacobiHF}, we obtain the estimate in this region. On $\Pi_{\R^2}^{-1}(D_O(2\rbalanced))$ the estimate is again the same as in \cite[Lemma 5.1]{LDg}, where instead of the formula for $H_{\textup{ф}}$, we use the formula on each $\tilde{S}'_j$ (recall \ref{def:norm})
	  	\begin{equation*}
	  		\rr^2 H^{2\omega}_{\textup{ф}}=\rr^2H^{2\omega}+\rr^2 \mathcal{L}_M\textup{ф}+\rr\tilde{Q}_{\textup{ф}/\rr}.
	  	\end{equation*}
	  	Finally, outside of $\Pi_{\R^2}^{-1}(D_O(2\rbalanced))$, the estimate follows as \cite[Proposition 10.1]{kapouleas:kleene:moller} or by \cite[Lemma 2.4]{Wang:JAMS}.
	  \end{proof}	
	  
	   \subsection*{The definition of {$\Rcal_M^{appr}$} and the main proposition}
	   \begin{definition}\label{def:cutoff}
	   	We define the collection of functions $\psibold'=\{\psi'_{j}\}\in C^{\infty}(\R^2)$ for $j=-J,-J+1,\dots,J$ and $\widecheck{\psi}\in C^{\infty}(M)$ by requesting the following.
	   	\renewcommand{\theenumi}{\roman{enumi}}
	   	\begin{enumerate}
	   		\item $\widecheck{\psi}$ is supported on $\Kcech[M]\subset M$ and $\psi'_j$ on $\tilde{S}'_j\subset\R^2$ (recall \ref{def:region}).
	   		\item $\psi'_j=1$ on $\R^2\setminus D_{L_j}(2b\tau_{j,\pm})\subset\R^2$ and for each $p\in L_{j,\pm}$ we have
	   		\begin{align*}
	   			\psi'_j&=\Psibold[b\tau_{j,\pm},2b\tau_{j,\pm};\dist_p](0,1)\text{ on }D_{p}^{\R^2}(2b\tau_{j,\pm}),\\
	   			\widecheck{\psi}&=\Psibold[2\delta'_{j,\pm},\delta'_{j,\pm};\dist_p\circ\Pi_{\R^2}](0,1)\text{ on } \Kcech[p,j\mp 1/2].
	   		\end{align*}
	   	\end{enumerate}
	   \end{definition}
	   
	   As in \cite[Definition 4.11]{kapouleas:equator} and \cite[Definition 4.17]{LDg}, we will construct a linear map (recall \ref{def:VW} and \eqref{eq:Lpm}) $\mathcal{R}_{M,appr}:C^{0,\beta}_{\sym}(M)\to C^{2,\beta}_{\sym}(M)\oplus \skernel[M]\oplus C^{0,\beta}_{\sym}(M)$, so that $(u_1,\wbold_{E,1,j},E_1):=\mathcal{R}_{M,appr}(E)$ is an approximate solution to the equation \ref{prop:lineareq}\ref{item:LDeqmodulow}. The approximate solution will be constructed by combining semi-local approximate solutions.
	   
	   Given $E\in C^{0,\beta}_{\sym}(M)$, we define $E'_j\in C^{0,\beta}_{\sym}(\R^2)$ by requiring that they are supported on $\tilde{S}'_j$ and that
	   \begin{equation}\label{eq:Eprime}
	   	E'_j\circ\Pi_{\R^2}=(\psi'_j\circ\Pi_{\R^2})E.
	   \end{equation}
	   By \ref{lem:linearRduo} and \ref{item:LDLpm}\ref{item:EL}, there are unique functions $u'_j\in \mathcal{CS}^{2,\beta}_{\sym}(\R^2)$ and $w_{E,1,j}\in \skernel[L_j]$ such that
	   \begin{equation}\label{eq:uprime}
	   	\Lcal u'_j=E'_j+w_{E,1,j},\quad \Ecalu_p u'_j=0, \forall p\in L_j.    
	   \end{equation}
	   Note that $\Lcal((1-\psi'_j)u'_j)=[\psi',\Lcal]u'_j+(1-\psi'_j)E'_j$ is supported on $\Ku[M]\subset\Kcech[M]\subset M$, we define $\tilde{E}\in C^{0,\beta}_{\sym}(\K_M)$, by requesting that it is supported on $\Pi_{\K}(\Ku[M])$ and that on $\Ku[M]$ we have
	   \begin{equation}\label{eq:Etilde}
	   	\tilde{E}\circ\Pi_{\K}=(1-\JM(\psibold'))E+\JM(\Lcal((1-\psibold')\ubold')),
	   \end{equation}
	   where $\ubold':=\{u'_j\}$ (recall \ref{not:JM}).
	   
	   For $k\in\{0,2\}$, we introduce decompositions $C^{k,\beta}(\K_M)=C^{k,\beta}_{low}(\K_M)\oplus C^{k,\beta}_{high}(\K_M)$ and also $H^{1}(\K_M)=H^{1}_{low}(\K_M)\oplus H^{1}_{high}(\K_M)$ into subspaces of functions which satisfy the condition that their restrictions to a parallel circle of a $\K[p,\tau_{\ell},h_{\ell},\kappau_{\ell}]$ belong or are ($L^2$-)orthogonal respectively to the span of the constants and the first harmonics on the circle. Notice that $\tilde{E}$ has compact support, we then have
	   \begin{equation}
	   	\tilde{E}=\tilde{E}_{low}+\tilde{E}_{high},
	   \end{equation}
	   with $\tilde{E}_{low}\in C^{0,\beta}_{low}(\K_M)\cap H^{1}_{low}(\K_M)$, $\tilde{E}_{high} \in C^{0,\beta}_{high}(\K_M)\cap H^{1}_{high}(\K_M)$ supported on $\Pi_{\K}(\Ku[M])\subset \K_M$.
	   
	   Let $\Lcal_{\K}^{Euc}$ denote the Jacobi linear operator on $\K[p,\tau_{\ell},h_{\ell},\kappau_{\ell}]$ (recall \eqref{eq:jacobiH}), we define $\tilde{u}=\tilde{u}_{low}+\tilde{u}_{high}$ by requesting $\tilde{u}_{low}\in C^{2,\beta}_{low}(\K_M)$, $\tilde{u}_{high} \in C^{2,\beta}_{high}(\K_M)$ to be the solutions of
	   \begin{equation}\label{eq:utilde}
	   	\mathcal{L}_{\K}\tilde{u}_{low}= \tilde{E}_{low},\quad \mathcal{L}_{\K}\tilde{u}_{high}= \tilde{E}_{high}
	   \end{equation}
	   determined uniquely as follows. By separating variables, the first equation amounts to uncoupled ODE equations which are solved uniquely by assuming vanishing initial data on the waist of the catenoids. For the second equation we can as usual change the metric conformally to $h= \frac{1}{2}\abs{A}^2_{\K}g_{\K}=\nu_{\K}^*g_{\Sph^2}$, and then we can solve uniquely because the inhomogeneous term is clearly orthogonal to the kernel. 
	   
	   We conclude now the definition of $\mathcal{R}_{M,appr}$:
	   \begin{definition}\label{def:Rappr}
	   	We define the operator 
	   	\begin{equation*}
	   		\mathcal{R}_{M,appr}:C_{\sym}^{0,\beta}(M)\to C_{\sym}^{2,\beta}(M)\oplus \skernel[M]\oplus C_{\sym}^{0,\beta}(M)
	   	\end{equation*}
	   	by $ \mathcal{R}_{M,appr}E=(u_1,\wbold_{E,1},E_1)$, where $\wbold_{E,1}=\{w_{E,1,j}\}$ is as above, $u_1:=\widecheck{\psi}\tilde{u}\circ\Pi_{\K}+\JM(\psibold'\ubold')$, $E_1=\Lcal_M u_1-E-\JM (\wbold_{E,1})$.
	   \end{definition}

	   \begin{proposition}\label{prop:lineareq}
	   	A linear map $\mathcal{R}_M:C^{0,\beta}_{\sym}(M)\to C^{2,\beta}_{\sym}(M) \oplus\skernel[M]$, $E\mapsto (u,\wbold_{E})$ can be defined by
	   	\begin{equation*}
	   		\mathcal{R}_M:=(u,\wbold_{E}):=\sum_{n=1}^{\infty}(u_n,\wbold_{E,n}),
	   	\end{equation*}
	   	where the sequence $\{(u_n,\wbold_{E,n},E_n)\}_{n\in \N}$ is defined inductively by
	   	\begin{equation*}
	   		(u_n,\wbold_{E,n},E_n):=-\mathcal{R}_{M,appr}E_{n-1}, \quad E_0:=-E.
	   	\end{equation*}
	   	Moreover, the following hold:
	   	\renewcommand{\theenumi}{\roman{enumi}}
	   	\begin{enumerate}
	   		\item\label{item:LDeqmodulow} $\Lcal_M u=E+\JM(\wbold_{E})$.
	   		\item We have the estimates
	   		\begin{align*}
	   			&\norm{u}_{2,\beta,\gamma;M}\lesssim m^{4+2\beta} \norm{E}_{0,\beta,\gamma-2;M},\\
	   			& \abs {\mubold_{E}}+\abs {\mubold'_{E}}\lesssim \tau_{\max}^{-1}m^{4+2\beta-\gamma}\norm{E}_{0,\beta,\gamma-2;M},
	   		\end{align*}
	   	 where $(\mubold_{E},\mubold'_{E})\in\Vcal[M]$ are defined such that (recall \ref{def:VM}) $	\Lcal\circ\Ecal_M^{-1}(\taubold\mubold_{E},\taubold\mubold'_{E})=-\wbold_E$.
	   	\end{enumerate}
	   \end{proposition}
	   \begin{proof}
	   The proof is the same as \cite[Proposition 4.18]{LDg} from \ref{lem:equivnorm} and \ref{lem:equivop}. The only difference is that instead of $\pm$ corresponding to two levels, we have the subscript $j$ for multiple ($2J+1$) levels.
	\end{proof}

	  \section{Main results}
\label{S:main} 

	  \begin{theorem}\label{thm}
If $m\in\N$ is large enough, then there is $(\breve{\zetaboldu},\breve{\kappaboldu})\in \BPcal$ as in \ref{def:Pcal}, $\breve{\tau}_{\ell}:=\tau_{\ell}\bbracket{\breve{\zetaboldu};m}$, $\breve{r}_{\ell}:=r_{\ell}\bbracket{\breve{\zetaboldu};m}$, $\breve{h}_{\ell}:=h_{\ell}\bbracket{\breve{\zetaboldu};m}$, $\breve{\varphi}_j:=\varphi_j\bbracket{\breve{\zetaboldu};m}$ as in \ref{def:varphi} satisfying \ref{lem:varphiest}, and moreover there is $\breve{\textup{ф}}\in C^{2,\beta}_{\sym}(M[\breve{\zetaboldu},\breve{\kappaboldu}])$, where $M[\breve{\zetaboldu},\breve{\kappaboldu}]$ is as in \ref{def:initial} such that in the notation of \ref{def:norm},
\begin{equation}\label{eq:phibrev}
	\norm{\breve{\textup{ф}}}_{2,\beta,\gamma;M}\leq \tau_{\max}^{1+\alpha/4},
\end{equation}
and furthermore in the notation of \ref{lem:nonlinear}, $\breve{M}=\breve{M}[m,J]:=M[\breve{\zetaboldu},\breve{\kappaboldu}]_{\breve{\textup{ф}}}$ is a $\Grp[m,J]$-invariant embedded self-shrinker in $\R^3$ (recall \ref{def:grp}) with genus $2J(m-1)$ and $2J+1$ ends. The hypersurfaces $\breve{M}[m,J]$ converge in the sense of varifolds as $m \to\infty$ to $(2J+1)\R^2$, and $\breve{M}[m,J]$ has asymptotic cone $\cup_j\breve{\Ccal}_j$, where $\breve{\Ccal}_j=\breve{\Ccal}_j[m,J]$ is the cone which is the graph of (recall \ref{lem:phimu})
\begin{equation*}
	\begin{cases}
	\pm\frac{\sqrt{\pi}m}{2}\breve{\tau}_{J-1/2}\phi_m(\rbalanced)e^{\rbalanced^2/8}(1+O(m^{-1}))r,\text{ when }j=\pm J,\\
	\frac{\sqrt{\pi}m}{2}(\breve{\tau}_{j-1/2}-\breve{\tau}_{j+1/2})\phi_m(\rbalanced)e^{\rbalanced^2/8}(1+O(m^{-1}))r,\text{ when }j\neq\pm J,0,\\
		O(\breve{\tau}_{1/2}^{1+\alpha/4})r\text{ when }j=0 \text{ and }J\in\N.
	\end{cases}
\end{equation*} 

	  \end{theorem}
	  \begin{proof}
	  		As in \cite[Lemma 5.5]{LDg}, there exists a family of diffeomorphisms $\mathcal{F}_{\zetaboldu,\kappaboldu}:M[0]\to M[\zetaboldu,\kappaboldu]$ 
continuously depending on $\zetaboldu,\kappaboldu$ such that for any $u\in C^{k,\beta}(M[\zetaboldu,\kappaboldu])$, we have
	  		\begin{equation}\label{eq:Fzeta}
	  			\norm{u\circ \mathcal{F}_{\zetaboldu,\kappaboldu}}_{k,\beta,\gamma;M[0]}\sim_k 	\norm{u}_{k,\beta,\gamma;M[\zetaboldu,\kappaboldu]}.
	  		\end{equation}
	  		
	  		 We first define $B\subset C^{2,\beta}(M[0])\times \Pcal$ by (recall \ref{def:Pcal})
	  	\begin{equation}\label{eq:B}
	  		B:=\{v\in C^{2,\beta}(M[0]):\norm{v}_{2,\beta,\gamma;M[0]}\leq \tau_{\max}\bbracket{0;m}^{1+\alpha}\}\times\BPcal.
	  	\end{equation}
	  	We next define a map $\Jcal:B\to C^{2,\beta}(M[0])\times \Pcal$ as the following.
	  	Let $(v,\zetaboldu,\kappaboldu)\in B$, we define $(u,\wbold_H):=-\mathcal{R}_{M[\zetaboldu,\kappaboldu]}(H^{2\omega}_M-\JM(\wbold))$ by \ref{prop:lineareq} (recall \eqref{eq:w}). We define $\textup{ф}:=v\circ \mathcal{F}_{\zetaboldu,\kappaboldu}^{-1}+u$. Then by \ref{lem:H}, \ref{prop:lineareq}, \eqref{eq:Fzeta} and the size of $v$ in \eqref{eq:B}
	  	\begin{equation}\label{eq2}
	  		\norm{\textup{ф}}_{2,\beta,\gamma;M[\zetaboldu,\kappaboldu]}\leq  \tau_{\max}^{1+\alpha/4},\quad \abs{\mubold_{H}}+\abs{\mubold'_{H}}\leq  \tau_{\max}^{\alpha/4},
	  	\end{equation}
	   where $(\mubold_{H},\mubold'_{H})\in\Vcal[M]$ are defined such that $	\Lcal\circ\Ecal_M^{-1}(\taubold\mubold_{H},\taubold\mubold'_{H})=-\wbold_H$.
	  	By \ref{prop:lineareq} again we define $(u_Q,\wbold_Q):=-\mathcal{R}_{M[\zetaboldu,\kappaboldu]}(H^{2\omega}_{\textup{ф}}-H^{2\omega}_M-\mathcal{L}_M\textup{ф})$, where $H^{2\omega}_{\textup{ф}}:=H^{2\omega}_{M_\textup{ф}}$ is the weighted mean curvature on $M_\textup{ф}$ in the notation of \ref{lem:nonlinear}. Then by \ref{lem:nonlinear} and \ref{prop:lineareq}
	  	\begin{equation}\label{eq4g}
	 \norm{u_Q}_{2,\beta,\gamma;M[\zetaboldu,\kappaboldu]}\leq  \tau_{\max}^{2-\alpha/4},\quad \abs{\mubold_{Q}}+\abs{\mubold'_{Q}}\leq  \tau_{\max}^{1-\alpha/4},
	  	\end{equation}
	  	 where $(\mubold_{Q},\mubold'_{Q})\in\Vcal[M]$ are defined such that $	\Lcal\circ\Ecal_M^{-1}(\taubold\mubold_{Q},\taubold\mubold'_{Q})=-\wbold_Q$.
	  	Combining the definitions we have
	  	\begin{equation}\label{eq5g}
	  		\Lcal_M(u_Q-v\circ \Fcal_{\zetaboldu,\kappaboldu}^{-1})+H_{\textup{ф}}=-\JM(\wbold_{sum}),
	  	\end{equation}
	  	where $\wbold_{sum}=\wbold+\wbold_H+\wbold_Q=-\Lcal\circ\Ecal_M^{-1}(\taubold\mubold_{sum},\taubold\mubold'_{sum})$ and $(\mubold_{sum},\mubold'_{sum})=(\mubold+\mubold_{H}+\mubold_{Q},\mubold'+\mubold'_{H}+\mubold'_{Q})$. 
	  	
	  	Finally, we define $\Jcal$ by (recall \ref{lem:zetamu})
	  	\begin{align*}
	  		\mathcal{J}(v,\zetaboldu,\kappaboldu):=(u_Q\circ \Fcal_{\zetaboldu,\kappaboldu},(\zetaboldu,\kappaboldu)-Z_{\zetaboldu,\kappaboldu}(\mubold_{sum},\mubold'_{sum})).
	  	\end{align*}
	  	
	  	$B$ is convex, and the embedding $B\hookrightarrow C^{2,\beta'}(M[0])\times \Pcal$ is compact for $\beta'\in(0,\beta)$. By \eqref{eq4g} and \eqref{eq:Fzeta}, $\mathcal{J}$ maps the first factor into $B$ itself. By \ref{def:Pcal}, \ref{lem:zetamu}, \eqref{eq2} and \eqref{eq4g} $\mathcal{J}$ maps the second factor into $B$ itself by choosing $\cu$ and then $m$ big enough. It is easy to check that $\mathcal{J}$ is a continuous map in the induced topology. By Schauder’s fixed point theorem \cite[Theorem 11.1]{gilbarg} then, there is a fixed point $(\breve{v},\breve{\zetaboldu},\breve{\kappaboldu})\in B$ of $\mathcal{J}$, which therefore satisfies $\breve{v}=\breve{u}_Q\circ \mathcal{F}_{\breve{\zetaboldu},\breve{\kappaboldu}}$ and $\breve{\wbold} + \breve{\wbold}_H + \breve{\wbold}_Q = 0$, where we use “$\breve{\cdot}$” to denote the various quantities for $\zetaboldu= \breve{\zetaboldu}$, $\kappaboldu = \breve{\kappaboldu}$ and $v = \breve{v}$. By \eqref{eq5g} then we conclude that $\breve{M}_{\breve{\textup{ф}}}$ is a self-shrinker. The smoothness follows from standard regularity theory, and the embeddedness follows from \ref{lem:phigl}\ref{item:phiglembed} and \eqref{eq2}. The topology then follows because we are connecting $2J+1$ planes adjacently with $m$ bridges, and the symmetry follows by construction. Finally the cone asymptotic follows from the bound on the norm of $\breve{\textup{ф}}$ in \eqref{eq2} and \ref{lem:phigl}\ref{item:phiglcone}, \eqref{eq:r}.
	  \end{proof}

	    \appendix
	    \section{A linear algebra lemma}
	     A similar lemma has also appeared in \cite[Lemma 3.18]{wiygul:stacking} and in \cite[Lemma A.1]{CSW}. Here we make a version more directly related to our application.
	    \begin{lemma}\label{lem:linearalg}
	    	For $N\in\N$, the \emph{tridiagonal Toeplitz} matrix $T_{N\times N}$ with
	    	\begin{equation*}
	    		T_{ij}:=
	    		\begin{cases}
	    			1\text{ if }\abs{i-j}=1,\\
	    			0\text{ otherwise },
	    		\end{cases}
	    	\end{equation*}
	    has eigenvalues
	    \begin{equation*}
	    	\lambda_k=2\cos{\frac{k\pi}{N+1}},
	    \end{equation*}
	    for $k=1,\dots,N$ and corresponding orthogonal eigenvectors
	    \begin{equation*}
	    	\vec{e}_k=\left(\sin{\frac{k\pi}{N+1}},\sin{\frac{2k\pi}{N+1}},\dots,\sin{\frac{Nk\pi}{N+1}}\right).
	    \end{equation*}
	    In particular, all the eigenvalues of $T$ are simple and there is only one positive eigenvalue $2\cos{\frac{\pi}{N+1}}$ corresponding to an eigenvector with all the entries positive.
	    	\end{lemma}
	    	\begin{proof}
	    		\cite[(4),(7)]{Toeplitz}. 
	    	\end{proof}
         \begin{cor}\label{cor:linearalg}
         	For $J\in\frac{1}{2}\N$, the following hold for $j=-J+1,-J+2,\cdots,J$.
         	\begin{align*}  
         		\cos\frac{\pi}{2J+1}&=\frac{\cos\frac{\pi(j-3/2)}{2J+1}+\cos\frac{\pi(j+1/2)}{2J+1}}{2\cos\frac{\pi(j-1/2)}{2J+1}}
         	\end{align*}
         \end{cor}
         \begin{proof}
         	    This follows by applying \ref{lem:linearalg} and rewriting $\sin$ by $\cos$ with $2J=N$. Notice that when $j=J$ it is the usual trigonometric identity
         	    \begin{equation*}
         	    		\cos\frac{\pi}{2J+1}=\frac{\sin\frac{2\pi}{2J+1}}{2\sin\frac{\pi}{2J+1}}=\frac{\cos\frac{\pi(J-3/2)}{2J+1}}{2\cos\frac{\pi(J-1/2)}{2J+1}}.
         	    \end{equation*}
         \end{proof}
         
         \section{Geodesics in conformal metrics}
         In this appendix, we use the notation $(N,\Sigma,g)$ (recall \ref{not:manifold}) and make the extra assumption that $\Sigma\subset N$ is totally geodesic. We also use the notation $\hat{g}=e^{2\omega}g$ for the conformal metric $\hat{g}$. We will consider the geodesics in $N$ under these two metrics. 
         
         \begin{lemma}[{\cite[Lemma 3.13]{LDg}}]\label{lem:logconformal}
         	For each $p\in N$ and $q$ in some neighborhood of $p$ in $N$,
         	\begin{equation*}
         		\abs{\log\dist_p^{\hat{g}}(q)-\log\dist_p^{g}(q)-\omega(p)-\frac{1}{2}\dd_p\omega((\exp_p^g)^{-1}(q))}\lesssim (\dist_p^{g}(q))^2.
         	\end{equation*}
         \end{lemma}
         
         We will now make the following assumption for $\omega$, and apparently the pair $(\R^3,\R^2)$ and the $\omega$ in \eqref{eq:weight} satisfy it.
         \begin{assumption}\label{ass:omegaz}
         	$\omega$ is smooth and moreover, 
         	\begin{equation}
         		\partial_{\zz}\omega|_{\zz=0}=0,
         	\end{equation}
         	where $\zz$ is the Fermi coordinates as in \ref{def:fermi}\ref{item:zz}.
         \end{assumption} 
         
         \begin{lemma}\label{lem:christoffel}
         	The following hold for the Christoffel symbols in a neighbourhood of $p\in\Sigma$, where $\Gamma$ denotes the Levi-Civita connection of the metric $g$, $\hat{\Gamma}$ denotes the Levi-Civita connection of the metric $\hat{g}$, $i,j,\dots$ denote the coordinates on $\Sigma_{\zz}$ (recall \ref{def:fermi}).
         	\begin{align*}
         		\hat{\Gamma}^{\zz}_{ij}=\Gamma^{\zz}_{ij}-\partial_{\zz}{\omega}g_{ij},\quad &\hat{\Gamma}^{\zz}_{i\zz}=\Gamma^{\zz}_{i\zz}+\partial_{i}{\omega},\quad \hat{\Gamma}^{\zz}_{\zz  \zz}=\partial_{\zz}{\omega}.
         	\end{align*}
         \end{lemma}
         \begin{proof}
         	Direct calculation from the formula of the Christoffel symbols under conformal change 
         	and the facts that $g_{\zz\zz}=1$, $g_{\zz i}=0$, $\nabla_{\partial\zz}\partial\zz=0$ (see e.g., \cite[Lemma A.5]{LDg}).
         \end{proof}
         \begin{corollary}\label{cor:geodesic}
         	$\Sigma$ is totally geodesic in $\hat{g}$.
         \end{corollary}
         \begin{proof}
         	This follows by the formula for $\hat{\Gamma}^{\zz}_{ij}$ in \ref{lem:christoffel} and \ref{ass:omegaz}.
         \end{proof}

         \begin{lemma}\label{lem:gamma}
         	Suppose $\Omega\subset\Sigma$ is compact. For $x \in\Omega$, we define the geodesic $\gamma_s=\gamma_s(x)$ in $(N,\hat{g})$ with Fermi coordinates $\gamma_s(x)=(X_s(x),Z_s(x))$ parametrised by arc length in $\hat{g}$ and with the initial conditions:
         	\begin{equation*}
         		\gamma(0)=x,\quad 	\dot{\gamma}(0)=(0,e^{-\omega(x)}).
         	\end{equation*}
         	Then the following hold when $\abs{s}$ is small enough.
         	\begin{align*}
         		\norm{\partial^i_s(X_s-x):C^k(\Omega)}\lesssim_k \min\{1,\abs{s}^{2-i}\}, \quad 
         		\norm{\partial^i_s(e^{\omega}Z_s-s):C^k(\Omega)}\lesssim_k \min\{1,\abs{s}^{3-i}\}.
         	\end{align*}
         \end{lemma}
        \begin{proof}
          The geodesic equation for $Z_s$ by \ref{lem:christoffel} is 
        	\begin{align*}
        		\ddot{Z}+(\Gamma_{ij}^{\zz}-\partial_{\zz}\omega g_{ij})\dot{X^i}\dot{X^j}+2(\Gamma_{i\zz}^{\zz}+\partial_{i}\omega)\dot{X^i}\dot{Z}+\partial_{\zz}\omega\dot{Z}^2=0.
        	\end{align*}
        	From the initial conditions and \ref{ass:omegaz}, the Taylor expansion of $X_s$ and $Z_s$ implies the $C^0$ estimates. The $C^k$ estimates follow by variations with respect to the initial conditions of the equation.
        \end{proof}
        \begin{corollary}\label{cor:domain}
       		Suppose $\Omega\subset\Sigma$, $f\in C^{\infty}(\Omega)$ with $\norm{f:C^{k}(\Omega,g)}$ small enough, the map $\Gamma_f:\Omega\to \Sigma $ defined by $\Gamma_f(x):=X_{f(x)}(x)$ satisfies the following.
       		\begin{equation*}
       			\norm{\Gamma_f-\Id_{\Omega}:C^k(\Omega,g)}\lesssim_k\norm{f:C^k(\Omega,g)}^2.
       		\end{equation*}
       \end{corollary}
       \begin{proof}
       	The $C^0$ estimate follows directly from the $C^0$ estimates in \ref{lem:gamma}. The $C^k$ estimate follows by taking derivatives and the $C^k$ estimates in \ref{lem:gamma}.
       \end{proof}
        \begin{lemma}\label{lem:graph}
        	Suppose $\Omega\subset\Sigma$, $f\in C^{\infty}(\Omega)$ with $\norm{f:C^{k}(\Omega,g)}$ small enough. Consider $\Graph_{\Omega}^{N,\hat{g}}(f)$ the graph of $f$ in the metric $\hat{g}$ in the sense of \ref{not:manifold}\ref{item:graph}. Then there is a domain $\Omega'\subset\Omega$ and a smooth function $f'\in C^{\infty}(\Omega')$ such that the following hold.
        	\begin{enumerate}[(i)]
        		\item\label{item:graphidentical} $\Graph_{\Omega'}^{N,g}(f')\subset \Graph_{\Omega}^{N,\hat{g}}(f)$.
        		\item\label{item:domaincontain} There is a constant $C$ such that $\Omega\subset D^g_{\Omega'}(C\norm{f:C^{0}(\Omega,g)^2})$, (recall \ref{not:manifold}\ref{item:neighbour}).
        		\item\label{item:fprime} $\norm{ \Tcal_{\exp(\omega)}f'-f:C^k(\Omega',g)}\lesssim_k \norm{f:C^{k}(\Omega,g)}^3$.
        	\end{enumerate}
        \end{lemma}
        \begin{proof}
        	\ref{item:graphidentical} is equivalent to
        	\begin{equation*}
        		f'(x)=Z_{f(x)}\circ\Gamma_f^{-1}(x).   
        	\end{equation*}
        	And we use it as the definition for $f'$. \ref{item:domaincontain} then follows by \ref{cor:domain}. \ref{item:fprime} follows by the definition of $f'$ and the estimates in \ref{lem:gamma} and \ref{cor:domain}.
        \end{proof}
	     \bibliographystyle{plain}
	     \bibliography{shrstack}
\end{document}